\documentclass[11pt]{article}
\usepackage{a4}
\usepackage{amsmath,amsfonts,amssymb,amscd,amsthm}
\usepackage{graphicx}

\usepackage{subfigure}
\usepackage{pict2e}
\usepackage{color}
\usepackage[percent]{overpic}
\usepackage{float} 
\usepackage{algorithm}
\usepackage{algorithmic}

\newtheorem{theorem}{Theorem}

\newtheorem{definition}[theorem]{Definition}
\newtheorem{lemma}[theorem]{Lemma}
\newtheorem{corollary}[theorem]{Corollary}
\newtheorem{remark}[theorem]{Remark}
\newtheorem{example}{Example}

\DeclareMathOperator{\dt}{\mathrm{d}t}
\DeclareMathOperator{\timeint}{\int \limits_0^\infty}

\DeclareMathOperator{\R}{\mathbb{R}}

\begin{document}
\title{Boundary elements with mesh refinements for the wave equation }
\author{ Heiko Gimperlein\thanks{Maxwell Institute for Mathematical Sciences and Department of Mathematics, Heriot--Watt University, Edinburgh, EH14 4AS, United Kingdom, email: h.gimperlein@hw.ac.uk.}  \thanks{Institute for Mathematics, University of Paderborn, Warburger Str.~100, 33098 Paderborn, Germany.} \and Fabian Meyer\thanks{Institute of Applied Analysis and Numerical Simulation, Universit\"{a}t Stuttgart, Pfaffenwaldring 57, 70569 Stuttgart, Germany.} \and Ceyhun \"{O}zdemir\thanks{Institute of Applied Mathematics, Leibniz University Hannover, 30167 Hannover, Germany. \newline H.~G.~acknowledges support by ERC Advanced Grant HARG 268105 and the EPSRC Impact Acceleration Account. C.~\"{O}.~is supported by a scholarship of the Avicenna Foundation.}\and  David Stark${}^\ast$ \and Ernst P.~Stephan${}^\S$}

\providecommand{\keywords}[1]{{\textit{Key words:}} #1}

\maketitle \vskip 0.5cm
\begin{abstract}
\noindent The solution of the wave equation in a polyhedral domain in $\mathbb{R}^3$ admits an asymptotic singular expansion in a neighborhood of the corners and edges. In this article we formulate boundary and screen problems for the wave equation as equivalent boundary  integral equations in time domain, study the regularity properties of their solutions and the numerical approximation. Guided by the theory for elliptic equations, graded meshes are shown to recover the optimal approximation rates known for smooth solutions. Numerical experiments illustrate the theory for screen problems. In particular, we discuss the Dirichlet and Neumann problems, as well as the Dirichlet-to-Neumann operator and applications to the sound emission of tires.  \end{abstract}

\section{Introduction}
\label{intro}

For solutions to elliptic or parabolic equations in a polyhedral domain, the asymptotic behavior near the edges and corners has been studied for several decades \cite{mazya}. Numerically, the explicit singular expansions allow to recover optimal convergence rates for finite \cite{Babuska, Babuska2} and boundary element methods \cite{disspetersdorff, petersdorff}.

In the case of the wave equation in domains with conical or wedge singularities, a similar asymptotic behavior has been obtained by Plamenevskii and collaborators since the late 1990's \cite{kokotov, kokotov2, matyu, plamenevskii}. Their results imply that
at a fixed time $t$, the solution to the wave equation admits an explicit singular expansion with the same exponents as for elliptic equations.   Recently, M\"{u}ller and Schwab have used these results to obtain optimal convergence rates for a finite element method in polygonal domains in $\mathbb{R}^2$ \cite{mueller}.

The realistic scattering and diffraction of waves in $\mathbb{R}^3$ is crucially affected by geometric singularities of the scatterer, with significant new challenges for both the singular and numerical analysis. This article studies the solution of the wave equation in the most singular case, outside a screen $\Gamma$ in $\mathbb{R}^3$ or, equivalently, for an opening crack. From the singular expansion we obtain optimal convergence rates for piecewise polynomial approximations on graded meshes. Numerical experiments using a time domain boundary element method confirm the theoretical predictions and show their use for a real-world application in traffic noise.

To be specific, for a polyhedral screen $\Gamma \subset \mathbb{R}^3$ with connected complement $\Omega = \mathbb{R}^3\setminus \overline{\Gamma}$ {this article} considers  the wave  equation
\begin{subequations} \label{eq:oriProblem}
\begin{alignat}{2}
c^{-2}\partial_t^2 u(t,x) -\Delta u(t,x)&=0 &\quad &\text{in }\mathbb{R}^+_t \times \Omega_x \\
Bu&=g &\quad &\text{on }  \Gamma = \partial \Omega\\
{u}(0,x)=\partial_t u(0,x)&=0 &\quad &\text{in }  \Omega  
\end{alignat}
\end{subequations}
where either inhomogeneous Dirichlet boundary conditions $Bu = u|_\Gamma$ or Neumann boundary conditions $Bu = \partial_\nu u|_\Gamma$ are considered on $\Gamma$. Here, $c$ denotes the speed of sound and for simplicity, in most of the article we choose units such that $c=1$.\\

Based on the above-mentioned results of Plamenevskii {and coauthors}, we obtain a precise description of the singularities of the solution near edges and corners. The solution $u$ and its normal derivative on $\Gamma$ admit an asymptotic expansion with the same singular exponents as in the elliptic case.

As in the elliptic case,  the precise asymptotic description of the solution has implications for the approximation by time domain boundary elements. We formulate \eqref{eq:oriProblem} as a time dependent integral equation on $\Gamma$, with either the single layer, the hypersingular or the Dirichlet-to-Neumann operator. {The Dirichlet trace $u|_\Gamma$ is approximated by tensor products of piecewise polynomial functions $\widetilde{V}^{p,q}_{\Delta t,h}$ on a $\beta$-graded mesh in space and a uniform mesh in time of step size $\Delta t$. $\tilde{V}^{p,q}_{\Delta t,h}$ is defined in \eqref{fespace2}, and its analogue $V^{p,q}_{\Delta t,h}$ for the approximation of the Neumann trace $\partial_\nu u|_\Gamma$ in \eqref{fespace}. See the bottom of page 8 for the definition of the $\beta$-graded meshes.} Our main result for the approximation of the solutions to the boundary integral equations {in space-time anisotropic Sobolev spaces (Definition \ref{sobdef})} is a consequence of:\\

\noindent \textbf{Theorem A.} {\textit Let $\varepsilon>0$.}\\
\textit{ a) Let $u$ be a strong solution to the homogeneous wave equation with inhomogeneous Neumann boundary conditions $\partial_\nu u|_\Gamma = g$, with $g$ smooth. Further, let $\phi_{h,\Delta t}^\beta$ be the best approximation {in the norm of ${H}^{r}_\sigma(\R^+, \widetilde{H}^{\frac{1}{2}-s}(\Gamma))$ to the Dirichlet trace $u|_\Gamma$ in $\widetilde{V}^{p,1}_{\Delta t,h}$} on a $\beta$-graded spatial mesh with $\Delta t \lesssim h^\beta$. Then $\|u-\phi_{h, \Delta t}^\beta\|_{r,\frac{1}{2}-s, \Gamma, \ast} \leq C_{\beta,\varepsilon} h^{\min\{\beta(\frac{1}{2}+s), \frac{3}{2}+s\}{-\varepsilon}}$, where $s \in [0,\frac{1}{2}]$ and {$r \in [0,p)$}. }

\textit{ b) Let $u$ be a strong solution to the homogeneous wave equation with inhomogeneous Dirichlet boundary conditions $u|_\Gamma = g$, with $g$ smooth. Further, let $\psi_{h,\Delta t}^\beta$ be the best approximation  {in the norm of ${H}^{r}_\sigma(\R^+, \widetilde{H}^{-\frac{1}{2}}(\Gamma))$ to the Neumann trace $\partial_\nu u|_\Gamma$ in ${V}^{p,0}_{\Delta t,h}$} on a $\beta$-graded spatial mesh  with $\Delta t \lesssim h^\beta$. Then $\|\partial_\nu u-\psi_{h, \Delta t}^\beta\|_{r,-\frac{1}{2}, \Gamma, \ast} \leq C_{\beta,\varepsilon} h^{\min\{\frac{\beta}{2}, \frac{3}{2}\}{-\varepsilon}}$, where {$r \in [0,p+1)$}.  }\\

For the circular screen this result may be found in Theorem \ref{approxtheorem1}, while for the polygonal screen it is Theorem \ref{approxtheorem2} (assuming $\beta$ is sufficiently large). It implies an approximation result for the solution to the boundary integral formulations, see Corollary \ref{approxcor1} for the circular screen, respectively Corollary \ref{approxcor2} for the polygonal screen:\\

\noindent \textbf{Corollary B.} {\textit{Let $\varepsilon>0$.}\\
\textit{a) Let $\phi$ be the solution to the hypersingular integral equation $W \phi = g$ and  $\phi_{h,\Delta t}^\beta$ the best approximation {in the norm of ${H}^{r}_\sigma(\R^+, \widetilde{H}^{\frac{1}{2}-s}(\Gamma))$ to $\phi$ in $\widetilde{V}^{p,1}_{\Delta t,h}$} on a $\beta$-graded spatial mesh  with $\Delta t \lesssim h^\beta$. Then $\|\phi-\phi_{h, \Delta t}^\beta\|_{r,\frac{1}{2}-s, \Gamma, \ast} \leq C_{\beta,\varepsilon} h^{\min\{\beta(\frac{1}{2}+s), \frac{3}{2}+s\}{-\varepsilon}}$, where $s \in [0,\frac{1}{2}]$  and {$r \in [0,p)$}. }

\textit{b) Let $\psi$ be the solution to the single layer integral equation $V \psi = f$ and  $\psi_{h,\Delta t}^\beta$ the best approximation  {in the norm of ${H}^{r}_\sigma(\R^+, \widetilde{H}^{-\frac{1}{2}}(\Gamma))$ to $\psi$ in ${V}^{p,0}_{\Delta t,h}$} on a $\beta$-graded spatial mesh  with $\Delta t \lesssim h^\beta$. Then $\|\psi-\psi_{h, \Delta t}^\beta\|_{r,-\frac{1}{2}, \Gamma, \ast} \leq C_{\beta,\varepsilon} h^{\min\{\frac{\beta}{2}, \frac{3}{2}\}{-\varepsilon}}$, where {$r \in [0,p+1)$}. } \\

Indeed, on the flat screen the solutions to the integral equations are given by $\phi = \left[u\right]|_\Gamma$ {in terms of the solution $u$ which satisfies Neumann conditions $Bu = \partial_\nu u|_\Gamma = g$, respectively} $\psi = \left[\partial_\nu u\right]|_\Gamma$ {in terms of the solution $u$ which satisfies Dirichlet conditions $Bu = u|_\Gamma=f$}.

{Note that the energy norm associated to the weak form of the single layer integral equation \eqref{weakform} is weaker than the norm of ${H}^{1}_\sigma(\R^+, {H}^{-\frac{1}{2}}(\Gamma))$ and stronger than the norm of ${H}^{0}_\sigma(\R^+, {H}^{-\frac{1}{2}}(\Gamma))$, according to the coercivity and continuity properties of $V$ on screens \cite{HGEPSN}. Similarly, for the weak form of the hypersingular integral equation \eqref{weakformW}, the energy norm is weaker than the norm of ${H}^{1}_\sigma(\R^+, {H}^{\frac{1}{2}}(\Gamma))$ and stronger than the norm of ${H}^{0}_\sigma(\R^+, {H}^{\frac{1}{2}}(\Gamma))$ \cite{gimperleintyre}.\\}

\noindent \textbf{Remark C.} 
{\textit{Together with the a priori estimates for the time domain boundary element methods on screens \cite{HGEPSN, gimperleintyre}, Corollary B implies convergence rates for the Galerkin approximations, which recover those for smooth solutions {(up to an arbitrarily small $\varepsilon>0$)} provided the grading parameter $\beta$ is chosen sufficiently large. }}\\

We prove the  approximation properties in detail on the circular screen, without corners, and discuss the approximation of the corner singularity on polygonal screens. On the square, the convergence rate is determined by the singularities at the edges, in spite of the smaller singular exponents in a corner.
In all cases, we show that  \emph{time independent} algebraically graded meshes adapted to the singularities recover the optimal approximation rates expected for smooth solutions. \\

Numerical experiments confirm the theoretical results for the singular exponents and achieve the predicted convergence rates. Furthermore, they indicate the efficiency of our approach. 
 For the Dirichlet problem on a circular or square screen, reduced to an equation for the single layer operator, the convergence rate in the energy norm is doubled when the uniform mesh is replaced by a $2$-graded one. Similar results are obtained for the sound pressure, which is often the crucial quantity in applications. Even the singular exponents of the numerical solution near the edges and corners agree with those of the exact solution. The results generalize to the formulation of the Neumann problem as a hypersingular integral equation, where the predicted convergence rates and singular exponents at the edges are obtained.  The main difference to the Dirichlet problem is that the numerically computed singular exponents in the corner are in qualitative, though no longer quantitative agreement. Beyond these model problems, we study the Dirichlet-to-Neumann operator on screens, as relevant for dynamic interface and contact problems. The results reflect those for the hypersingular integral equation, and the errors due to the numerical approximation of the operator are seen to be negligible.\\

Finally, we show the relevance of graded meshes for a real-world question from traffic noise, where graded meshes allow to  accurately resolve the sound amplification around resonance frequencies.  \\

Graded meshes thus lead to optimal algorithms to resolve geometric singularities of the computational domain. They provide a key example for efficient approximations of the solution of transient wave equations by time-independent, adapted meshes. Such meshes also arise in adaptive algorithms based on time-integrated a posteriori error estimates \cite{apost}. \\

\noindent \emph{The article is organized as follows:} 
Section \ref{faframework} recalls the boundary integral operators associated to the wave equation as well as their mapping properties between suitable space-time anisotropic Sobolev spaces. It concludes by reformulating the Dirichlet and Neumann problems for the wave equation \eqref{eq:oriProblem} as boundary integral equations in the time domain. The following Section \ref{discretization} introduces graded meshes on $\Gamma$, corresponding space-time discretizations and a time domain boundary element method to solve the integral equations. The asymptotic expansions of solutions to the wave equation and their approximation are the content of Section \ref{sectasympt}, for circular and polygonal screens. Section \ref{algosect} discusses some algorithmic properties of the implementation, before numerical experiments are used to confirm the theoretical predictions in Section \ref{experiments}. The article concludes with a real-world application to traffic noise and computes the amplification of noise in the singular horn geometry between a tire and the road surface.

\section{Boundary integral operators and Sobolev spaces}
\label{faframework}

{To be specific, in $\mathbb{R}^3$ let $\Gamma$ be the boundary of a polyhedral domain, consisting of curved, polygonal boundary faces, or an open polyhedral surface (screen). In $\mathbb{R}^2$, $\Gamma$ is the boundary of a curved polygon, or $\Gamma$ is an open polygonal curve.} 

We make an ansatz for the solution to \eqref{eq:oriProblem} using the single layer potential in time domain,
\begin{equation}\label{singlay}
u(t,x) =\int_{\mathbb{R}^+ \times \Gamma} G(t- \tau,x,y)\ {\psi}(\tau,y)\ d\tau\ ds_y\ ,
\end{equation}
where $G$ is a fundamental solution to the wave equation and $\psi(\tau,y)=0$ for $\tau<0$. Specifically in 3 dimensions, we may choose
\begin{align*}
u(t,x) &=\frac{1}{{4} \pi} \int_\Gamma \frac{\psi(t-|x-y|,y)}{|x-y|}\ ds_y \ ,
\end{align*}
but for applications to traffic noise also different choices are relevant, see \eqref{greentraffic}.  Taking the Dirichlet boundary values on $\Gamma$ of the integral \eqref{singlay}, we obtain the single layer operator,
$$V \psi(t,x)={  \int_{\mathbb{R}^+\times \Gamma} G(t- \tau,x,y)\ \psi(\tau,y)\ d\tau\ ds_y\, , }$$
It allows to reduce the wave equation \eqref{eq:oriProblem} with Dirichlet boundary conditions, $u=f$ on $\Gamma$, to an equivalent integral equation
\begin{equation}\label{dirproblemV}
V \psi = u|_\Gamma =  f \ .
\end{equation}
After solving equation \eqref{dirproblemV} for the density $\psi$, the solution to the wave equation is obtained using equation \eqref{singlay}. \\

\noindent We also require the adjoint double layer  operator $K'$, as obtained from the Neumann boundary values, as well as the double layer operator $K$ and the hypersingular operator $W$ on $\Gamma$:
\begin{align}\label{operators}
\nonumber K\phi(t,x)&=\int_{\mathbb{R}^+\times \Gamma} \frac{\partial G}{\partial n_y}(t- \tau,x,y)\ \phi(\tau,y)\ d\tau\ ds_y,\\
K' \phi(t,x)&= \int_{\mathbb{R}^+\times \Gamma} \frac{\partial G}{\partial n_x}(t- \tau,x,y)\ \phi(\tau,y)\ d\tau\ ds_y\, ,\\
\nonumber W \phi(t,x)&= \int_{\mathbb{R}^+\times \Gamma} \frac{\partial^2 G}{\partial n_x \partial n_y}(t- \tau,x,y)\ \phi(\tau,y)\ d\tau\ ds_y \ .
\end{align}

\begin{remark}\label{vanish}
For a flat screen $\Gamma \subset \mathbb{R}^2 \times \{0\}$, the normal derivative of $G$ vanishes, and $K\phi=K'\phi=0$ in this case.
\end{remark}

{The boundary integral operators} are considered between space-time aniso\-tropic Sobolev spaces $H_\sigma^{{r}}(\mathbb{R}^+,\widetilde{H}^{{s}}(\Gamma))$, see  \cite{HGEPSN} or \cite{haduong}. {To define them, if $\partial\Gamma\neq \emptyset$, first extend $\Gamma$ to a closed, orientable Lipschitz manifold $\widetilde{\Gamma}$. }

{On $\Gamma$ one defines the usual Sobolev spaces of supported distributions:
$$\widetilde{H}^{{s}}(\Gamma) = \{u\in H^{{s}}(\widetilde{\Gamma}): \mathrm{supp}\ u \subset {\overline{\Gamma}}\}\ , \quad\ {{s}} \in \mathbb{R}\ .$$
Furthermore, ${H}^{{s}}(\Gamma)$ is the quotient space $ H^{{s}}(\widetilde{\Gamma}) / \widetilde{H}^{{s}}({\widetilde{\Gamma}\setminus\overline{\Gamma}})$.} \\
{To write down an explicit family of Sobolev norms, introduce a partition of unity $\alpha_i$ subordinate to a covering of $\widetilde{\Gamma}$ by open sets $B_i$. For diffeomorphisms $\phi_i$ mapping each $B_i$ into the unit cube $\subset \mathbb{R}^n$, a family of Sobolev norms is induced from $\mathbb{R}^d$:
\begin{equation*}
 ||u||_{{{s}},\omega,{\widetilde{\Gamma}}}=\left( \sum_{i=1}^p \int_{\mathbb{R}^n} (|\omega|^2+|\xi|^2)^{{s}}|\mathcal{F}\left\{(\alpha_i u)\circ \phi_i^{-1}\right\}(\xi)|^2 d\xi \right)^{\frac{1}{2}}\ .
\end{equation*}
The norms for different $\omega \in \mathbb{C}\setminus \{0\}$ are equivalent, and $\mathcal{F}$ denotes the Fourier transform. They induce norms on $H^{{s}}(\Gamma)$, $||u||_{{{s}},\omega,\Gamma} = \inf_{v \in \widetilde{H}^{{s}}(\widetilde{\Gamma}\setminus\overline{\Gamma})} \ ||u+v||_{{{s}},\omega,\widetilde{\Gamma}}$, and on $\widetilde{H}^{{s}}(\Gamma)$, $||u||_{{{s}},\omega,\Gamma, \ast } = ||e_+ u||_{{{s}},\omega,\widetilde{\Gamma}}$. $e_+$ extends the distribution $u$ by $0$ from $\Gamma$ to $\widetilde{\Gamma}$. It is stronger than $||u||_{{{s}},\omega,\Gamma}$ whenever ${{s}} \in \frac{1}{2} + \mathbb{Z}$.}

{We now define a class of space-time anisotropic Sobolev spaces:
\begin{definition}\label{sobdef}
For ${{r,s}} \in\mathbb{R}$ {and $\sigma>0$} define
\begin{align*}
 H^{{r}}_\sigma(\mathbb{R}^+,{H}^{{s}}(\Gamma))&=\{ u \in \mathcal{D}^{'}_{+}(H^{{s}}(\Gamma)): e^{-\sigma t} u \in \mathcal{S}^{'}_{+}(H^{{s}}(\Gamma))  \textrm{ and }   ||u||_{{{r,s}},\Gamma} < \infty \}\ , \\
 H^{{r}}_\sigma(\mathbb{R}^+,\widetilde{H}^{{s}}({\Gamma}))&=\{ u \in \mathcal{D}^{'}_{+}(\widetilde{H}^{{s}}({\Gamma})): e^{-\sigma t} u \in \mathcal{S}^{'}_{+}(\widetilde{H}^{{s}}({\Gamma}))  \textrm{ and }   ||u||_{{{r,s}},\Gamma, \ast} < \infty \}\ .
\end{align*}
$\mathcal{D}^{'}_{+}(E)$ resp.~$\mathcal{S}^{'}_{+}(E)$ denote the spaces of distributions, resp.~tempered distributions, on $\mathbb{R}$ with support in $[0,\infty)$, taking values in $E = {H}^{{s}}({\Gamma}), \widetilde{H}^{{s}}({\Gamma})$. The relevant norms are given by
\begin{align*}
\|u\|_{{{r,s}},\Gamma}&=\left(\int_{-\infty+i\sigma}^{+\infty+i\sigma}|\omega|^{2{{r}}}\ \|\hat{u}(\omega)\|^2_{{{s}},\omega,\Gamma}\ d\omega \right)^{\frac{1}{2}}\ ,\\
\|u\|_{{{r,s}},\Gamma,\ast}&=\left(\int_{-\infty+i\sigma}^{+\infty+i\sigma}|\omega|^{2{{r}}}\ \|\hat{u}(\omega)\|^2_{{{s}},\omega,\Gamma,\ast}\ d\omega \right)^{\frac{1}{2}}\,.
\end{align*}
\end{definition}
For $|{{s}}|\leq 1$ the spaces are independent of the choice of $\alpha_i$ and $\phi_i$. 

A useful technical result localizes estimates for fractional Sobolev norms, extending \cite[Lemma 3.2]{disspetersdorff} to space-time:
\begin{lemma}\label{lemma3.2}
Let $\Gamma,\, \Gamma_j\; (j=1,\dots ,N)$ be Lipschitz domains with $\overline{\Gamma} = \bigcup\limits_{j=1}^N \overline{\Gamma}_j$, $\tilde{u}\in {H^r_\sigma(\mathbb{R}^+}, \widetilde{H}^s(\Gamma)),\; u\in {H^r_\sigma(\mathbb{R}^+}, H^s(\Gamma)),\; s\in\mathbb{R}.$ Then for all $s \in [-1,1]$, $r \in \mathbb{R}$ {and $\sigma>0$}
\begin{align}
	\sum\limits_{j=1}^N \| u\|^2_{r,s,\Gamma_j} &\leq \| u\|^2_{r,s,\Gamma} \label{3.21a}\ ,\\
	\| \tilde{u}\|^2_{r,s,\Gamma, \ast} &\leq \sum\limits_{j=1}^N \| \tilde{u}\|^2_{r,s,\Gamma_j, \ast}\ .\label{3.21b}
\end{align}
\end{lemma}
The proof is an immediate extension of the time-independent case.

{The boundary integral operators} obey the following mapping properties between the space-time Sobolev spaces:
\begin{theorem}[\cite{HGEPSN}]\label{mappingproperties}
The following operators are continuous for $r\in \R$, ${{\sigma>0}}$:
\begin{align*}
& V:  {H}^{r+1}_\sigma(\R^+, \tilde{H}^{-\frac{1}{2}}(\Gamma))\to  {H}^{r}_\sigma(\R^+, {H}^{\frac{1}{2}}(\Gamma)) \ ,
\\ & K':  {H}^{r+1}_\sigma(\R^+, \tilde{H}^{-\frac{1}{2}}(\Gamma))\to {H}^{r}_\sigma(\R^+, {H}^{-\frac{1}{2}}(\Gamma)) \ ,
\\ & K:  {H}^{r+1}_\sigma(\R^+, \tilde{H}^{\frac{1}{2}}(\Gamma))\to {H}^{r}_\sigma(\R^+, {H}^{\frac{1}{2}}(\Gamma)) \ ,
\\ & W:  {H}^{r+1}_\sigma(\R^+, \tilde{H}^{\frac{1}{2}}(\Gamma)))\to {H}^{r}_\sigma(\R^+, {H}^{-\frac{1}{2}}(\Gamma)) \ .
\end{align*}
\end{theorem}

\noindent When {$\Gamma = \mathbb{R}^{n-1}_+$}, Fourier methods yield improved estimates for $V$ and $W$:
\begin{theorem}[\cite{haduongcrack}, pp. 503-506]\label{mapimproved} The following operators are continuous for $r,s \in \mathbb{R}$, ${{\sigma>0}}$:
\begin{align*}
& V:{H}^{r+\frac{1}{2}}_\sigma(\R^+, \tilde{H}^{s}(\Gamma))\to {H}^{r}_\sigma(\R^+, {H}^{s+1}(\Gamma)) \ , \\
& W: {H}^{r}_\sigma(\R^+, \tilde{H}^{s}(\Gamma))\to {H}^{r}_\sigma(\R^+, {H}^{s-1}(\Gamma))\ .
\end{align*}
\end{theorem}

The space-time Sobolev spaces allow a precise statement and analysis of the weak formulation for the Dirichlet problem \eqref{dirproblemV}:
{Find $\psi \in H^1_{\sigma}(\mathbb{R}^+,\widetilde{H}^{-\frac{1}{2}}(\Gamma))$ such that} for all $\Psi\in H^1_{\sigma}(\mathbb{R}^+,\widetilde{H}^{-\frac{1}{2}}(\Gamma))$
\begin{equation}\label{weakform}
 \int_0^\infty\int_\Gamma (V \psi(t,{\bf x})) \partial_t\Psi(t,{\bf x}) \ ds_x\ d_\sigma t = \int_0^\infty \int_\Gamma f(t,{\bf x}) \partial_t\Psi(t,{\bf x}) \ ds_x\ d_\sigma t\ , 
\end{equation}
where $d_\sigma t = e^{-2 \sigma t}dt$.\\

To obtain an analogous weak formulation for  the Neumann problem, one starts from a double layer potential ansatz for $u$:
\begin{align}\label{doubleansatz}
u(t,x)&=\int_{\mathbb{R}^+ \times \Gamma}  \frac{\partial G}{\partial n_y}(t- \tau,x,y)\ \phi(\tau,y)\ d\tau\ ds_y
\end{align}
with $\phi(s,y) = 0$ for $s\leq 0$. The corresponding integral formulation is the hypersingular equation
\begin{align}\label{hypersingeq}
 W \phi = \frac{\partial u}{\partial n}\Big|_\Gamma= g\ .
\end{align}

Find $\phi \in H^1_\sigma(\mathbb{R}^+, \widetilde{H}^{\frac{1}{2}}(\Gamma))$ such that for all $\Phi \in H^1_\sigma(\mathbb{R}^+, \widetilde{H}^{\frac{1}{2}}(\Gamma))$ there holds:
\begin{align}\label{weakformW}
 \int_{\mathbb{R}^+\times \Gamma} (W  \phi(t,{\bf x})) \ {\partial_t}\Phi(t,{\bf x}) \,d_\sigma t \, ds_x\  = \int_{\mathbb{R}^+\times \Gamma}   g(t,{\bf x})\ {\partial_t}\Phi(t,{\bf x})\,dt \, ds_x\ .
\end{align}

The weak formulations \eqref{weakform}, respectively \eqref{weakformW}, for the Dirichlet and Neumann problems are well-posed \cite{HGEPSN, gimperleintyre}:
\begin{theorem}{Let $\sigma>0$.}\\
a) Assume that $f \in H^2_{\sigma}(\mathbb{R}^+,H^{\frac{1}{2}}(\Gamma))$. Then there exists a unique solution $\psi \in H^1_{\sigma}(\mathbb{R}^+,\widetilde{H}^{-\frac{1}{2}}(\Gamma))$  of \eqref{weakform} and
\begin{equation}
\|\psi\|_{1, -\frac{1}{2}, \Gamma, \ast} \lesssim_\sigma \|f\|_{2, \frac{1}{2}, \Gamma}\ .
\end{equation}
b) Assume that $g \in H^{2}_{\sigma}(\mathbb{R}^+,H^{-\frac{1}{2}}(\Gamma))$. Then there exists a unique solution $\phi \in H^{1}_{\sigma}(\mathbb{R}^+,\widetilde{H}^{\frac{1}{2}}(\Gamma))$  of \eqref{weakformW} and
\begin{equation}
\|\phi\|_{1,\frac{1}{2}, \Gamma, \ast} \leq C \|g\|_{2,-\frac{1}{2}, \Gamma} \ .
\end{equation}
\end{theorem}
While a theoretical analysis requires $\sigma>0$, practical computations use $\sigma=0$ \cite{BamHa,gimperleinreview}.\\

With a view towards contact problems \cite{contact}, we also consider an equation for the Dirichlet-to-Neumann operator $\mathcal{S}_\sigma$. For $\sigma>0$ and given boundary data $u_\sigma$, we consider
\begin{align}\label{1.3}
\begin{cases}
\left(\frac{\partial}{\partial t}+\sigma\right)^2w_\sigma - \Delta w_\sigma = 0 \ ,\hspace{1cm}  &\text{for}\ (t,x)\in\mathbb{R}\times\Omega\ ,
\\ w_\sigma=u_\sigma\ ,\hspace{1cm}  &\text{for}\ (t,x)\in\mathbb{R}\times\Gamma\ , \\ w_\sigma=0, &\text{for}\ (t,x)\in(-\infty,0)\times\Omega\ .
\end{cases}
\end{align}
The Dirichlet-to-Neumann operator is defined as \begin{align}\label{1.4}
\mathcal{S}_\sigma u_\sigma|_{\Gamma} := \frac{\partial w_\sigma}{\partial \nu}\Big|_{\Gamma}\ ,
\end{align}
We recall from \cite{sako}, p.~48:
\begin{theorem}\label{sakosol}
Let $h\in  H^{\frac{3}{2}}_\sigma(\mathbb{R}^+,{H}^{-\frac{1}{2}}(\Gamma))$. Then there exists a unique $u_\sigma\in {H}^{\frac{1}{2}}_\sigma(\R^+, \tilde{H}^{\frac{1}{2}}(\Gamma))$ such that {for all} $v \in {H}^{-\frac{1}{2}}_\sigma(\R^+, \tilde{H}^{\frac{1}{2}}(\Gamma))$:	
\begin{align}\label{DNeq}
\langle \mathcal{S}_\sigma u_\sigma,v \rangle =\langle h,v \rangle \ .
\end{align} 
\end{theorem}

\section{Discretization}
\label{discretization}
For the time discretization we consider a uniform decomposition of the time interval $[0,\infty)$ into subintervals $[t_{n-1}, t_n)$ with time step $\Delta t$, such that $t_n=n\Delta t \; (n=0,1,\dots)$. \\

{In $\mathbb{R}^3$, we may assume that $\Gamma$ consists of } closed triangular faces $\Gamma_i$ such that $\Gamma=\cup_{i} \Gamma_i$. {In $\mathbb{R}^2$, $\Gamma=\cup_{i} \Gamma_i$ is partitioned into} line segments $\Gamma_i$.} 

We choose a basis $\{ \xi_h^1 , \cdots , \xi_h^{N_s}\}$ of the space $V_h^q(\Gamma)$ of piecewise polynomial functions of degree $q$ in
space. Moreover we define $\widetilde{V}_h^q(\Gamma)$ as the space $V_h^q(\Gamma)$, where the polynomials vanish on $\partial \Gamma$ for $q\geq 1$. For the time discretization we choose a basis $\{\beta_{\Delta t}^1,\cdots,\beta_{\Delta t}^{N_t}\}$  of  the space $V^p_{t}$ of piecewise  polynomial  functions of degree of $p$ in time (continuous and vanishing at $t=0$ if $p\geq 1$).\\
Let $\mathcal{T}_S={\{\Delta_1,\cdots,\Delta_{N}\}}$ be a quasi-uniform triangulation of $\Gamma$ and $\mathcal{T}_T=\{[0,t_1),[t_1,t_2),\cdots,$ $[t_{M-1},T)\}$ the time mesh for a finite subinterval $[0,T)$.\\

We consider the tensor product of the approximation spaces in space and time, $V_h^q$ and $V^p_{\Delta t}$, associated to the space-time mesh $\mathcal{T}_{S,T}=\mathcal{T}_S \times\mathcal{T}_T$, and we write
\begin{align}\label{fespace}
V_{\Delta t,h}^{p,q}:=  V_{\Delta t}^p \otimes V_{h}^q\ .
\end{align}
We analogously define 
\begin{align}\label{fespace2}
\tilde{V}_{\Delta t,h}^{p,q}:=  V_{\Delta t}^p \otimes \tilde{V}_{h}^q\ .
\end{align}
For $u_{\Delta t,h} \in V_{\Delta t,h}^{p,q} $ we thus may write 
\begin{align*}
u_{\Delta t,h}(t,x)=\sum \limits_{i=0}^{N_t} \sum \limits_{j=0}^{N_s} c_j^i \beta_{\Delta t}^i(t)\xi_h^j(x) \ .
\end{align*} 
In the following we use the notation
\begin{itemize}
\item $\gamma_{\Delta t}^n(t)$ \text{ for the basis of piecewise constant functions in time,}
\item $\beta_{\Delta t}^n(t)$ \text{for the basis of piecewise linear functions in time,}
\item $\psi_{h}^i(x)$ \text{for the basis of piecewise  constant functions in space,}
\item $\xi_{h}^i(x)$ \text{for the basis of piecewise linear functions in space.}
\end{itemize}

The Galerkin discretization of the Dirichlet problem \eqref{weakform} is then given by:\\

\noindent Find $\psi_{\Delta t, h} \in V_{\Delta t,h}^{p,q}$ such that} for all $\Psi_{\Delta t, h}\in V_{\Delta t,h}^{p,q}$
\begin{equation}\label{weakformh}
 \int_0^\infty\int_\Gamma (V \psi_{\Delta t, h}(t,{\bf x})) \partial_t\Psi_{\Delta t, h}(t,{\bf x}) \ ds_x\ d_\sigma t = \int_0^\infty \int_\Gamma f(t,{\bf x}) \partial_t\Psi_{\Delta t, h}(t,{\bf x}) \ ds_x\ d_\sigma t\ .
\end{equation}

For the Neumann problem \eqref{weakformW}, we have:\\

\noindent Find $\phi_{\Delta t, h} \in \widetilde{V}_{t,h}^{p,q}$ such that for all $\Phi_{\Delta t, h}\in \widetilde{V}_{t,h}^{p,q}$
\begin{equation}\label{weakformWh}
 \int_0^\infty\int_\Gamma (W \phi_{\Delta t, h}(t,{\bf x})) \partial_t\Phi_{\Delta t, h}(t,{\bf x}) \ ds_x\ d_\sigma t = \int_0^\infty \int_\Gamma g(t,{\bf x}) \partial_t\Phi_{\Delta t, h}(t,{\bf x}) \ ds_x\ d_\sigma t\ .
\end{equation}

From the weak coercivity of $V$, respectively $W$, the discretized problems \eqref{weakformh} and \eqref{weakformWh} admit unique solutions.\\

 Our computations are mainly conducted on graded meshes on the square [$-1,1]^2$, respectively on the circular screen $\{(x,y,0) : \sqrt{x^2+y^2}\leq 1\}$. To define $\beta$-graded meshes on the square, due to symmetry, it suffices to consider a $\beta$-graded mesh on $[-1,0]$. We define $y_k=x_k=-1+(\frac{k}{N_l})^{\beta}$ for $k=1,\ldots,N_l$ and for a constant $\beta\geq 0$. The nodes of the $\beta$-graded mesh {on the square} are therefore $(x_k,y_l),~k,l=1,\ldots, N_l$. We note that for $\beta=1$ we would have a uniform mesh.  \\

In a general convex, polyhedral geometry graded meshes are locally modeled on this example. In particular, on the circular screen of radius $1$, for $\beta =1$ we take a uniform mesh with nodes on concentric circles of radius $r_k=1-\frac{k}{N_l}$ for $k=0,\ldots,N_l-1$. For the $\beta$-graded mesh, the radii are moved to $r_k=1-(\frac{k}{N_l})^\beta$ for $k=0,\ldots,N_l-1$. While the triangles become increasingly flat near the boundary, their total number remains proportional to $N_l^2$. 

Examples of the resulting $2$-graded meshes on the square and the circular screens are depicted in Figure \ref{gradmesh}. \\

 \begin{figure}[t] \centering
   \subfigure[]{
  \includegraphics[height=5.3cm, width=5.3cm]{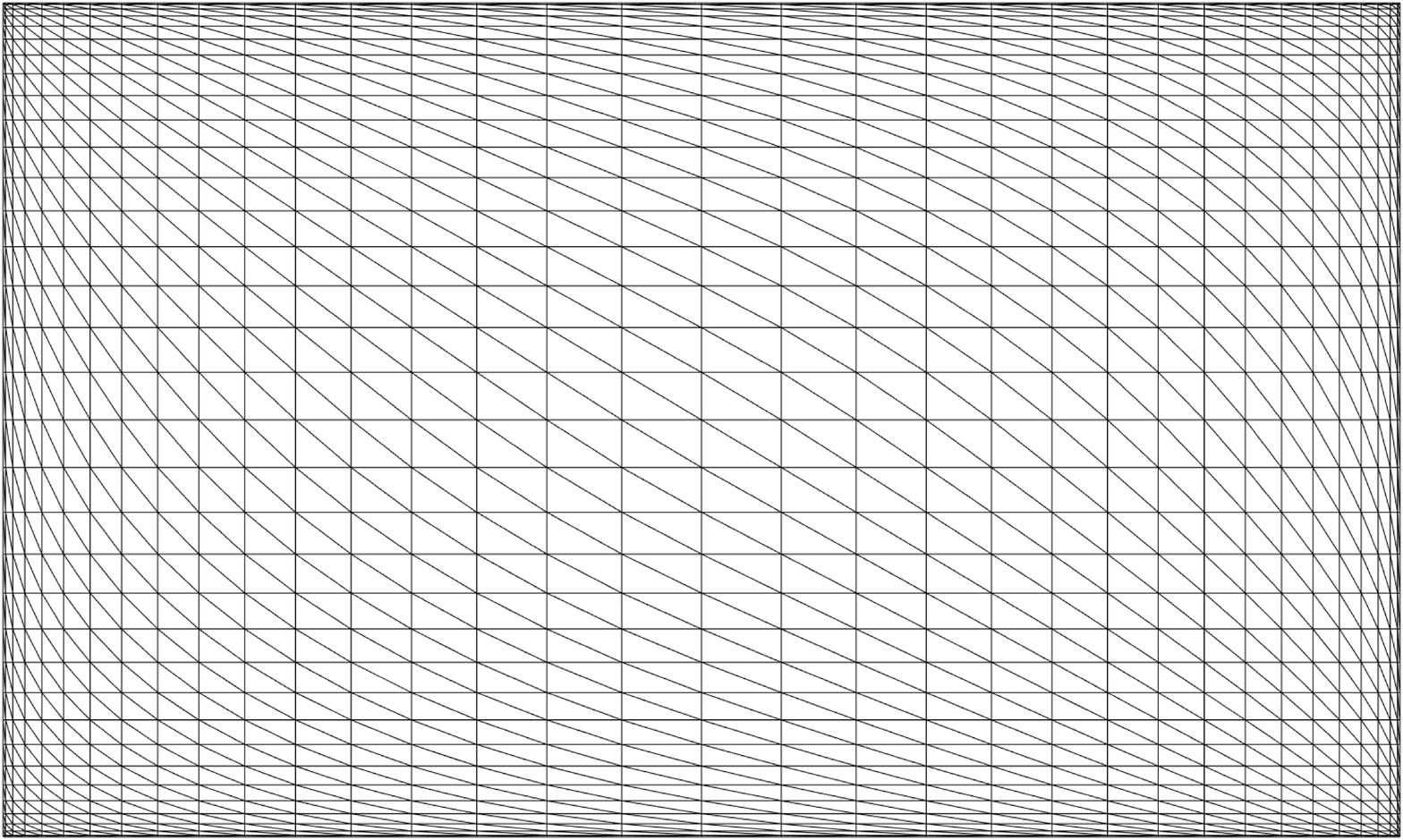}}
 \subfigure[]{
  \includegraphics[height=5.3cm, width=5.3cm]{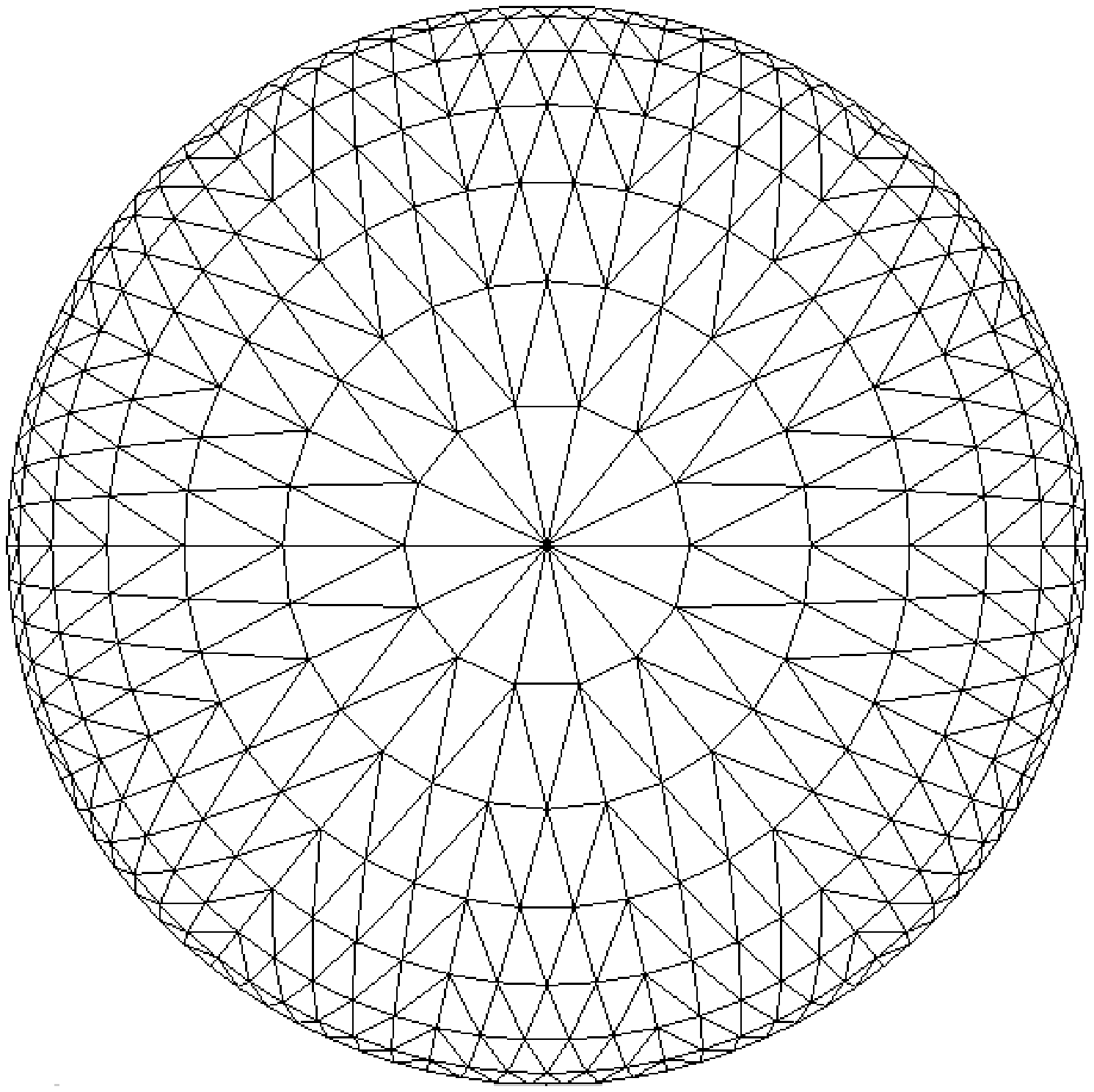}}
   \caption{$\beta$-graded meshes for (a) square and (b) circular screens, with $\beta=2$}\label{gradmesh}
 \end{figure}

While we use triangular meshes in our computations, for the ease of presentation we first discuss the approximation properties of graded meshes with rectangular elements. Reference \cite{disspetersdorff} shows how to deduce approximation results on triangular meshes from the rectangular case. 

Key ingredients in our analysis are projections from $L^2(\Gamma)$ onto $V_h^p$ on the graded mesh. We collect some  key approximation properties used below:

An analogon of \cite[Lemma 3.3]{disspetersdorff} reads:
\begin{lemma}\label{lemma3.3}
	Let $r\geq 0$, $0\leq s_1, s_2\leq 1$, $I_j = [0, h_j],\; f_2\in\widetilde{H}^{-s_2}(I_2)$, $f_1 \in \widetilde{H}^r_\sigma(\mathbb{R}^+, H^{-s_1}(I_1))$. Then there holds
	\begin{equation*}	
		\| f_1(t,x) f_2(y) \|_{r, -s_1 - s_2, I_1\times I_2, \ast} \leq \| f_1\|_{r, -s_1, I_1, \ast} \| f_2\|_{\tilde{H}^{-s_2}(I_2)} \ .
	\end{equation*}
\end{lemma}
\begin{proof}
This is a consequence of the estimate $$(\sigma^2+|\omega|^2 + \xi_1^2 + \xi_2^2)^{-(s_1+s_2)/2} \lesssim (\sigma^2+|\omega|^2 + \xi_1^2 )^{-s_1/2} (1+ \xi_2^2 )^{-s_2/2}\ $$
in Fourier space.
\end{proof}
We have a similar result for positive Sobolev indices:
\begin{lemma}\label{lemma3.3a}
	Let $r\geq 0$, $0\leq s\leq 1$, $I_j = [0, h_j],\; f_2\in\widetilde{H}^{s}(I_2)$, $f_1 \in {H}^r_\sigma(\mathbb{R}^+, \widetilde{H}^{s}(I_1))$. Then there holds
	\begin{equation*}	
		\| f_1(t,x) f_2(y) \|_{r, s, I_1\times I_2, \ast} \leq \| f_1\|_{r, s, I_1, \ast} \| f_2\|_{\tilde{H}^{s}(I_2)}\ .
	\end{equation*}
\end{lemma}
\begin{proof}
This is a consequence of the estimate $$(\sigma^2+|\omega|^2 + \xi_1^2 + \xi_2^2)^{s/2} \lesssim (\sigma^2+|\omega|^2 + \xi_1^2 )^{s/2} (1+ \xi_2^2 )^{s/2}$$ in Fourier space.
\end{proof}

Next we approximate $\tilde{H}^s$-functions on rectangles by constants, as in \cite[Lemma 3.4]{disspetersdorff}. The proof is a combination with \cite[Proposition 3.54 and 3.57]{glaefke}, see also \cite{HGEPSN} for screens. The formulation localizes from $\mathbb{R}^+$ to a single time interval $[0,\Delta t]$, and uses the restriction $H^{r}_\sigma([0,\Delta t], H^s(R))$ of $H^{r}_\sigma(\mathbb{R}^+, H^s(R))$.
\begin{lemma}\label{lemma3.4}
	Let $-1\leq s\leq 0,\; {0\leq r\leq \rho \leq p+1},\; R=[0,h_1]\times[0,h_2],\; u\in H^{\rho}_\sigma([0,\Delta t], H^1(R))$, $\Pi_t^{{p}} u$ the orthogonal projection onto piecewise polynomials in $t$ of order {$p$}, $\Pi_{x,y}^0 u=\frac{1}{h_1 h_2}\int\limits_R u(t,x,y)  dy\, dx$. Then for ${U} =  \Pi_t^{{p}} \Pi^0_{x,y} u$ we have 
	\begin{align}\label{3.22}
		\| u-{U}\|_{r,s,R,\ast}&\lesssim (\Delta t)^{\rho-r}{\max\{h_1,h_2,\Delta t \}^{-s}}\|\partial_t^\rho u\|_{L^2([0,\Delta t]\times R)} \\ & +\qquad  \max\{h_1,h_2,\Delta t \}^{-s}\left(h_1 \| u_x\|_{L^2([0,\Delta t] \times R)}  + h_2 \| u_y\|_{L^2([0,\Delta t] \times R)} \right)\ . \nonumber
	\end{align}
	If $u(t,x,y) =u_1(t,x)u_2(y),\; u_1\in H^{\rho}_\sigma( [0,\Delta t], H^1([0,h_1])), \; u_2\in H^1( [0,h_2])$ then 
	\begin{align*}
		\| u-{U}\|_{r,s,R,\ast}& \lesssim (\Delta t)^{\rho-r}{\max\{h_1,\Delta t \}^{-s}}\|\partial_t^\rho u\|_{L^2([0,\Delta t]\times R)} \\ & \qquad+\left(h_1^{1-s}\| u_x\|_{{L^2}([0,\Delta t] \times R)} + h_2^{1-s}\| u_y\|_{{L^2}([0,\Delta t] \times R)} \right)\ .
	\end{align*}
\end{lemma}
{\begin{proof}
As the proof is similar to the time-independent case, \cite[Lemma 3.4]{disspetersdorff}, we only show \eqref{3.22} for $r=0$, $s=-1$. First note that 
\begin{equation}\label{auxineq} \| u-{U}\|_{0,0,R,\ast} \lesssim (\Delta t) \|\partial_t u\|_{0,0,R}+ h_1 \|u_x\|_{0,0,R}+ h_2 \|u_y\|_{0,0,R} \ .\end{equation}
By the Hahn-Banach theorem  we have   
\begin{align*}
\| u-{U}\|_{0,-1,R,\ast} & = \sup_{v \in H^0_\sigma([0,\Delta t], H^1(R))} \frac{|\langle u-U , v\rangle|}{\|v\|_{0,1,R}}\\& = \sup_{v \in H^0_\sigma([0,\Delta t], H^1(R))} \frac{|\langle u-U , v-Z\rangle|}{\|v\|_{0,1,R}} \\& \leq \|u-U \|_{0,0,R}  \sup_{v \in H^0_\sigma([0,\Delta t], H^1(R))} \frac{\|v-Z\|_{0,0,R}}{\|v\|_{0,1,R}} \ ,
\end{align*}
for  any constant $Z$. Using \eqref{auxineq} on the right hand side, we obtain 
\begin{align*}\| u-{U}\|_{0,-1,R,\ast} &\lesssim \left((\Delta t) \|\partial_t u\|_{0,0,R}+ h_1 \|u_x\|_{0,0,R}+ h_2 \|u_y\|_{0,0,R}\right)\\ & \qquad \ \ \ \sup_{v \in H^0_\sigma([0,\Delta t], H^1(R))} \frac{(\Delta t) \|\partial_t v\|_{0,0,R}+ h_1 \|v_x\|_{0,0,R}+ h_2 \|v_y\|_{0,0,R}}{\|v\|_{0,1,R}} \\ & \lesssim \left((\Delta t) \|\partial_t u\|_{0,0,R}+ h_1 \|u_x\|_{0,0,R}+ h_2 \|u_y\|_{0,0,R}\right) \max\{h_1,h_2,\Delta t \} \ .
\end{align*}
The general case of \eqref{3.22} follows by interpolation and by using the higher smoothness in $t$.\\
The proof of the second inequality applies these arguments and Lemma \ref{lemma3.3} to the factorization $u - U = (u_1 - U_1) (u_2-U_2)$. Here $U = U_1 U_2$, with $U_1 = \Pi_t^p \Pi_x^0 u_1$ and $U_2 = \Pi_y^0 u_2$. 
\end{proof}}

An analogous result holds for bilinear interpolants on rectangles, as in \cite[Lemma 3.14]{disspetersdorff}.
\begin{lemma}\label{TobiasTRACE}
Let $Q = [0,h_1]\times [0,h_2],u\in H^3_\sigma([0,\Delta t]\times Q)$, ${U}$ the bilinear interpolant of $u$ at the vertices of $Q$. Then there holds for {$r \geq 0$}
\begin{align}
	\| u-{U}\|_{r,0,[0,\Delta t)\times Q} &\lesssim\max\{h_1, \Delta t\}^2\| u_{xx}\|_{r,0,[0,\Delta t)\times Q} + \max\{h_2, \Delta t\}^2\| u_{yy}\|_{r,0,[0,\Delta t)\times Q}\nonumber \\  & \qquad +  ( \max\{h_1, \Delta t\}^2+ \max\{h_2, \Delta t\}^2 )\| u_{tt}\|_{r,0,[0,\Delta t)\times Q} \nonumber \\ & \qquad + \max\{h_1, \Delta t\}^2 \max\{h_2, \Delta t\}\| u_{xxy}\|_{r,0,[0,\Delta t)\times Q} \ , \label{bilinearinterpolant2D1} \\
	\|(u-{U})_x\|_{r,0,[0,\Delta t)\times Q} &\lesssim\max\{h_1, \Delta t\}\| u_{xx}\|_{r,0,[0,\Delta t)\times Q} + \max\{h_1, \Delta t\}\| u_{xt}\|_{r,0,[0,\Delta t)\times Q} \nonumber \\ &\qquad+ \max\{h_2, \Delta t\}^2\| u_{xyy}\|_{L^2(Q)} \ . \label{bilinearinterpolant2D2}
\end{align}
\end{lemma}

The proofs of the following results are given in  \cite[Satz 3.7, Satz 3.10]{disspetersdorff}.
\begin{lemma}\label{keylemmagrad}For $a>0$ and $s \in [-1, -a+\frac{1}{2})$ there holds with the piecewise constant interpolant $ \Pi_{y}^{0} y^{-a}$ of $ y^{-a}$ {on the $\beta$-graded mesh}
$$\|y^{-a} - \Pi_{y}^{0} y^{-a}\|_{\widetilde{H}^{s}([0,1])} \lesssim  h^{\min\{\beta(-a-s+\frac{1}{2}), {1}-s\}{-\varepsilon}} . $$\ 
\end{lemma}

\begin{lemma}\label{keylemmagrad2}For $a>0$ and $s \in [0, a+\frac{1}{2})$ there holds with the linear interpolant $ \Pi_{y}^{1} y^{a}$ of $ y^{a}$ {on the $\beta$-graded mesh}
$$\|y^{a} - \Pi_{y}^{1} y^a\|_{\widetilde{H}^{s}([0,1])} \lesssim  h^{\min\{\beta(a-s+\frac{1}{2}), 2-s\}{-\varepsilon}} . $$
\end{lemma}

\section{Asymptotic expansions and numerical approximation}\label{sectasympt}

\subsection{Asymptotic expansion of solutions to the wave equation in a wedge}

Solutions of the Laplace and Helmholtz equations exhibit well-known singularities at non-smooth boundary points of the domain. In this section we describe a similar decomposition of the solution to the wave equation with Dirichlet or Neumann boundary conditions near an edge or a corner, into a leading part given by explicit singular functions plus less singular terms.  The strategy of translating the results from the Helmholtz equation to the time-dependent wave or Lam\'{e} equations has  been studied in a series of papers by Plamenevskii and coauthors \cite{kokotov, kokotov2, matyu, plamenevskii}. We here recall their key result for a wedge.

To be specific, let $0\leq d\leq n-2$ and $K \subset \mathbb{R}^{n-d}$ an open cone with vertex at $0$, smooth outside the vertex. We denote by $\mathcal{K} = K \times \mathbb{R}^d$ the wedge over $K$ and consider the wave equation in $\mathcal{K}$: 
\begin{subequations} \label{eq:oriProblem2}
\begin{alignat}{2}
\partial_t^2 u(t,x) -\Delta u(t,x)&=0 &\quad &\text{in }\mathbb{R}^+_t \times \mathcal{K}_x\ , \\
Bu&=g &\quad &\text{on }  \Gamma = \partial \mathcal{K} \ ,\\
{u}(0,x)=\partial_t u(0,x)&=0 &\quad &\text{in } \mathcal{K} , 
\end{alignat}
\end{subequations}
where either inhomogeneous Dirichlet boundary conditions $Bu = u|_\Gamma$ or Neumann boundary conditions $Bu = \partial_\nu u|_\Gamma$ are considered on $\Gamma$.  We will describe the asymptotic behavior of a solution to the wave equation with Dirichlet or Neumann boundary conditions in $\mathcal{K}$ near $\{0\}\times \mathbb{R}^d$. Locally, the edge of a screen in $\mathbb{R}^3$ corresponds to $d=1$, a cone point to $d=0$.

The analysis uses the Fourier-Laplace transformation in time to reduce the {time dependent} problem to the Helmholtz equation {with frequency $\omega$. Then a Fourier transform is applied changing $z \in \mathbb{R}^d$ into $\zeta \in \mathbb{R}^d$. Using polar coordinates, the conical variable $y \in K$ is transformed into the radius $r$ and the spherical variable $\theta$. A series expansion is applied, where the eigenfunctions are determined by separation of variables.}

{More concretely, the Fourier-Laplace transform leads to the Helmholtz equation:}
\begin{align}\label{helmholtz} \nonumber
\omega^2\hat{u}(\omega,x)+ \Delta \hat{u}(\omega,x)&=0,~ x\in \mathcal{K}\ ,\\
B\hat{u}&=\hat{g} \quad \text{on }  \Gamma \ .
\end{align}
In this case a singular decomposition of the solution is known for every complex frequency $\omega$. 

Doing a separation of variables near the edge of $\mathcal{K}$, we consider the operator $\mathfrak{A}_{B}(\lambda) = (i \lambda)^2+i(n-d-2)\lambda-\Delta_{S}$ with $B=D$ for Dirichlet and $B=N$ for Neumann boundary conditions in the subset $\Xi=K \cap S^{n-d-1}$ of the sphere. Here $\Delta_S$ denotes the  Laplace operator on $S^{n-d-1}$. Denoting the eigenvalues of $\Delta_S$ in $\Xi$ by $\{\mu_{k,B}\}_{k=0}^\infty$, the eigenvalues of $\mathfrak{A}_B(\lambda)$ are given by $\lambda_{\pm k,B} = \frac{i(n-d-2)}{2}\mp i \nu_{k,B}$ with $\nu_{k,B}= \frac{((n-d-2)^2+4\mu_{k,B})^{1/2}}{2}$. The associated orthogonal eigenfunctions $\Phi_{k,B}$ of the angular variables $\theta$ are normalized as $\|\Phi_{k,B}\|^2_{L^2(\Xi)} = \nu_{k,B}^{-1}$.

For $d=1$, $n=3$, the nonzero eigenvalues $\lambda_{\pm k,B}= \mp \frac{k\pi}{\alpha}$ are simple if $\frac{k\pi}{\alpha} \not \in \mathbb{N}$, and have multiplicity $2$ otherwise. For $k>0$ $\Phi_{k,N}(\theta) = (k\pi)^{-\frac{1}{2}} \cos(k\pi \theta/\alpha)$, $\Phi_{k,D}(\theta) = (k\pi)^{-\frac{1}{2}} \sin(k\pi \theta/\alpha)$. For Neumann boundary conditions, the eigenvalue $\lambda_{0,N}=0$ has multiplicity $2$. Here, $\alpha$ denotes the opening angle of $K\subset\mathbb{R}^2$. 

We recover a screen with flat boundary as $\alpha$ tends to $2\pi^{-}$, and the discussion can be adapted to circular edges as in \cite{petersdorff3}. In this case $\lambda_{\pm k,B} = \mp \frac{k\pi}{\alpha}$.

The asymptotic expansion   involves special solutions of the Dirichlet or Neumann problem in $K$, see \cite[(3.5)]{kokotov3}, respectively \cite[(4.4)]{kokotov}: 
$$w_{-k,B}(y, \omega, \zeta) = \frac{2^{1-\nu_{k,B}}}{\Gamma(\nu_{k,B})}(i|y|\sqrt{-|\zeta|^2+\omega^2})^{\nu_{k,B}} K_{\nu_{k,B}}(i|y|\sqrt{-|\zeta|^2+\omega^2}) |y|^{i \lambda_{{-k,B}}} \Phi_{k,B}(y/|y|) \ .$$
Here $K_\nu$ is the modified Bessel function of the third kind.

One then transforms back into the time domain. Explicit formulas for the inverse Fourier transform $\mathcal{F}^{-1}_{{(\omega, \zeta)} \to (t,z)} w_{-k,B}(y, \overline{\omega}, \zeta)$ can be found in Lemma 8.1 of \cite{kokotov}.

The main theorem for the inhomogeneous wave equation involves an expansion in terms of singular functions. We refer to \cite[Theorem 7.4 and Remark 7.5]{kokotov} for the details in the case of the Neumann problem in a wedge, respectively \cite[Theorem 4.1]{kokotov3} for the Dirichlet problem in a cone.
\begin{theorem}\label{maintheorem}
Let $\beta\leq 1$ and assume that the line $\mathrm{Im} \ \lambda = \beta-1+\frac{n-d-2}{2}$ does not intersect the spectrum of $\mathfrak{A}_B$. Further, define  $$J_{\beta,{B}} = \left\{{j}: \frac{n-d-2}{2}> \mathrm{Im}\ \lambda_{{j,B}} > \beta-1+\frac{n-d-2}{2}\right\}\ ,$$ if $n-d>2$, and $$J_{\beta,{B}} = \left\{{j}: 0> \mathrm{Im}\ \lambda_{{j},B} > \beta-1\right\}\cup A\ ,$$ with $A = \{0\}$ for $\beta\leq 0$ and $A=\emptyset$ otherwise.\\
 If $u$ is a strong solution to the inhomogeneous wave equation with {right hand side $f$ and} homogeneous Dirichlet or Neumann boundary conditions ($B=D$, resp.~$N$) in $\mathcal{K}$ near $\{0\}\times \mathbb{R}^d$, then $u$ is of the form
$$\sum_{j \in J_{\beta,B}} \Gamma(1+\nu_{j,B}) |y|^{i \lambda_{j,B}} \Phi_{j,B}(\theta)\sum_{m=0}^{N_j}\frac{(\partial_t^2-\Delta_z)^m(i|y|)^{2m}}{2^{2m}m! \Gamma(m+\nu_{j,B}+1)} {\mathcal{F}^{-1}_{(\omega, \zeta) \to (t,z)}{c}_{j,B}} + \check{v}(y,t,z)\ ,$$
assuming that $i \lambda_{j,B} \not \in \mathbb{N}$. Here $N_j$ {is} sufficiently large, and $c_{j,B}(\omega, \zeta) = \langle \hat{f}(\cdot, \omega, \zeta), w_{-j,B}(\cdot, \overline{\omega}, \zeta)\rangle_{L^2(K)}$; its regularity is determined by the right hand side. 
The remainder $\check{v}$ is less singular, in the sense that  $\|\check{v}\|_{DV_{\beta, q}(K \times \mathbb{R}; \gamma)} \lesssim \|f\|_{RH_{\beta, q}(K \times \mathbb{R}, \gamma)}$, {$\gamma>0$, $q \in \mathbb{N}_0$}. We refer to \cite{kokotov} for the definition of the weighted spaces $DV_\beta(K \times \mathbb{R}, \gamma), RH_{\beta, q}(K \times \mathbb{R}, \gamma)$, {$\gamma>0$, $q \in \mathbb{N}_0$}. \\
If $i \lambda_{j,B} \in \mathbb{N}$ additional terms $|y|^{i \lambda_{j,B}} \log(|y|)$ appear.
\end{theorem}
Further information can be obtained from the singular functions $W_{-{j},B}(y,t,z) = \mathcal{F}^{-1}_{(\zeta,\omega) \to (t,z)}w_{-{j},B}$, using the convolution representation
$${\mathcal{F}^{-1}_{(\omega, \zeta) \to (t,z)}{c}_{j,B}} = \int_{\mathbb{R}^d}dz_1\int_{\mathbb{R}} dt_1 \int_K dy f(y,z_1,t_1) W_{-j,B}(y, {t-t_1, z-z_1})$$ 
of the asymptotic expansion in Theorem \ref{maintheorem}. 
Because the singular support of $W_{-j,B}$ lies on the lightcone $\{(y,t,z) \in \mathbb{R}^{n+1} : t = \sqrt{|y|^2+|z|^2}\}$ emanating from the edge, we note that {$\mathcal{F}^{-1}_{(\omega, \zeta) \to (t,z)}{c}_{j,B}$} is smooth in $$\{(t,z)\in\mathbb{R}^{d+1} : t>\sup\{t_1 + \sqrt{|y|^2+|z-z_1|^2}  : (y,z_1,t_1) \in \mathrm{singsupp}\ f\}\}\ .$$ In particular, if $f$ is smooth, $\mathrm{singsupp}\ f = \emptyset$ and {$\mathcal{F}^{-1}_{(\omega, \zeta) \to (t,z)}{c}_{j,B}$} is smooth everywhere.

Theorem \ref{maintheorem} can be translated into a result for inhomogeneous  boundary conditions, as for elliptic problems \cite[Section 5]{petersdorff2}. If $B u = g$ on ${\mathbb{R}^+_t \times} \partial \mathcal{K}$, choose a function $\widetilde{g}$ in ${\mathbb{R}^+_t \times}\mathcal{K}$ such that $B \widetilde{g} = g$ on ${\mathbb{R}^+_t \times}\partial \mathcal{K}$. The function $U = u-\widetilde{g}$ satisfies homogeneous boundary conditions $BU=0$, and $\partial_t^2 U - \Delta U = {f} -\partial_t^2 \widetilde{g} + \Delta \widetilde{g}$. According to Theorem \ref{maintheorem}, $U$ admits an asymptotic expansion, and therefore so does $u = U + \widetilde{g}$. 

For the analysis of the solutions to the boundary integral formulations of the wave equation, the resulting asymptotic expansions of the boundary values $u|_\Gamma$ and $\partial_\nu u|_\Gamma$ will be crucial. They are directly obtained from the expansion in the interior. In particular, for $i \lambda_{j,B} \not \in \mathbb{N}$ the singularities of  $u|_\Gamma$ are proportional to $|y|^{i \lambda_{j,B}+2m}$, and the singularities of  $\partial_\nu u|_\Gamma$ are proportional to $|y|^{i \lambda_{j,B}+2m-1}$. When $i \lambda_{j,B}\in \mathbb{N}$,  additional terms $|y|^{i \lambda_{j,B}+2m} \log(|y|)$, respectively $|y|^{i \lambda_{j,B}+2m-1} \log(|y|)$ appear.

\subsection{Singularities for circular screens and approximation}

We first illustrate the above results for the exterior of a circular wedge with exterior opening angle $\alpha$. For $\alpha \to 2\pi^-$, the wedge degenerates into the circular screen $\{(x_1,x_2,0) \in \mathbb{R}^3 :  {x_1^2+x_2^2}\leq 1\}$. Near the edge $\{(x_1,x_2,0) \in \mathbb{R}^3 :  {x_1^2+x_2^2}= 1\}$ we use the coordinates $(y,z,\theta)$, where in polar coordinates in the $x_1-x_2$-plane $y=r-1$, $z=\theta$. Using \cite{petersdorff3}, an analogous expansion to Theorem \ref{maintheorem} also holds in this curved geometry, with the same leading singular term $|y|^{i\lambda}$, where $\lambda\to -\frac{i}{2}$ as $\alpha \to 2\pi^-$:
\begin{align}
u(y,t,z)|_\Gamma &= a(t,z) |y|^{1/2} + \check{v}(y,t,z)\  , \label{decompostionEdget}\\
\partial_\nu u(y,t,z)|_\Gamma &= b(t,z) |y|^{-\frac{1}{2}} + {\tilde{v}}(y,z, t)\ . \label{decompositionEdge}
 \end{align}
{Here $a$ and $b$ are smooth for smooth data.}
 
From these decompositions we obtain optimal approximation properties on the graded mesh. Here we show how the analysis performed by T.~von Petersdorff in \cite{disspetersdorff} may be extended to the hyperbolic case. The results are derived for the h-version on graded meshes and contain automatically the case of a quasi-uniform mesh by setting the grading parameter $\beta=1$. \\

\begin{theorem}\label{approxtheorem1} {Let $\varepsilon>0$.}
a) Let $u$ be a strong solution to the homogeneous wave equation with inhomogeneous Neumann boundary conditions $\partial_\nu u|_\Gamma = g$, with $g$ smooth. Further, let $\phi_{h,\Delta t}^\beta$ be the best approximation {in the norm of ${H}^{r}_\sigma(\R^+, \widetilde{H}^{\frac{1}{2}-s}(\Gamma))$ to the Dirichlet trace $u|_\Gamma$ in $\widetilde{V}^{p,1}_{\Delta t,h}$} on a $\beta$-graded spatial mesh with $\Delta t \lesssim h^\beta$. Then $\|u-\phi_{h, \Delta t}^\beta\|_{r,\frac{1}{2}-s, \Gamma, \ast} \leq C_{\beta,\varepsilon} h^{\min\{\beta(\frac{1}{2}+s), \frac{3}{2}+s\}{-\varepsilon}}$, where $s \in [0,\frac{1}{2}]$ and {$r \in [0,p)$}.

b) Let $u$ be a strong solution to the homogeneous wave equation with inhomogeneous Dirichlet boundary conditions $u|_\Gamma = g$, with $g$ smooth. Further, let $\psi_{h,\Delta t}^\beta$ be the best approximation  {in the norm of ${H}^{r}_\sigma(\R^+, \widetilde{H}^{-\frac{1}{2}}(\Gamma))$ to the Neumann trace $\partial_\nu u|_\Gamma$ in ${V}^{p,0}_{\Delta t,h}$} on a $\beta$-graded spatial mesh  with $\Delta t \lesssim h^\beta$. Then $\|\partial_\nu u-\psi_{h, \Delta t}^\beta\|_{r,-\frac{1}{2}, \Gamma, \ast} \leq C_{\beta,\varepsilon} h^{\min\{\frac{\beta}{2}, \frac{3}{2}\}{-\varepsilon}}$, where  {$r \in [0,p+1)$}.
\end{theorem}

Theorem \ref{approxtheorem1} implies a corresponding result for the solutions of the single layer and hypersingular integral equations on the screen:
\begin{corollary}\label{approxcor1} Let $\varepsilon>0$.
a) Let $\phi$ be the solution to the hypersingular integral equation \eqref{hypersingeq} and  $\phi_{h,\Delta t}^\beta$ the best approximation {in the norm of ${H}^{r}_\sigma(\R^+, \widetilde{H}^{\frac{1}{2}-s}(\Gamma))$ to $\phi$ in $\widetilde{V}^{p,1}_{\Delta t,h}$} on a $\beta$-graded spatial mesh  with $\Delta t \lesssim h^\beta$. Then $\|\phi-\phi_{h, \Delta t}^\beta\|_{r,\frac{1}{2}-s, \Gamma, \ast} \leq C_{\beta,\varepsilon} h^{\min\{\beta(\frac{1}{2}+s), \frac{3}{2}+s\}{-\varepsilon}}$, where $s \in [0,\frac{1}{2}]$ and {$r \in [0,p)$}.

b) Let $\psi$ be the solution to the single layer integral equation \eqref{dirproblemV} and  $\psi_{h,\Delta t}^\beta$ the best approximation  {in the norm of ${H}^{r}_\sigma(\R^+, \widetilde{H}^{-\frac{1}{2}}(\Gamma))$ to $\psi$ in ${V}^{p,0}_{\Delta t,h}$} on a $\beta$-graded spatial mesh  with $\Delta t \lesssim h^\beta$. Then $\|\psi-\psi_{h, \Delta t}^\beta\|_{r,-\frac{1}{2}, \Gamma, \ast} \leq C_{\beta,\varepsilon} h^{\min\{\frac{\beta}{2}, \frac{3}{2}\}{-\varepsilon}}$, where {$r \in [0,p+1)$}.
\end{corollary}
Indeed, on the flat screen the solutions to the integral equations are given by $\phi = \left[u\right]|_\Gamma$ {in terms of the solution $u$ which satisfies Neumann conditions $Bu = \partial_\nu u|_\Gamma = g$, respectively} $\psi = \left[\partial_\nu u\right]|_\Gamma$ {in terms of the solution $u$ which satisfies Dirichlet conditions $Bu = u|_\Gamma=f$}.\\

The proof of Theorem \ref{approxtheorem1} relies on the auxiliary results in Section \ref{discretization}. We first consider the approximation of the Neumann trace.
\begin{theorem} 
Under the assumptions of Theorem \ref{approxtheorem1}, there holds
$\|\partial_\nu u - \Pi_x^{{0}} \Pi_t^{{p}} \partial_\nu u\|_{r, -\frac{1}{2}, \Gamma, \ast} \lesssim  h^{\min\{\beta/2, \frac{3}{2}\}{-\varepsilon}}$.
\end{theorem}
As before, our results extend from rectangular to triangular elements as in reference \cite{disspetersdorff}.
\begin{proof}{Using the decomposition \eqref{decompositionEdge} for $ \partial_{\nu} u$, we can separate the singular and regular parts on the rectangular mesh:
\begin{align*}&\|\partial_\nu u - \Pi_x^{{0}} \Pi_t^{{p}} \partial_\nu u\|_{r, -\frac{1}{2}, \Gamma, \ast} \leq \|b(t,z) |y|^{-\frac{1}{2}} - \Pi_t^{{p}} \Pi_x^{{0}} b(t,z) |y|^{-\frac{1}{2}}\|_{r, -\frac{1}{2}, \Gamma,\ast}  + \|\tilde{v} - \Pi_t^{{p}} \Pi_x^{{0}}{\tilde{v}}\|_{r, -\frac{1}{2}, \Gamma,\ast}\\ & \leq \|b(t,z) |y|^{-\frac{1}{2}} -\Pi_t^{{p}} b(t,z) |y|^{-\frac{1}{2}}\|_{r, -\frac{1}{2}, \Gamma,\ast}+\|\Pi_t^{{p}} b(t,z) |y|^{-\frac{1}{2}}- \Pi_t^{{p}} \Pi_x^{{0}} b(t,z) |y|^{-\frac{1}{2}}\|_{r, -\frac{1}{2}, \Gamma,\ast} \\ & \qquad + \|\tilde{v} - \Pi_t^{{p}} \Pi_x^{{0}}{\tilde{v}}\|_{r, -\frac{1}{2}, \Gamma,\ast}\\ &
\leq \|b(t,z) -\Pi_t^{{p}} b(t,z)\|_{r, \epsilon-\frac{1}{2}} \||y|^{-\frac{1}{2}}\|_{\widetilde{H}^{-\varepsilon}(I)}  +  \|\Pi_t^{{p}} b(t,z) |y|^{-\frac{1}{2}}- \Pi_t^{{p}} \Pi_z^{{0}} b(t,z) |y|^{-\frac{1}{2}}\|_{r, -\frac{1}{2}, \Gamma,\ast} \\ & \qquad + \| \Pi_t^{{p}} \Pi_z^{{0}} b(t,z) |y|^{-\frac{1}{2}}-\Pi_t^{{p}} \Pi_z^{{0}} b(t,z) \Pi_y^{{0}} |y|^{-\frac{1}{2}}\|_{r, -\frac{1}{2}, \Gamma,\ast}+
\|\tilde{v} - \Pi_t^{{p}} \Pi_x^{{0}}{\tilde{v}}\|_{r, -\frac{1}{2}, \Gamma,\ast}\ .
\end{align*}
Here, for the first term we have used Lemma \ref{lemma3.3}, and for the second $\Pi_x^{{0}} =  \Pi_z^{{0}} \Pi_y^{{0}}$. We note that the first term is bounded by $$\|b(t,z) -\Pi_t^{{p}} b(t,z)\|_{r, \epsilon-\frac{1}{2}} \lesssim (\Delta t)^{p+1-r} \max\{h_1, \Delta t\}^{\frac{1}{2}- \epsilon} \|b(t,z)\|_{p+1,0}\ .$$
The second and third terms we obtain with Lemma \ref{lemma3.3}:
\begin{align*}&
 \|\Pi_t^{{p}} b(t,z) |y|^{-\frac{1}{2}}- \Pi_t^{{p}} \Pi_z^{{0}} b(t,z) |y|^{-\frac{1}{2}}\|_{r, -\frac{1}{2}, \Gamma,\ast} + \| \Pi_t^{{p}} \Pi_z^{{0}} b(t,z) |y|^{-\frac{1}{2}}-\Pi_t^{{p}} \Pi_z^{{0}} b(t,z) \Pi_y^{{0}} |y|^{-\frac{1}{2}}\|_{r, -\frac{1}{2}, \Gamma,\ast}\\ & \lesssim \|\Pi_t^{{p}} b(t,z) - \Pi_t^{{p}} \Pi_z^{{0}} b(t,z)\|_{r, \varepsilon - \frac{1}{2} } \||y|^{-\frac{1}{2}}\|_{\widetilde{H}^{-\varepsilon}(I)}  +\| \Pi_t^{{p}} \Pi_z^{{0}} b(t,z)\|_{r,0} \||y|^{-\frac{1}{2}} - \Pi_y^{{0}}|y|^{-\frac{1}{2}}\|_{\widetilde{H}^{-\frac{1}{2}}(I) } \ .
\end{align*}}
From Lemma \ref{keylemmagrad} we have $\||y|^{-\frac{1}{2}} - \Pi_y^{{0}}|y|^{-\frac{1}{2}}\|_{\widetilde{H}^{-\frac{1}{2}}(I) } \lesssim h^{\min\{\frac{\beta}{2}, \frac{3}{2}\}-\varepsilon}$ and $\|\Pi_t^{{p}} b(t,z) - \Pi_t^{{p}} \Pi_z^{{0}} b(t,z)\|_{r, \varepsilon - \frac{1}{2} } \lesssim h^{3/2} \|\Pi_t^{{p}} b\|_{r, 1+\varepsilon}$. 

{After possibly expanding finitely many terms, which may be treated as above, we may assume that the regular part $\tilde{v}$ in \eqref{decompositionEdge} is $H^1$ in space. Localizing in space and time to the space-time elements $(t_j, t_{j+1}]\times R_{kl}$, as in Figure \ref{fig:A1},} 
{$$\|\tilde{v} - \Pi_t^{{p}} \Pi_x^{{0}}{\tilde{v}}\|_{r, -\frac{1}{2}, \Gamma,\ast} \lesssim \sum_j \sum_{k,l} \|\tilde{v} - \Pi_t^{{p}} \Pi_x^{{0}}{\tilde{v}}\|_{r, -\frac{1}{2}, (t_j, t_{j+1}]\times R_{kl},\ast}$$
and using Lemma \ref{lemma3.4} for $\tilde{v}$ and Lemma \ref{lemma3.2}, 
\begin{align*}\|\tilde{v} - \Pi_x^{{0}}\Pi_t^{{p}} {\tilde{v}}\|_{r, -\frac{1}{2},(t_j, t_{j+1}]\times R_{kl},\ast} &\lesssim_\sigma   (\Delta t)^{p+1-r} \max\{h_1,h_2, \Delta t\}^{1/2}\|\partial_t^{p+1} \tilde{v}\|_{L^2([t_j,t_{j+1}]\times R_{kl})} \\ & +  \max\{h_1,h_2, \Delta t\}^{\frac{1}{2}}\left(h_1 \| \tilde{v}_x\|_{L^2([t_j,t_{j+1}]\times R_{kl})}  + h_2 \| \tilde{v}_y\|_{L^2([t_j,t_{j+1}]\times R_{kl})} \right)\ .
\end{align*}}
{By summing over all rectangles $R_{kl}$ of the mesh of the screen and noting the exponential weight $e^{-2\sigma t}$, we conclude that for $\Delta t \lesssim \min\{h_1,h_2\}$ we have $\|\partial_\nu u - \Pi_x \Pi_t \partial_\nu u\|_{r, -\frac{1}{2}, \Gamma, \ast} \lesssim  h^{\min\{\beta/2, \frac{3}{2}\}-\varepsilon}$.}
\end{proof}

\subsubsection{Approximation of the trace}
We now consider the approximation of the solution $u$ to the wave {e}quation on the screen, with expansion \eqref{decompostionEdget}, or equivalently the solution to the hypersingular integral equation. Apart from the energy norm, here the $L^2$-norm is of interest, and we state the result for general Sobolev indices:  
\begin{theorem}
For {$r \in [0,p)$} and $s\in [0, \frac{1}{2}]$ there holds
$\|u - \Pi_x^{{1}} \Pi_t^{{p}} u\|_{r,\frac{1}{2}-s, \Gamma, \ast} \lesssim  h^{\min\{\beta (\frac{1}{2}+s), \frac{3}{2}+s\}{-\varepsilon}}$.
\end{theorem}
\begin{proof}
Similarly to above, one estimates on every rectangle $R$ of the mesh:
\begin{align*}
\|\Pi_t^{{p}}  u - \Pi_x^{{1}} \Pi_t^{{p}}  u\|_{r, \frac{1}{2}, (t_k, t_{k+1}]\times R,\ast} &\leq \| \Pi_t^{{p}} a(t,z) |y|^{1/2} - \Pi_t^{{p}} \Pi_x^{{1}} a(t,z) |y|^{1/2}\|_{r, \frac{1}{2}, (t_k, t_{k+1}]\times R,\ast} \\ & \qquad + \|\Pi_t^{{p}} {\check{v}} - \Pi_x^{{1}}\Pi_t^{{p}} {\check{v}}\|_{r, \frac{1}{2}, (t_k, t_{k+1}]\times R,\ast}\ .
\end{align*}
For the first term we note with Lemma \ref{lemma3.3a}:
\begin{align*}
&\| \Pi_t^{{p}} a(t,z) |y|^{1/2} - \Pi_t^{{p}} \Pi_x^{{1}} a(t,z) |y|^{1/2}\|_{r, \frac{1}{2}, (t_k, t_{k+1}]\times R,\ast}\\
&\leq \| \Pi_t^{{p}} a(t,z) |y|^{1/2}- \Pi_t^{{p}} \Pi_z^{{1}} a(t,z) |y|^{\frac{1}{2}} + \Pi_t^{{p}} \Pi_z^{{1}} a(t,z) |y|^{\frac{1}{2}} - \Pi_t^{{p}} \Pi_z^{{1}} a(t,z) \Pi_y^{{1}} |y|^{1/2}\|_{r, \frac{1}{2}, (t_k, t_{k+1}]\times R,\ast}\\
&\leq \|\Pi_t^{{p}} a(t,z) - \Pi_t^{{p}} \Pi_z^{{1}} a(t,z)\|_{r,  \frac{1}{2},(t_k, t_{k+1}]\times I,\ast} \||y|^{\frac{1}{2}}\|_{ \widetilde{H}^{\frac{1}{2}}(I)  }  \\ & \qquad +\| \Pi_t^{{p}} \Pi_z^{{1}} a(t,z)\|_{r, \frac{1}{2}, (t_k, t_{k+1}]\times I, \ast} \||y|^{1/2} - \Pi_y^{{1}}|y|^{\frac{1}{2}}\|_{\widetilde{H}^{\frac{1}{2}}(I) } \ .
\end{align*}
Now note that 
$$\|\Pi_t^{{p}} a(t,z) - \Pi_t^{{p}} \Pi_z^{{1}} a(t,z)\|_{r, \frac{1}{2}, (t_k, t_{k+1}]\times I,\ast} \leq C \|\Pi_t^{{p}} a(t,z)\|_{r, 2, (t_k, t_{k+1}]\times I} h^{\frac{3}{2}}$$
and, from Lemma \ref{keylemmagrad2}, 
$$\||y|^{1/2} - \Pi_y^{{1}}|y|^{\frac{1}{2}}\|_{\widetilde{H}^{\frac{1}{2}}(I) } \lesssim h^{\min\{\frac{\beta}{2}, \frac{3}{2}\}{-\varepsilon}}\ .$$
After possibly expanding finitely many terms, which may be treated as above, we may assume that the regular part $\check{v}$ in \eqref{decompostionEdget} is in $H^{3}$ in space. 

To approximate the regular part $\check{v}$, we let {$U$} denote the interpolant of $\check{v}$  in space and time on the graded mesh and use Lemma \ref{TobiasTRACE}. 
On $Q:=[0,1] \times [0,1]$, decomposed into rectangles $R_{jk}:=[x_{j-1},x_{j}] \times [y_{k-1},y_{k}]$ with side length $ h_{j},h_{k}$,
\begin{align*}
	\| \check{v} -{U}\|^2_{r,0,Q} &= \sum\limits_{l}\sum\limits_{j,k=1}^N\| \check{v}-{U}\|_{r, 0, [t_l, t_{l+1})\times R_{jk}}^{2} \\ &\lesssim \sum\limits_{l} \sum\limits_{j,k=1}^N \Big(\max\{h_j, \Delta t\}^4 \| \check{v}_{xx}\|^2_{r, 0, [t_l, t_{l+1})\times R_{jk}} + \max\{h_k, \Delta t\}^4\| \check{v}_{yy}\|^2_{r, 0, [t_l, t_{l+1})\times R_{jk}} \\ &\qquad+ (\max\{h_j, \Delta t\}^4+\max\{h_k, \Delta t\}^4)\| \check{v}_{tt}\|^2_{r, 0, [t_l, t_{l+1})\times R_{jk}}  \\ &\qquad+  \max\{h_j, \Delta t\}^{4}\max\{h_k, \Delta t\}^2 \| \check{v}_{xxy}\|_{r, 0, [t_l, t_{l+1})^{{2}}\times R_{jk}} \Big) \\&
	\lesssim \max\{h, \Delta t\}^4\| \check{v}\|^2_{r, 3, Q}
\end{align*}
and
\begin{align*}
	\| \check{v} -{U}\|^2_{r, 1, Q} &= \sum\limits_{l} \sum\limits_{j,k=1}^N \| \check{v}-{U}\|^2_{r, 1, [t_l, t_{l+1})\times R_{jk}} \\&
	\lesssim \sum\limits_{l}\sum\limits_{j,k=1}^N\Big(\max\{h_j,\Delta t\}^2\| \check{v}_{xx}\|^2_{r, 0, [t_l, t_{l+1})\times R_{jk}}+ \max\{h_k,\Delta t\}^2\| \check{v}_{yy}\|^2_{r, 0, [t_l, t_{l+1})\times R_{jk}}\\ & \qquad+\max\{h_j,\Delta t\}^2\| \check{v}_{xt}\|^2_{r, 0, [t_l, t_{l+1})\times R_{jk}}+\max\{h_k,\Delta t\}^{4}\| \check{v}_{xxy}\|^2_{r, 0, [t_l, t_{l+1})\times R_{jk}}\\ & \qquad +\max\{h_k,\Delta t\}^{{2}}\| \check{v}_{xyy}\|^2_{r, 0, [t_l, t_{l+1})\times R_{jk}} \Big)\\ &
	\lesssim\max\{h, \Delta t\}^2\| \check{v}\|^2_{r, 3,Q} \ .
\end{align*}
Here we have used $h_{k} \leq \beta \ h$ {and used the restriction $\|\ \cdot\ \|_{r, 0, [t_l, t_{l+1})\times R_{jk}}$ of the $H^r_\sigma(\mathbb{R}^+, H^0(R_{jk}))$ to the time interval $[t_l, t_{l+1})$.}
Interpolation yields $\| \check{v} -{U}\|_{r,\frac{1}{2},Q, \ast}  \lesssim \max\{h, \Delta t\}^{\frac{3}{2}-\varepsilon}\| \check{v}\|_{r, 3,Q}$.
\end{proof}
The approximation argument extends from rectangular to triangular elements as in \cite{disspetersdorff}.

\subsection{Singularities for polygonal screens and approximation}

We consider the singular expansion of the solution to the wave equation \eqref{eq:oriProblem2} with Dirichlet or Neumann boundary conditions on a polygonal screen $\Gamma$. Additional singularities now arise from the corners of the screen. For simplicity, we restrict ourselves to the model case of a square screen $\Gamma=(0,1)\times(0,1)\times\{0\} \in \R^3$. In this geometry, for elliptic problems asymptotic expansions and their implications for the numerical approximation are discussed in \cite{screenMaischak, petersdorff}.

The following result gives  a decomposition of  the solution to the Helmholtz equation  and its normal derivative on $\Gamma$ near the vertex $(0,0)$, in terms of polar coordinates $(r,\theta)$ centered at this point. Note that we have two boundary values, $\hat{u}_\pm$, from the upper and lower sides of the screen. 

\begin{theorem}\label{decompS}
For fixed $\omega \neq 0$ with $\mathrm{Im}\ \omega \geq 0$, let $\hat{u}_\omega$ be the solution to the Helmholtz equation 
\begin{align}\label{helmholtzneu} \nonumber
\omega^2\hat{u}(\omega,x)- \Delta \hat{u}(\omega,x)=0,~ x\in \mathbb{R}^n \setminus \Gamma\ ,
\\ B\hat{u}(\omega,x)=\hat{g}(\omega,x),~x\in \Gamma\ ,
\end{align}
{where $\hat{g}$ is sufficiently smooth.}
a) Assume $Bu = \partial_\nu u|_\Gamma$. If $\hat{g} \in  H^1(\Gamma)$, then
\begin{align}\label{decompositionS}
\hat{u}(\omega, x)|_+ &= \chi(r)r^{\gamma} \alpha_\omega(\theta) +\tilde{\chi}(\theta)b_{1,\omega}(r)(\sin(\theta))^{\frac{1}{2}} \\& \nonumber \qquad + \tilde{\chi}(\frac{\pi}{2}-\theta)b_{2, \omega}(r)(\cos(\theta))^{\frac{1}{2}}+\hat{u}_{0,\omega}(r, \theta)\ ,
\end{align}
where for all $\epsilon>0$ we have $\hat{u}_{0,\omega}\in \widetilde{H}^{2-\epsilon}(\Gamma)$, $\alpha_\omega\in H^{2-\epsilon}[0,\frac{\pi}{2}],\ b_{i,\omega}=c_{i,\omega, 1}r^{\gamma-\frac{1}{2}}+ c_{i,\omega,2}r^{\lambda-\frac{1}{2}}+d_{i,\omega}(r)$, $d_{i,\omega}(r)\in H^{\frac{3}{2}-\varepsilon}(\R^+)$ with $r^{\frac{3}{2}-\varepsilon} d_{i,\omega}(r) \in L^2(\mathbb{R^+})$, $c_{i,\omega,j}\in \R$. Here $\chi, \tilde{\chi} \in C^\infty_c$ are cut-off functions, $\chi, \tilde{\chi}= 1$ in a neighborhood of $0$.\\
b) Assume $Bu = u|_\Gamma$. If $\hat{g} \in H^2(\Gamma)$, then
\begin{align*}
\partial_\nu \hat{u}(\omega,x)|_+&= \chi(r)r^{\gamma-1} \alpha_\omega(\theta) +\tilde{\chi}(\theta)b_{1,\omega}(r)r^{-1}(\sin(\theta))^{-\frac{1}{2}}\\ & \qquad + \tilde{\chi}(\frac{\pi}{2}-\theta)b_{2, \omega}(r)r^{-1}(\cos(\theta))^{-\frac{1}{2}} + \hat{\psi}_{0,\omega}(r, \theta)\ ,
\end{align*}
where for all $\epsilon>0$ we have $\hat{\psi}_{0,\omega}\in H^{1-\epsilon}(\Gamma)$, $\alpha_\omega\in H^{1-\epsilon}[0,\frac{\pi}{2}],~b_{i,\omega}=c_{i,\omega}r^\gamma+d_{i,\omega}(r)$, $r^{-\frac{1}{2}}d_{i,\omega}(r)\in H^1(\R^+),~r^{-\frac{3}{2}}d_{i,\omega}(r)\in L_2(\R^+),~ c_{i,\omega}\in \R$. Here $\chi, ~\tilde{\chi} \in C^\infty_c$ are cut-off functions, $\chi, \tilde{\chi}= 1$ in a neighborhood of $0$. 
\end{theorem}
In fact, if $\hat{g}$ is a Schwartz function of $\omega$, the decomposition depends smoothly on this variable. For the square screen $\gamma \approx 0.2966$ and $\lambda \approx 1.426$ are determined by the lowest eigenvalues of the operator $\mathfrak{A}_B$ on $S^2 \setminus (\mathbb{R}_+^2 \times \{0\})$.
For the proof of Theorem \ref{decompS}, see \cite{heuer}, p. 108-109.

As above, in analogy with the work of Plamenevskii and coauthors, the asymptotic expansion translates into the time domain:
\begin{align}\label{decompositiont}\nonumber
{u(t,x)|_+} &=v_{0}(t,r, \theta) + \chi(r)r^{\gamma} \alpha(t,\theta) +\tilde{\chi}(\theta)b_{1}(t,r)(\sin(\theta))^{\frac{1}{2}}\\& \qquad + \tilde{\chi}(\textstyle{\frac{\pi}{2}}-\theta)b_{2}(t,r)(\cos(\theta))^{\frac{1}{2}} \ ,\\
\nonumber
{\partial_\nu u(t,x)|_+} &=\psi_{0}(t,r, \theta) + \chi(r)r^{\gamma-1} \alpha(t,\theta) +\tilde{\chi}(\theta)b_{1}(t,r)r^{-1}(\sin(\theta))^{-\frac{1}{2}}\\& \qquad + \tilde{\chi}(\textstyle{\frac{\pi}{2}}-\theta)b_{2}(t,r)r^{-1}(\cos(\theta))^{-\frac{1}{2}} \ . \label{decomposition}
\end{align}
To control the remainder terms in these formal computations requires elliptic a priori weighted estimates near the singularities, as discussed in \cite{matyu}.

From the decomposition, similar to Theorem \ref{approxtheorem1} we obtain  optimal approximation properties on the graded mesh, where the error is dominated by the edge singularities, not the corners. The beta needs to be chosen large enough, depending on the singular exponent $\gamma$ in \eqref{decompositiont}, \eqref{decomposition}. See \cite{disspetersdorff, petersdorff} for similar results in the time-independent case. 

\begin{theorem}\label{approxtheorem2} Let $\varepsilon>0$.
a) Let $u$ be a strong solution to the homogeneous wave equation with inhomogeneous Neumann boundary conditions $\partial_\nu u|_\Gamma = g$, with $g$ smooth. Further, let $\phi_{h,\Delta t}^\beta$ be the best approximation {in the norm of ${H}^{r}_\sigma(\R^+, \widetilde{H}^{\frac{1}{2}-s}(\Gamma))$ to the Dirichlet trace $u|_\Gamma$ in $\widetilde{V}^{p,1}_{\Delta t,h}$} on a $\beta$-graded spatial mesh with $\Delta t \lesssim h^\beta$ and $\beta \geq \frac{3}{2(\gamma + \frac{1}{2})}$. Then $\|u-\phi_{h, \Delta t}^\beta\|_{r,\frac{1}{2}-s, \Gamma, \ast} \leq C_{\beta,\varepsilon} h^{\min\{\frac{\beta}{2}, \frac{3}{2}\}+s{-\varepsilon}}$, where $s \in [0,\frac{1}{2}]$ and {$r \in [0,p)$}. 

b) Let $u$ be a strong solution to the homogeneous wave equation with inhomogeneous Dirichlet boundary conditions $u|_\Gamma = g$, with $g$ smooth. Further,  let $\psi_{h,\Delta t}^\beta$ be the best approximation  {in the norm of ${H}^{r}_\sigma(\R^+, \widetilde{H}^{-\frac{1}{2}}(\Gamma))$ to the Neumann trace $\partial_\nu u|_\Gamma$ in ${V}^{p,0}_{\Delta t,h}$} on a $\beta$-graded spatial mesh  with $\Delta t \lesssim h^\beta$ and $\beta \geq \frac{3}{2(\gamma + \frac{1}{2})}$. Then $\|\partial_\nu u-\psi_{h, \Delta t}^\beta\|_{r,-\frac{1}{2}, \Gamma, \ast} \leq C_{\beta,\varepsilon} h^{\min\{\frac{\beta}{2}, \frac{3}{2}\}{-\varepsilon}}$, where {$r \in [0,p+1)$}.  
\end{theorem}

The theorem again implies a corresponding result for the solutions of the single layer and hypersingular integral equations on the screen:
\begin{corollary}\label{approxcor2} Let $\varepsilon>0$.
a) Let $\phi$ be the solution to the hypersingular integral equation \eqref{hypersingeq} and  $\phi_{h,\Delta t}^\beta$ the best approximation {in the norm of ${H}^{r}_\sigma(\R^+, \widetilde{H}^{\frac{1}{2}-s}(\Gamma))$ to $\phi$ in $\widetilde{V}^{p,1}_{\Delta t,h}$} on a $\beta$-graded spatial mesh  with $\Delta t \lesssim h^\beta$ and  and $\beta \geq \frac{3}{2(\gamma + \frac{1}{2})}$. Then $\|\phi-\phi_{h, \Delta t}^\beta\|_{r,\frac{1}{2}-s, \Gamma, \ast} \leq C_{\beta,\varepsilon} h^{\min\{\beta(\frac{1}{2}+s), \frac{3}{2}+s\}{-\varepsilon}}$, where $s \in [0,\frac{1}{2}]$ and {$r \in [0,p)$}. 

b) Let $\psi$ be the solution to the single layer integral equation \eqref{dirproblemV} and  $\psi_{h,\Delta t}^\beta$ the best approximation  {in the norm of ${H}^{r}_\sigma(\R^+, \widetilde{H}^{-\frac{1}{2}}(\Gamma))$ to $\psi$ in ${V}^{p,0}_{\Delta t,h}$} on a $\beta$-graded spatial mesh  with $\Delta t \lesssim h^\beta$ and  and $\beta \geq \frac{3}{2(\gamma + \frac{1}{2})}$. Then $\|\psi-\psi_{h, \Delta t}^\beta\|_{r,-\frac{1}{2}, \Gamma, \ast} \leq C_{\beta,\varepsilon} h^{\min\{\frac{\beta}{2}, \frac{3}{2}\}{-\varepsilon}}$, where {$r \in [0,p+1)$}.  
\end{corollary}

The proof of Theorem \ref{approxtheorem2} and Corollary \ref{approxcor2} relies on arguments by von Petersdorff \cite{disspetersdorff}. We refer to this reference for a detailed analysis in the time-independent case.

We recall a key elliptic result from \cite{petersdorff}, proven there for closed polyhedral surfaces:
\begin{theorem}\label{theo:77gamma}
  Let $ \psi \in \widetilde{H}^{-\frac{1}{2}}(\Gamma) $ have a singular decomposition like the one in Theorem \ref{decompS} near  every corner of $\Gamma$. 
Then we can approximate $ \psi $ for $ \beta \geq 1$ on the graded mesh in the following way:
With $ \phi_{h} = \Pi_x^{{0}} \psi$ we have for all $ \epsilon > 0$
\begin{equation*}
   \| \psi - \phi_h \|_{\widetilde{H}^{-\frac{1}{2}}(\Gamma)} \leq C_{\beta, \varepsilon} h^{\min\{ \frac{\beta}{2}, \frac{3}{2} \}{-\varepsilon}}\ .
\end{equation*}
\end{theorem}
\begin{proof}{(of Theorem 20 b)} 
For simplicity, let $\Gamma$ be the square $Q=[0,1]^{2}$. As the approximation of the regular part $\psi_0$ and the regular edge functions of \eqref{decomposition} are already considered in the proof for the circular screen, it remains to analyze the approximation of the corner singularity and the corner edge singularity of the expansion \eqref{decomposition}. 
In the following we approximate the corner singularity:

In every space-time element we estimate
\begin{align*}
\|r^{\gamma-1} \alpha(t, \theta) - \Pi_t^{{p}} \Pi_{x,y}^{{0}} r^{\gamma-1} \alpha(t, \theta)\| & \leq \|r^{\gamma-1} \alpha(t, \theta) - \Pi_t^{{p}}  r^{\gamma-1} \alpha(t, \theta)\| \\ & \qquad + \|r^{\gamma-1} \Pi_t^{{p}} \alpha(t, \theta) -  \Pi_{x,y}^{{0}} r^{\gamma-1} \Pi_t^{{p}}\alpha(t, \theta)\| \ .
\end{align*}
$\Pi_t \alpha(t, \theta)$ is of the same form as the singular function $\alpha(\theta)$ in the elliptic case. One may therefore adapt the elliptic approximation results to $\|(1-\Pi_{x,y}) r^{\gamma-1} \Pi_t \alpha(t, \theta)\|$. This is then summed over all elements. We consider
\begin{equation*}
 \lVert r^{\gamma-1} \Pi_t^{{p}} \alpha - \Pi_{x y}^{{0}} r^{\gamma-1} \Pi_t^{{p}} \alpha \rVert = \lVert (1-\Pi_{x y}^{{0}}) r^{\gamma-1} \Pi_{t}^{{p}} \alpha(t,\theta) \rVert
\end{equation*}
Let $ \Pi_t^{{p}} \alpha(t,\theta) = \sum_{{m=0}}^{{p}} t^{{m}} \alpha_{{m}}(\theta)$ and $ f_{{m}}(x,y)=r^{\gamma-1} \alpha_{{m}}(\theta)$ on $[t_{{l}},t_{{l+1}})$. \\With $U|_{R_{kl}} =\sum_{{m=0}}^{{p}} \frac{t^{{m}}}{h_{k}h_{l}} \int\limits_{R_{kl}} f_{{m}}(x,y) dy dx $ one obtains from \eqref{3.22} 
\begin{align*}
 \lVert r^{\gamma-1} \Pi_t^{{p}} \alpha - \Pi_{x,y}^{{0}} r^{\gamma-1} \Pi_{t}^{{p}} \alpha \rVert_{r,-\frac{1}{2},Q,*}^{2} 
 &\lesssim \sum_j \sum_{k,l=1}^{{N}} \max \lbrace \Delta t , h_{k} , h_{l} \rbrace \\ &( h_{k}^{2} \lVert \partial_x (r^{\gamma-1} \Pi_t^{{p}} \alpha) \rVert_{r,0,[t_{j},t_{j+1}) \times R_{kl}}^{2} + h_{l}^{2} \lVert \partial_y (r^{\gamma-1} \Pi_t^{{p}} \alpha) \rVert_{r,0,[t_{j},t_{j+1}) \times R_{kl}}^{2} ) \\
 & + \lVert r^{\gamma-1} \Pi_t^{{p}} \alpha - \Pi_{x,y}^{{0}} r^{\gamma-1} \Pi_t^{{p}} \alpha \rVert_{r,-\frac{1}{2}, R_{11}}^{2}\ .
\end{align*}
The individual summands are estimated for different ranges of $k,l$:

\begin{figure}[ht]
\centering
\includegraphics[scale=0.7]{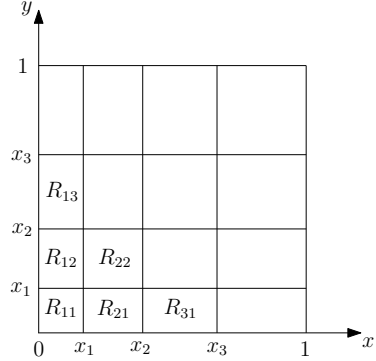}
\caption{Mesh on a square}
\label{fig:A1}
\end{figure}

Estimate for $k\geqslant 2,\; l\geqslant 2$: Note for $k\geqslant 2,\; x\in [x_{k-1},x_k]$ there holds $\vert h_k\vert \leq \beta 2^{\beta\tilde{\gamma}}h\, x^{\tilde{\gamma}}$ with $\tilde{\gamma} =1-\frac{1}{\beta}>0$. Therefore, if $ \Delta t  \leq h_{k} $ for all $k$
\begin{equation*}
	\max\{h_k, h_l, \Delta t\}h_k^2\| \partial_x (r^{\gamma-1} \Pi_t \alpha) \|_{r,0,[t_{j},t_{j+1}) \times R_{kl}}^{2} \lesssim h^{3} \| \partial_x (r^{\gamma-1} \Pi_t^{{p}} \alpha) \max \{ x^{\tilde{\gamma}} ,y^{\tilde{\gamma}} \}^{1/2} x^{\tilde{\gamma}} \|_{r,0,[t_{j},t_{j+1}) \times R_{kl}}^{2}
\end{equation*}
and
\begin{align}
 \lVert r^{\gamma-1} \Pi_t^{{p}} \alpha - \Pi_{x,y}^{{0}} r^{\gamma-1} \Pi_{t}^{{p}} \alpha\|_{r,-\frac{1}{2}, \bigcup_{k\geq2, l \geq 2} R_{kl},*}^{2}
 \lesssim h^{3} \lVert \partial_x (r^{\gamma-1} \Pi_t^{{p}} \alpha) \max \{ x^{\tilde{\gamma}} , y^{\tilde{\gamma}}\}^{{1/2}} x^{\tilde{\gamma}} \rVert_{r,0,Q}^2 \label{3.J} \\
 + h^{3} \lVert \partial_y (r^{\gamma-1} \Pi_t^{{p}} \alpha) \max \{ x^{\tilde{\gamma}} , y^{\tilde{\gamma}} \}^{{1/2}} y^{\tilde{\gamma}} \rVert_{r,0,Q}^2 \ .\nonumber
\end{align}
As $ |\partial_x (r^{\gamma-1} \Pi_t \alpha ) | \lesssim r^{\gamma-2} \tilde{\alpha}(t,\theta)$ for some 
$\tilde{\alpha}$ square-integrable in {$\theta$ and piecewise polynomial in $t$}, and $ \max \{ x^{\tilde{\gamma}} , y^{\tilde{\gamma}} \} \leq r^{\tilde{\gamma}}$, the right hand side of \eqref{3.J} is finite if 
\begin{equation}\label{betaestimation}
\beta>\frac{3}{2(\gamma+1/2)}\ .
\end{equation}
Therefore 
\begin{align*}
 \lVert r^{\gamma-1} \Pi_t^{{p}} \alpha - \Pi_{x,y}^{{0}} r^{\gamma-1} \Pi_{t}^{{p}} \alpha \rVert_{r,-\frac{1}{2}, \bigcup_{k\geq2, l \geq 2} R_{kl},*}^{2}
 \lesssim h^{3}\ ,
\end{align*}
provided $ \Delta t \leq h_{k}$ for all $k$.

Estimate for $k=1,\, l>1$ (analogously $k>1,\, l=1$): With $ f(x,y) = r^{\gamma-1} \alpha(\theta)$
\begin{align*}
 & \sum\limits_{j} \sum\limits_{l=2}^{N} \lVert (1-\Pi_{x y}^{{0}}) \Pi_t^{{p}} f \rVert_{r,-\frac{1}{2},[t_{j},t_{j+1}) \times R_{k,l},*}^{2}
 \\ &\leq \sum\limits_{j}\sum\limits_{l=2}^N\max\{\Delta t,h_k,h_l \}\left(h_1^2\Vert \partial_x (r^{\gamma-1} \Pi_t \alpha )\Vert^2_{r,0,[t_{j},t_{j+1}) \times R_{k,l},*} + h_l^2 \Vert \partial_y (r^{\gamma-1} \Pi_t \alpha ) \Vert^2_{r,0,[t_{j},t_{j+1}) \times R_{k,l},*} \right) \ .
\end{align*}
Proceed as in \eqref{3.J} to see that also this term is bounded for $\beta > \frac{3}{2 (\gamma+\frac{1}{2})}$.

Estimate for $k=1,\, l=1$: $r^{\gamma-1}\in L^2(R_{11})$ because $\gamma >0$, so the 
$L^{2}$-error on $R_{11}$ is $ \leq h_{1}^{3}$.

\begin{align*}
 & \|r^{\gamma-1} \Pi_t^{{p}} \alpha - \Pi_t^{{p}} \Pi_{x,y}^{{0}} r^{\gamma-1} \Pi_t^{{p}} \alpha(t, \theta)\|_{r,-\frac{1}{2},R_{11},*}^{2} \\
 & \lesssim \lVert (1-\Pi_{x y}^{{0}}) r^{\gamma-1} \Pi_{t}^{{p}} \alpha(t,\theta) \rVert_{r,-1,R_{11},*} \lVert (1-\Pi_{x y}^{{0}}) r^{\gamma-1} \Pi_{t}^{{p}} \alpha(t,\theta) \rVert_{r,0,R_{11},*} \ .
\end{align*}
The second term is $ \leq h^{\gamma}$. For the first term we obtain
\begin{align*}
 \lVert (1-\Pi_{x y}^{{0}}) r^{\gamma-1} \Pi_{t}^{{p}} \alpha(t,\theta) \rVert_{r,-1,R_{11},*} \equiv \sup\limits_{g \in H^{-1}(\mathbb{R^{+}},\tilde{H}^{1}(R_{11}))}
   \frac{\langle (1-\Pi_{x y}^{{0}}) r^{\gamma-1} \Pi_{t}^{{p}} \alpha(t,\theta) , g \rangle}{\lVert g \rVert_{-r,1,R_{11}}}\ .
\end{align*}
Replacing $g$ by $g-{G}$, where ${G}$ is the $ H^{-r}(\mathbb{R}^{+},H^{0}(R_{11}))$-projection of $g$, we obtain for $\Delta t \leq h_{1}$:
\begin{align*}
 \lVert (1-\Pi_{x y}^{{0}}) r^{\gamma-1} \Pi_{t}^{{p}} \alpha(t,\theta) \rVert_{r,-1,R_{11},*} \leq \lVert (1-\Pi_{x y}^{{0}}) r^{\gamma-1} \Pi_{t}^{{p}} \alpha(t,\theta) \rVert_{r,0,R_{11}} \sup\limits_{g} \frac{\lVert g-{G} \rVert_{-r,0,R_{11}}}{\lVert g \rVert_{-r,1,R_{11}}} \leq h_{1}^{\gamma} h_{1} \ .
\end{align*}
We conclude
\begin{equation*}
 \|r^{\gamma-1} \Pi_t^{{p}} \alpha - \Pi_t^{{p}} \Pi_{x,y}^{{0}} r^{\gamma-1} \Pi_t^{{p}} \alpha(t, \theta)\|_{r,-\frac{1}{2},R_{11},*}^{2}
 \lesssim h_{1}^{2 \gamma +1} \leq h^{3}\ .
\end{equation*}

The approximation of the corner-edge singularities $r^{-1}(\sin(\theta))^{-\frac{1}{2}}$ are similarly obtained from the elliptic results. For brevity we omit the details.
\end{proof}

The proof of Theorem 20 a) uses the following the key elliptic result in \cite{petersdorff}
for the trace $u|_\Gamma$ and follows analogously to the above case. It was  proven there for closed polyhedral surfaces.

\begin{theorem}\label{keyucorner}
  Let $u \in \widetilde{H}^{\frac{1}{2}}(\Gamma) $ have a singular decomposition like the one in Theorem \ref{decompS} near every corner of $\Gamma$. 
Then we can approximate $u$ {by piecewise linear functions  on the $\beta$-graded mesh for $\beta \geq 1$} in the following way: \\
For $ u_{h} = \Pi_x^{{1}} \psi$, we have for all $ \epsilon > 0$ and all $s \in [0,\frac{1}{2}]$
\begin{equation*}
   \| u - u_h \|_{H^{s}(\Gamma)} \leq C_{\beta, \varepsilon} h^{\min\{\beta(1-s), 2-s{\}}{-\varepsilon}}\ .
\end{equation*}
\end{theorem}

\section{Algorithmic considerations}
\label{algosect}

On the left hand side of \eqref{weakformh}, we use  ansatz functions $\psi_{\Delta t, h}(t,x)=\sum_{m,i}\psi_i^m\gamma_{\Delta t}^m(t){\psi_h^i(x)  \in V^{0,0}_{h,\Delta t}}$ and test functions $\Psi^{n,l}(t,x)=\gamma_{\Delta t}(t){\psi_h^l(x)(x)}  \in V^{0,0}_{h,\Delta t}$ to obtain for the single layer potential:
\begin{align*}
\int_0^\infty \langle V \psi_{\Delta t,h}, \dot{\gamma}^n_{\Delta t}{\psi_h^l}\rangle dt 
&= \sum_{m,i}\psi_i^m\frac{1}{4\pi} \int_0^\infty \int_{\Gamma \times \Gamma}\frac{1}{|x-y|}\gamma_{\Delta t}^m(t-|x-y|)
{\psi_h^i}(y)\dot{\gamma}^n_{\Delta t}(t){\psi_h^l}(x) ds_x ds_y dt\\
&= \sum_{m,i}\psi_i^m\frac{1}{4\pi} \int_{\Gamma \times \Gamma}\frac{{\psi_h^i}(y){\psi_h^l}(x)}{|x-y|}\int_0^\infty\gamma_{\Delta t}^m(t-|x-y|)\dot{\gamma}^n_{\Delta t}(t)dt\ ds_x ds_y\\
&= \sum_{m,i}\psi_i^m\frac{1}{4\pi} \int_{\Gamma \times \Gamma}\frac{{\psi_h^i}(y){\psi_h^l}(x)}{|x-y|}\ (\chi_{E_{n-m-1}}(x,y)-\chi_{E_{n-m}}{(x,y)})\ ds_x ds_y\\
&= \sum_{m,i}\psi_i^m\frac{1}{4\pi}[ \int_{E_{n-m-1}}\frac{{\psi_h^i}(y){\psi_h^l}(x)}{|x-y|}\ ds_x ds_y -\int_{E_{n-m}}\frac{{\psi_h^i}(y){\psi_h^l}(x)}{|x-y|}\ ds_x ds_y]
\end{align*}
\noindent for all $n=1, ..., N_t$ and $l=1, ..., N_s$. Here the light cone $E_l$ is defined as
\begin{align*}
 E_l:=\left\{(x,y)\in\Gamma\times \Gamma:\ t_l\leq|x-y|\leq t_{l+1}\right\}\ ,
\end{align*}
and its indicator function is defined as $\chi_{E_l}{(x,y)=1}$ {if $(x,y) \in E_l$, and $\chi_{E_l}(x,y)=0$ otherwise}. The integrals are evaluated using a composite $hp$-graded quadrature \cite{gimperleinreview}.
\\

For piecewise constant test functions in time, the Galerkin discretization leads to a block--lower--triangular system of equations, which can be solved by blockwise forward substitution. For the Dirichlet problem \eqref{weakformh} we obtain an algebraic system of the form
\begin{align*}
 \textstyle{\sum\limits_{m=1}^{n}V^{n-m}\psi^m=f^{n-1}-f^{n}}\ ,
\end{align*}
{where $\psi^m$ is the vector with components $\psi_i^m$ of the the ansatz function $\psi_{\Delta t, h}(t,x)$ and $f^n = \int_\Gamma f(t_n,x) \ ds_x$.} Forward substitution gives rise to the \emph{marching-in-on-time} (MOT) scheme
\begin{align}\label{eq:MOT_LGS}
V^0\psi^n=f^{n-1}-f^{n}- \textstyle{\sum\limits_{m=1}^{n-1}}V^{n-m}\psi^m\,.
\end{align}
The resulting algorithm is given as Algorithm \ref{MOT}.

\begin{algorithm}
  \caption{\label{MOT} {Marching-on-in-time} algorithm.}
  \begin{algorithmic}
  \FOR{$n=1,2,\ldots$}
    \IF{$n-1>\left[\frac{ \operatorname{diam}\Gamma}{\Delta t}\right]$}
			\STATE $V^{n-1}=0$
		\ELSE
			\STATE Compute and store
			  $$ (V^{n-1})_{il}=\frac{1}{4\pi} \int_{E_{n-1}}\frac{{\psi_h^i}(y){\psi_h^l}(x)}{|x-y|}\ ds_x ds_y, \quad i,l=1,\ldots, N_s$$
		\ENDIF
    \STATE Compute right hand side $f^{n-1}-f^n-\sum_{m=1}^{n-1}V^{n-m}\psi^m$
		\STATE Solve system of linear equations \eqref{eq:MOT_LGS}
		\STATE Store solution $\psi^n$
  \ENDFOR
  \end{algorithmic}
\end{algorithm}
We remark that for a bounded surface $\Gamma$ the matrices  $V^{n-m}$ vanish whenever the time difference $l=n-m$ satisfies
$l>\left[\frac{ \operatorname{diam}\Gamma}{\Delta t}\right]$, i.e.~the light cone has passed the entire surface $\Gamma$.\\

The implementation of $W$ is based on the weak form \eqref{weakformWh} and the formula 
\begin{align*}
 \int_{\mathbb{R}^+\times \Gamma} (W \phi)\ \partial_t \Phi \ dt \, ds_x\  &= \frac{1}{2 \pi} \int_{0}^{\infty} \int_{\Gamma \times \Gamma}
 \Big\lbrace \frac{-n_x \cdot n_y}{|x-y|} \dot \phi(\tau,y) \ddot \Phi(t,x) 
 \\ &\qquad\ \ + \frac{( \nabla_\Gamma \phi )(\tau,y) \cdot ( \nabla_\Gamma \dot\Phi)(t,x)}{|x-y|} \Big\rbrace ds_y\ ds_x \ dt\ ,
\end{align*}
see \cite{gimperleintyre} for details. We use ansatz functions in $\widetilde{V}^{1,1}_{h,\Delta t}$. To obtain an MOT scheme the test functions $\dot\Phi_{h, \Delta t}(t,x) \in \widetilde{V}^{0,1}_{h,\Delta t}$ are piecewise constant in time and piecewise linear in space.

Similar formulas hold for the operators $K, K'$, and variants of the discretizations for $V$, $W$. The resulting MOT schemes are described in \cite{banz2}. They can be combined into a stable scheme for the Dirichlet-to-Neumann operator $\mathcal{S}$ from \eqref{1.4}, with $\sigma=0$, using the representation $\mathcal{S}= W-(K'-\frac{1}{2}I)V^{-1}(K-\frac{1}{2}I)$ in terms of layer potentials. As in \cite{contact}, the Dirichlet-to-Neumann equation \eqref{DNeq}, $\mathcal{S} u= h$, 
 is equivalently reformulated as  follows:
\\ For given $h\in H^{\frac{3}{2}}_\sigma(\mathbb{R}^+,{H}^{-\frac{1}{2}}(\Gamma)) $, find $\phi \in H^{\frac{1}{2}}_\sigma(\mathbb{R}^+,\tilde{H}^{\frac{1}{2}}(\Gamma)) ,\psi\in {H}^{\frac{1}{2}}_\sigma(\mathbb{R}^+,\tilde{H}^{-\frac{1}{2}}(\Gamma)) $ such that
\begin{align}
\timeint \langle W\phi - (K'-\textstyle{\frac{1}{2}})\psi ,{\Phi}\rangle_\Gamma \dt =\timeint \langle h,{\Phi}\rangle_\Gamma\dt\ ,
\\ \timeint [\langle V\psi,{\partial_t \Psi}\rangle_\Gamma  - \langle(K-\textstyle{\frac{1}{2}})\phi, {\partial_t \Psi} \rangle_\Gamma] \dt =0,
\end{align}
holds for all $\Phi \in {H}^{\frac{1}{2}}_\sigma(\mathbb{R}^+,\tilde{H}^{\frac{1}{2}}(\Gamma)) $,$\Psi \in {H}^{\frac{1}{2}}_\sigma(\mathbb{R}^+,\tilde{H}^{-\frac{1}{2}}(\Gamma)) $.\\ For the discretization, we look for $\phi_{\Delta t,h} \in \tilde{V}^{1,1}_{\Delta t,h}$, $\psi_{\Delta t,h} \in V^{1,1}_{\Delta t,h}$ linear in space and time. To obtain a marching-on-in-time scheme we test the first equation against constant test functions in time and the second equation against the time derivative of constant test functions.

\section{Numerical experiments}
\label{experiments}

\subsection{Single layer potential}

\begin{example}\label{example0}
Using the discretization from Section \ref{discretization}, we compute the solution to the integral equation $V\psi=f$ on {$\mathbb{R}_t^+ \times \Gamma$ with} the circular screen $\Gamma = \{(x,y,0) : 0 \leq \sqrt{x^2+y^2}\leq 1\}$ depicted in Figure \ref{gradmesh}. We use the weak form \eqref{weakformh} with constant test and ansatz functions in  space and time. The right hand side is given by $f(t,x)=\cos(|k|t-k\cdot x) \exp({-}1/(10 t^{2}))$, where $k=(0.2,0.2,0.2)$. The time discretization errors are negligibly small in this numerical experiment, when the time step is chosen to be $\Delta t=0.005$. We compute the solution up to $T=1$.  The finest graded mesh consists of 2662 triangles, and we use the solution on this mesh as reference solution using the same $\Delta t=0.005$. 
\end{example}
Figure \ref{edgeGradedann} shows the density along a cross-section on a $\beta$-graded mesh with $\beta$=2 and 2662 triangles at time $T=0.5$. The figure exhibits the edge singularities predicted by the decomposition in equation \eqref{decompositionEdge} and illustrates the qualitative behavior of the solution. 
\begin{figure}[H]
\centering     
\makebox[\linewidth]{
\includegraphics[width=12cm]{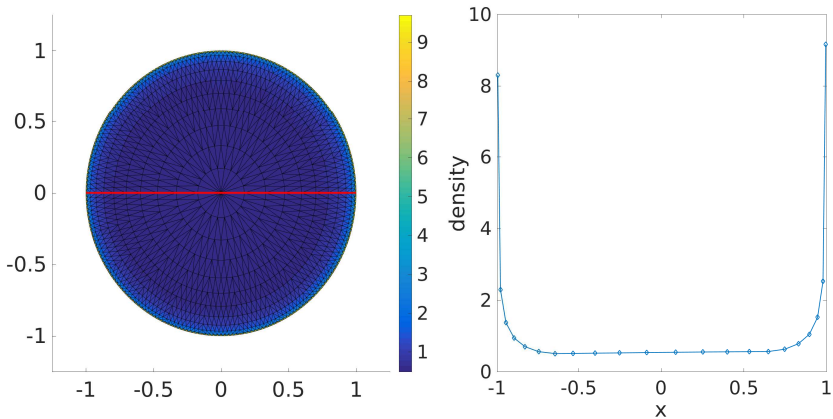}}
\caption{Solution of the single layer equation at $T=0.5$ along $y=0$ on the circular screen, Example \ref{example0}}
\label{edgeGradedann}
\end{figure}

Figure  \ref{edge1ann} examines the detailed singular behavior near the outer edge at $(1,0)$. It plots the numerical density at times $T=0.5,\ 0.75,\ 1.0$ against the distance to the edge along $x$-axis. In the log-log plot the slope of the curve near $0$ corresponds to the edge exponent in decomposition \eqref{decompositionEdge}. \\ The numerical solution exhibits edge singularities in close agreement with \eqref{decompositionEdge}. Numerically, the singular exponents are within $8 \%$ of the theoretical value of $-\frac{1}{2}$ for the edge at these times. Note that the convergence of our boundary element method in the energy norm does not a priori  imply convergence for the numerically computed singular exponents.    

For Example \ref{example0}, we finally consider the error compared to the benchmark solution on the 2-graded mesh. Because of the low spatial regularity of the solution, the numerical solutions cannot be expected to converge in $L^2([0,T] \times \Gamma)$. As a weaker measure, we consider 
the energy norm defined by the single layer operator, which is computed from the stiffness matrix $V$ and the solution vector $u$ as $E(\psi)=\frac{1}{2}\psi^\top V \psi -\psi^\top f$. It is comparable or weaker than the norm of $H^{0}_\sigma(\mathbb{R}^+,{H}^{-\frac{1}{2}}(\Gamma))$.  For the error as a function of the degrees of freedom, Figure \ref{l1errorann} shows convergence in the energy norm with a rate $-0.52$ on the 2-graded  mesh, respectively $-0.26$ on the uniform mesh. The error therefore behaves in agreement with the approximation properties proportional to $\sim h$ (equivalently, $\sim DOF^{-\frac{1}{2}}$) on the 2-graded mesh, while the convergence is $\sim h^{1/2}$ ($\sim DOF^{-1/4}$) on a uniform mesh.

\begin{figure}[H] 
\makebox[\linewidth]{
\includegraphics[width=12cm]{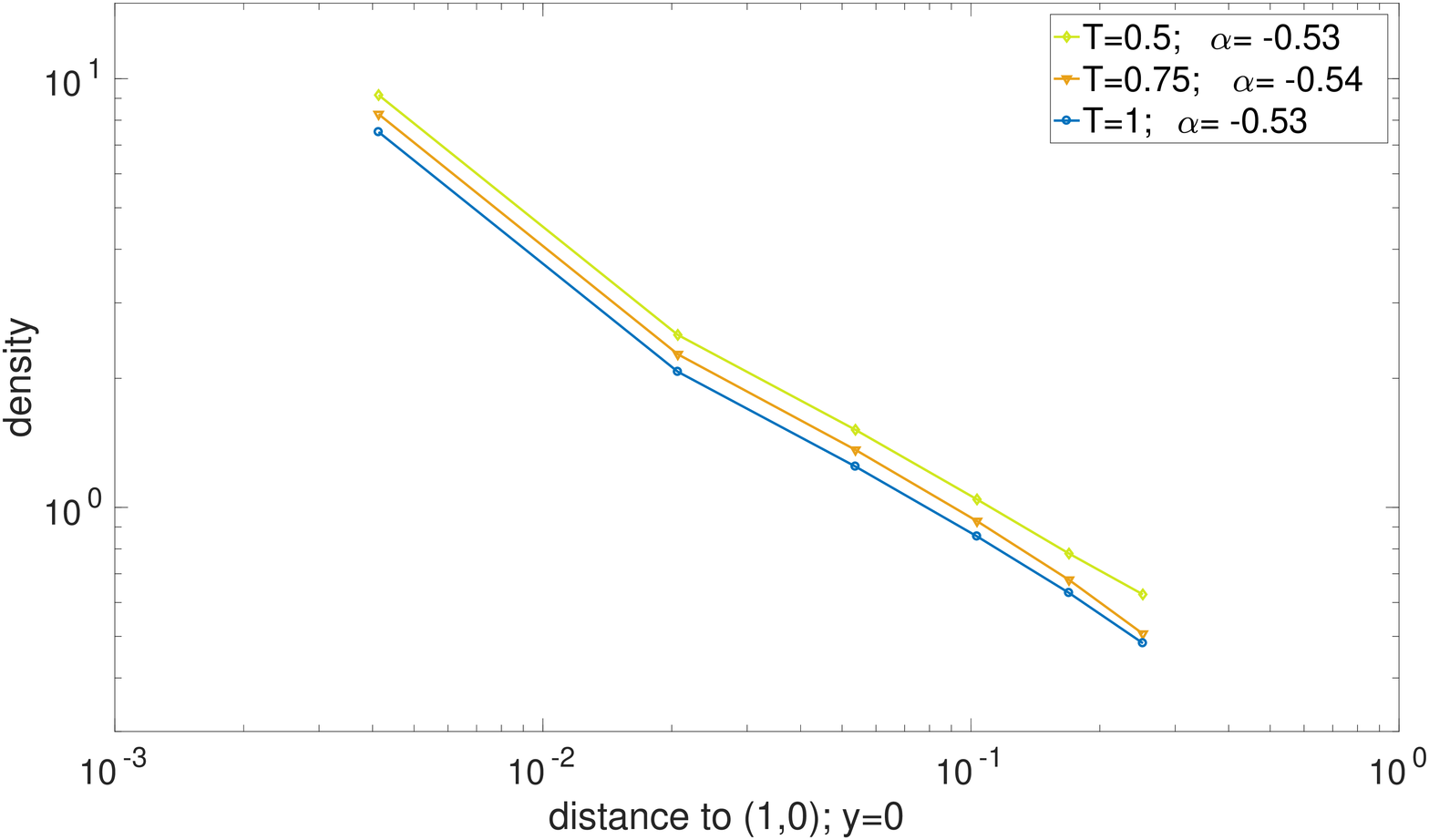} }
\caption{Asymptotic behavior of the solution to the single layer equation near edge along $y=0$, Example \ref{example0}}
\label{edge1ann}
\end{figure}
\begin{figure}[H]{
 \makebox[\linewidth]{
 \includegraphics[width=12cm]{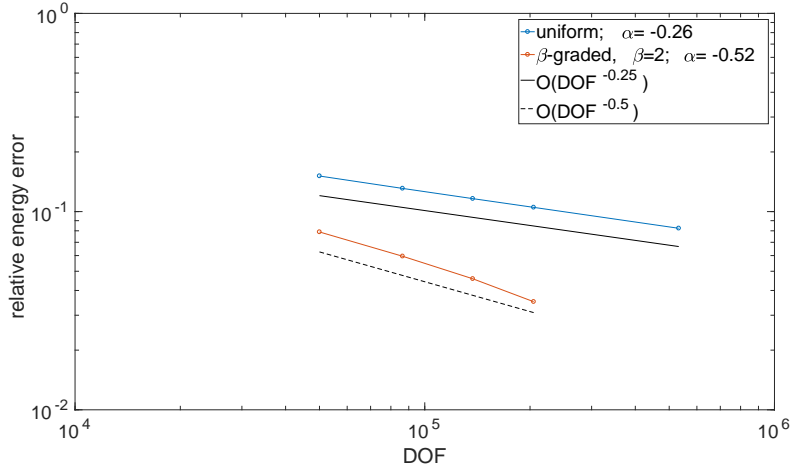}
 }
 \caption{Energy error for single layer equation on circular screen, Example \ref{example0}} \label{l1errorann}}
 \end{figure}

\begin{example}\label{example1}
Using the discretization from Section \ref{discretization}, we compute the solution to the integral equation $V\psi=f$ on {$\mathbb{R}_t^+ \times \Gamma$ with} the square screen $\Gamma = [-1,1]^2\times\{0\}$ using the weak form \eqref{weakformh}, with constant test and ansatz functions in  space and time. The right hand side is given by $f(t,x)=\cos(|k|t-k\cdot x) \exp(-1/(10 t^{2}))$, where $k=(0.2,0.2,0.2)$. The time discretization errors are negligibly small in this numerical experiment, when the time step is chosen to be $\Delta t=0.005$. We compute the solution up to $T=1$.  The finest graded mesh consists of $2312$ triangles, and we use the solution on this mesh as reference solution using the same $\Delta t=0.005$. 
\end{example}
Figures \ref{cornerGraded} and \ref{edgeGraded} show the density along a cross-section and along a longitudinal section on a $\beta$-graded mesh with $\beta$=2 and $2312$ triangles at time $T=0.5$. Both figures exhibit the corner and edge singularities predicted by the decomposition \eqref{decomposition} and illustrate the qualitative behavior of the solution. Figure \ref{allCorner} compares the solution along the cross-section on a $2$-graded mesh against the solution on two uniform meshes. We see that the $2$-graded mesh yields a {higher} resolution of the corner singularities compared to the uniform meshes. 
\begin{figure}[H]
\centering     
\makebox[\linewidth]{
\includegraphics[width=12cm]{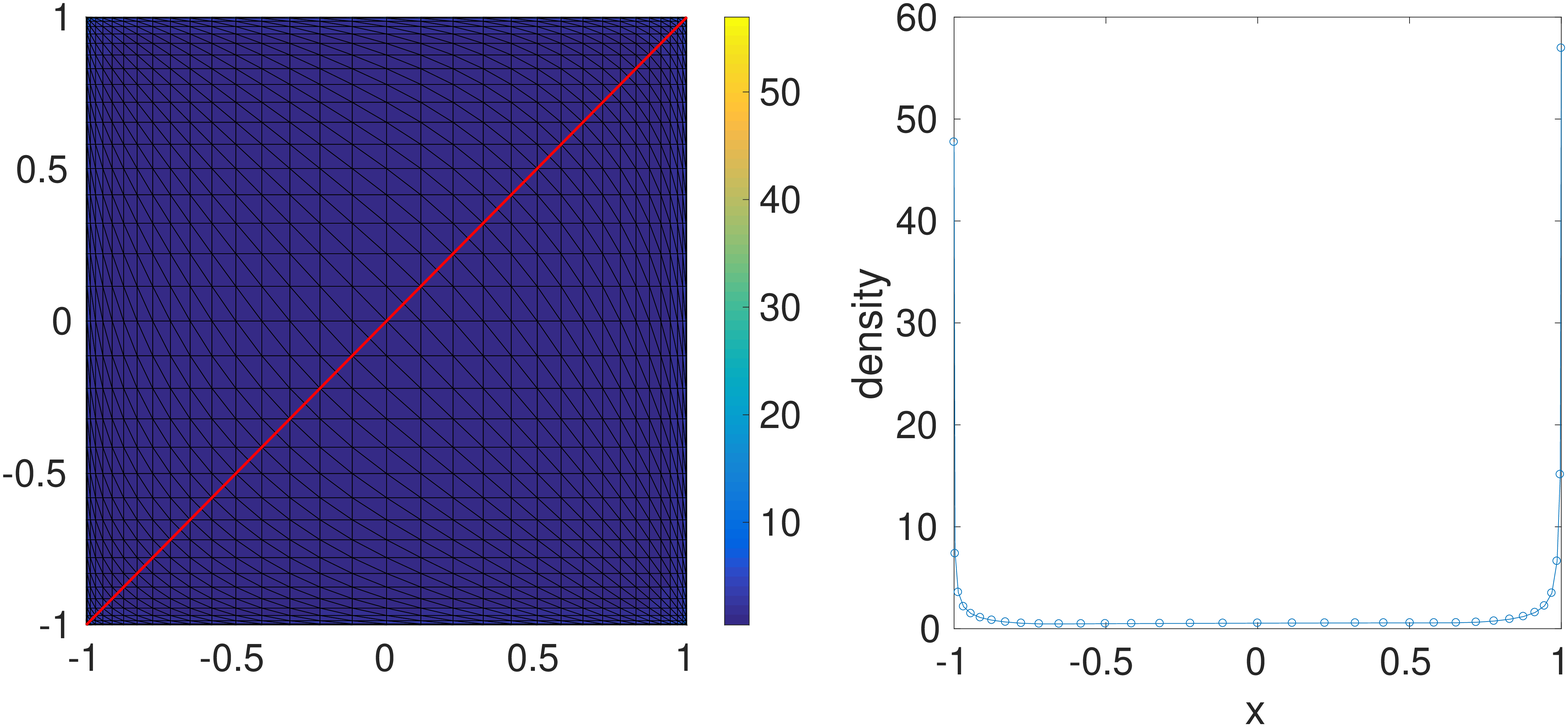}}
\caption{Solution of the single layer equation at $T=0.5$ along $y=x$ on the square screen, Example \ref{example1}}
\label{cornerGraded}
\end{figure}
\begin{figure}[H]
\centering     
\makebox[\linewidth]{
\includegraphics[width=12cm]{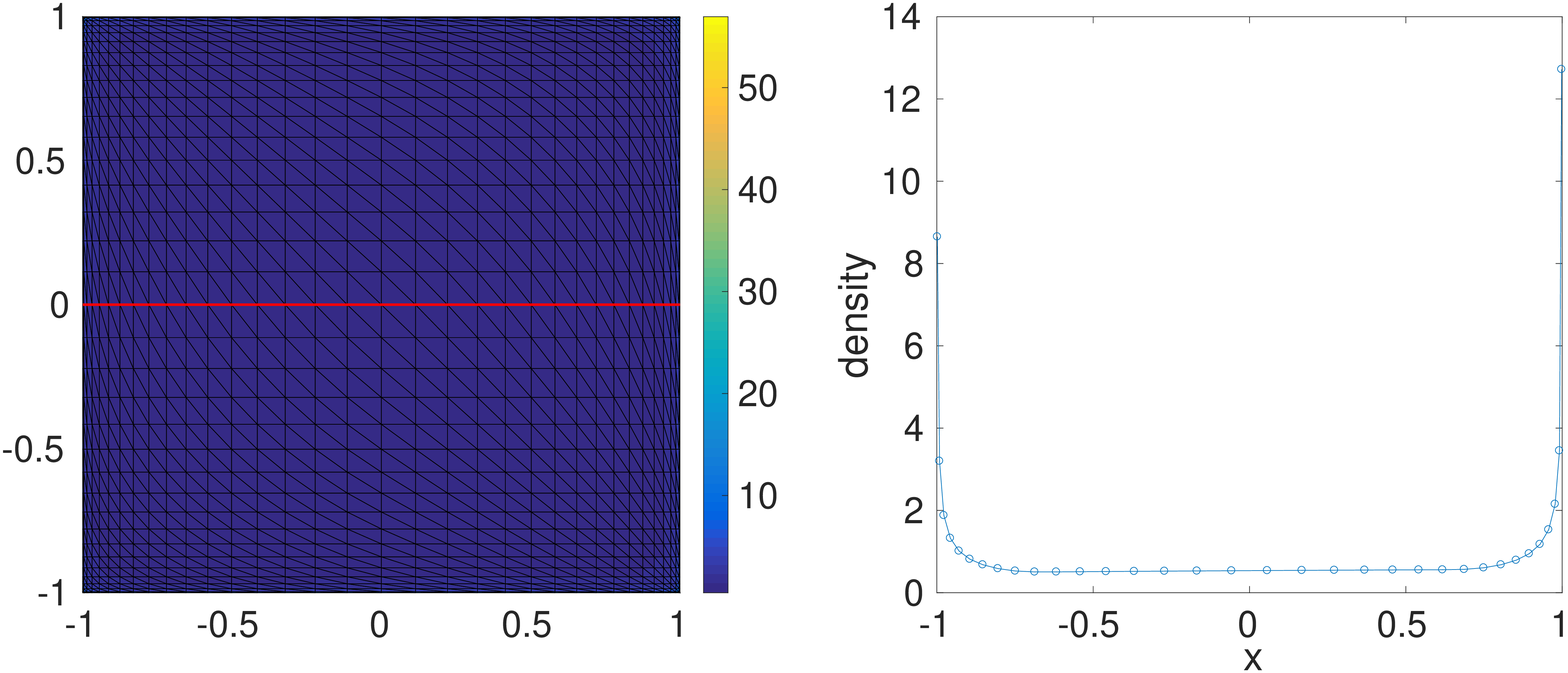}}
\caption{Solution of the single layer equation at $T=0.5$ along $y=0$ on the square screen, Example \ref{example1}}
\label{edgeGraded}
\end{figure}

\begin{figure}[H]
\centering     
\makebox[\linewidth]{
\includegraphics[width=12cm]{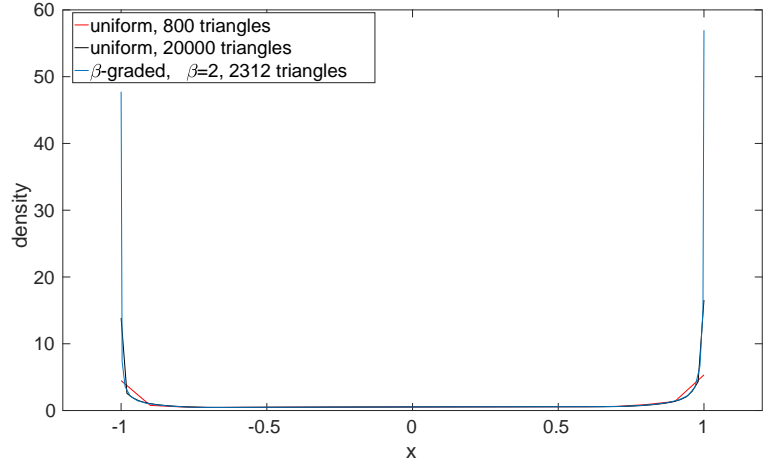}}
\caption{Numerical computation of the corner singularity along diagonal from $(-1,-1)$ to $(1,1)$ at time $T=0.5$, Example \ref{example1}}
\label{allCorner}
\end{figure}
Figure  \ref{corner1} examines the detailed singular behavior near the corner $(1,1)$. It plots the numerical density at times $T=0.5,\ 0.75,\ 1.0$ against the distance to the corner along the diagonal of the screen. In the log-log plot the slope of the curve near $0$ corresponds to the corner exponent in decomposition \eqref{decomposition}. Similarly, Figure \ref{edge1} shows the density as a function of $x$ for $y=0$, perpendicular to the edge, at the same times.\\ After a short computational time, the numerical solution exhibits edge and corner singularities corresponding to \eqref{decomposition}. Numerically, the singular exponents at large enough times $T=0.5,\ 0.75,\ 1$ are within $2 \%$ of the theoretical value of $-\frac{1}{2}$ for the edge, while they are around $-0.78$ for the corner, approximately $10 \%$ higher than the theoretical exponent $\gamma-1$. Note that the convergence of our boundary element method in the energy norm does not a priori  imply convergence for the numerically computed singular exponents.    
\begin{figure}[H] 
\makebox[\linewidth]{
\includegraphics[width=12cm]{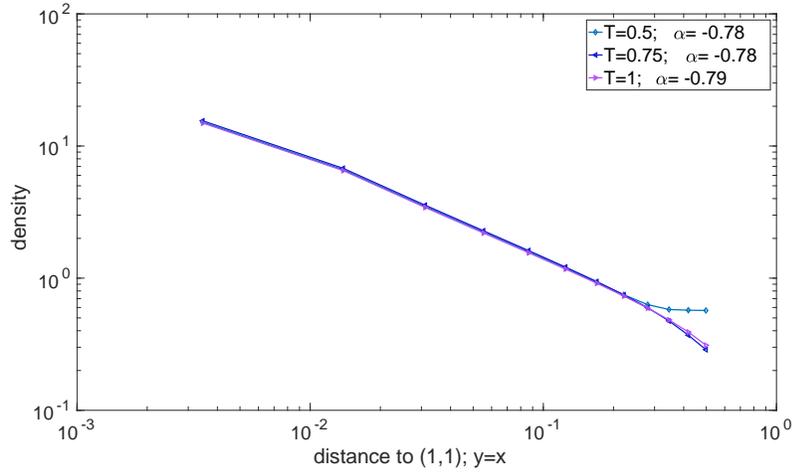} }
\caption{Asymptotic behavior of the solution to the single layer equation near corner along $y=x$, Example \ref{example1}}
\label{corner1}
\end{figure}
\begin{figure}[H] 
\makebox[\linewidth]{
\includegraphics[width=12cm]{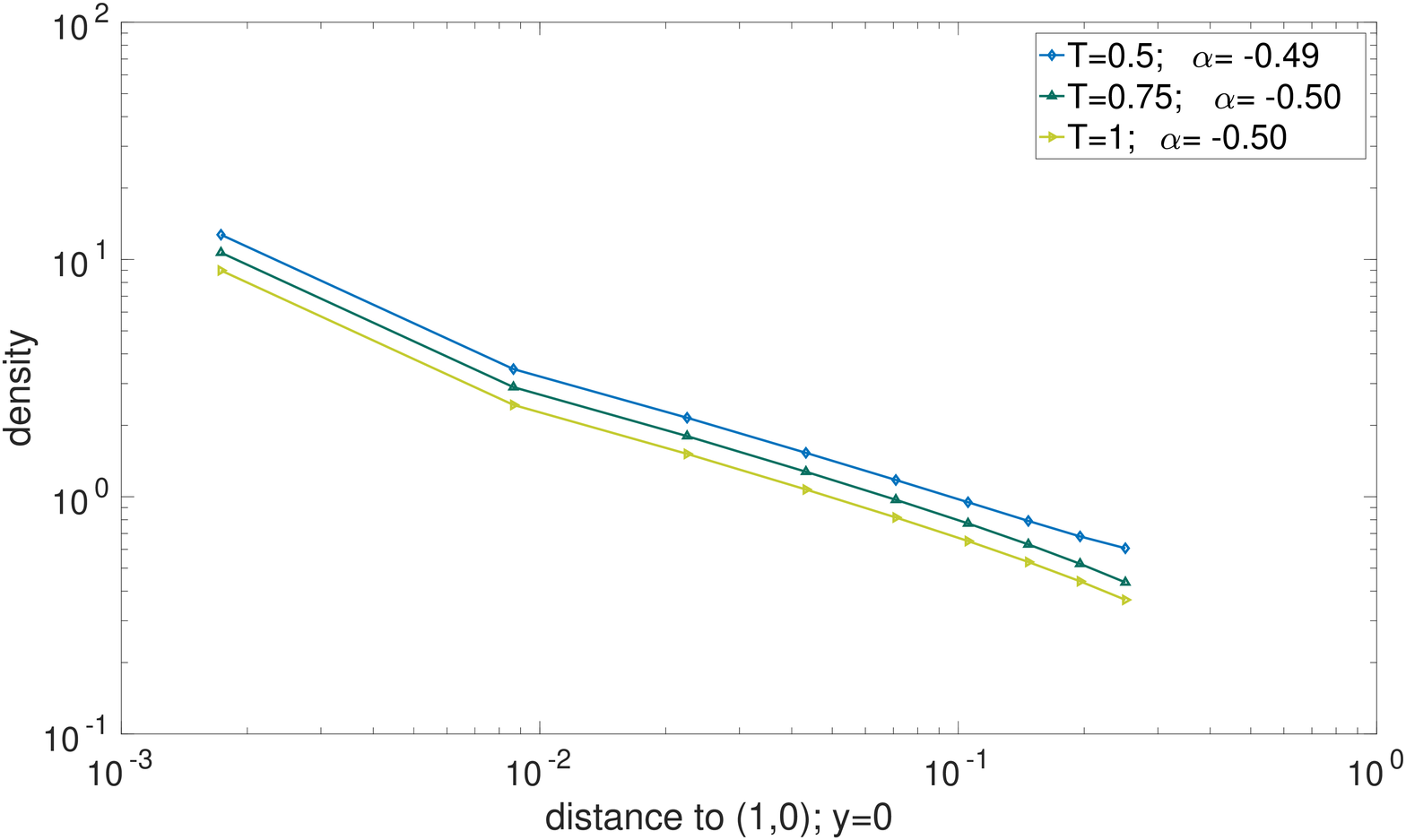} }
\caption{Asymptotic behavior of the solution to the single layer equation near edge along $y=0$, Example \ref{example1}}
\label{edge1}
\end{figure}
\begin{figure}[H]{
 \makebox[\linewidth]{
 \includegraphics[width=12cm]{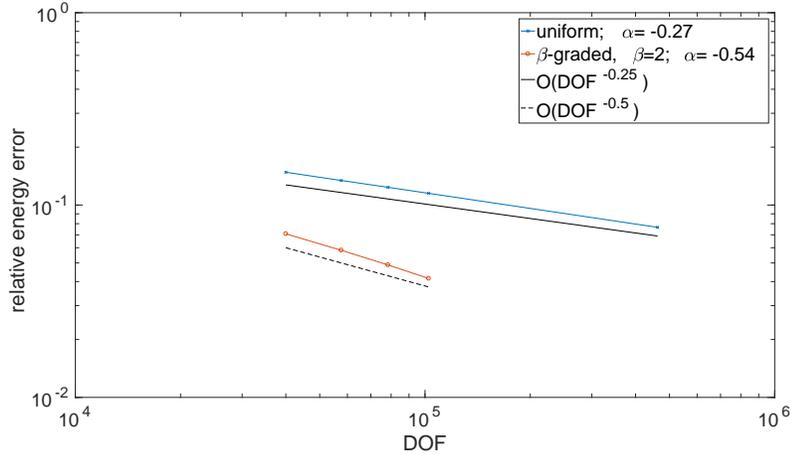}
 }
\caption{Energy error norm for single layer equation on square screen, Example \ref{example1}} }\label{l1errorfig}
 \end{figure}

For Example \ref{example1}, we finally consider the error compared to the benchmark solution on the 2-graded mesh. Like in Example \ref{example0}, we consider the energy norm defined by the single layer operator. 
 Figure 11 shows convergence of the norm with rates $-0.54$ on the 2-graded mesh, respectively $-0.27$ on the uniform mesh in terms of degrees of freedom. These closely mirror the approximation results, which predict an approximation error proportional to $\sim h$ (equivalently, $\sim DOF^{-\frac{1}{2}}$) on the 2-graded mesh, while the approximation error is $\sim h^{\frac{1}{2}}$ ($\sim DOF^{-\frac{1}{4}}$) on a uniform mesh. In particular, compared to Example \ref{example0}, the corner singularities of the square screen do not affect the convergence rate.

To further probe the effect of the corners we also consider the $L_2$ norm in time of the sound pressure evaluated in a point. For applications the approximation of the sound pressure away from the screen is often the most relevant measure. We evaluate the sound pressure by substituting the density $\psi_{\Delta t,h}$ into the single layer potential, $p_{\Delta t,h}=S \psi_{\Delta t,h}$, and use a tensor product Gaussian quadrature with $400$ nodes per triangle to evaluate the integral.  Figure \ref{soundErr} shows the $L_2$ error in time of the sound pressure evaluated in three points outside of the screen, $(1,1,0.004)$, $(0.75, 0.75, 1)$ and $(1,1.25, 0.25)$. In each of the points, the convergence is proportional to $\sim h^2$, resp.~$\sim h$, as for the energy norm. However, while the convergence rate is in agreement with the energy norm, the error in the sound pressure strongly depends on the location of the point. In $(1,1,0.004)$, at distance $0.004$ from the corner of the screen, the error is an order of magnitude higher than in the points $(0.75, 0.75, 1)$ and $(1,1.25, 0.25)$, which are at a distance of order $1$.

\begin{figure}[H]{
 \makebox[\linewidth]{
 \includegraphics[width=14.3cm]{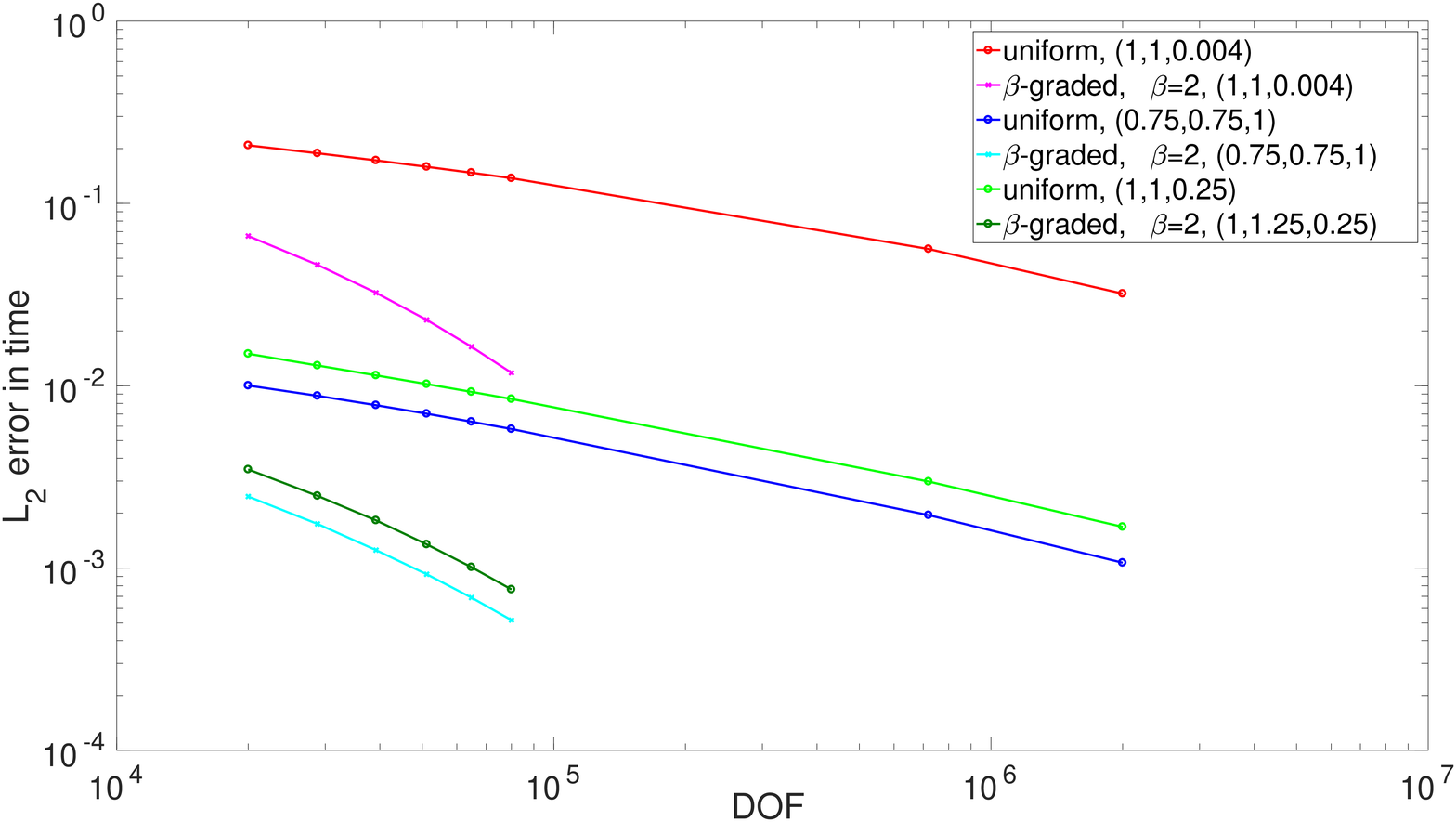}
 }
 \caption{$L_2([0,T])$ error for the sound pressure in three points outside square screen, computed from single layer equation, Example \ref{example1}} \label{soundErr}}
 \end{figure}

\subsection{Hypersingular operator}

\begin{example}\label{example2a0}
Using the discretization from Section \ref{discretization}, we compute the solution to the integral equation $W\phi=g$ {on $\mathbb{R}_t^+ \times \Gamma$ with} the circular screen $\Gamma = \{(x,y,0) : 0 \leq \sqrt{x^2+y^2}\leq 1\}$ depicted in Figure \ref{gradmesh}. We use the weak form \eqref{weakformWh} with linear ansatz and test functions in space, linear ansatz and constant test functions in time. Here,
\begin{align*}
 g{(t,x)}&=\textstyle{(-\frac{3}{4}+\cos(\frac{\pi}{2}(4-t))+\frac{\pi}{2}\sin(\frac{\pi}{2}(4-t)) -\frac{1}{4}(\cos({\pi}(4-t))+ \pi\sin({\pi}(4-t))))}\\ & \qquad \times [H(4-t)-H(-t)],
\end{align*}
{where $H$ is the Heaviside function.}
The time discretization errors are negligibly small in this numerical experiment, when the time step is chosen to be $\Delta t=0.01$. We compute the solution up to $T=4$.  The finest graded mesh consists of 2662 triangles, and we use the solution on this mesh as reference solution using the same $\Delta t=0.01$. 
\end{example}

\begin{figure}[H] 
\makebox[\linewidth]{
\includegraphics[width=12cm]{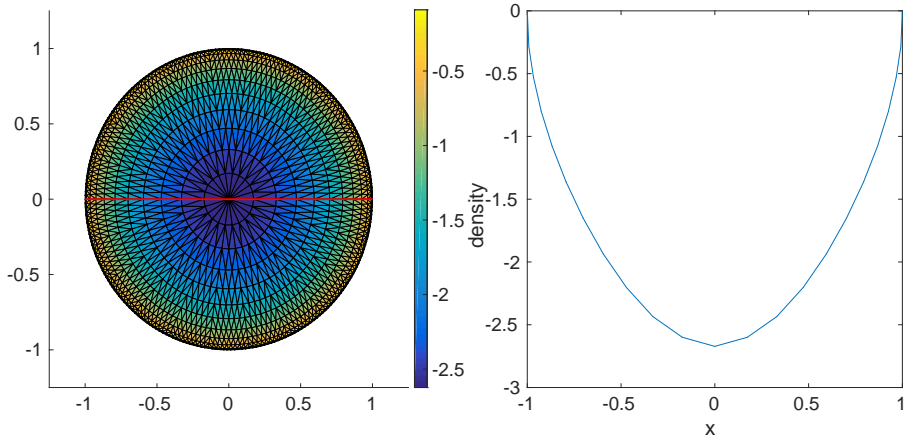} }
\caption{Solution of the hypersingular equation at $T=2$ along $y=0$ on the circular screen, Example \ref{example2a0}}\label{WedgeGradedanndens}
\end{figure}

Figure \ref{WedgeGradedanndens} shows the density along a cross-section on a $\beta$-graded mesh with $\beta$=2 and 2662 triangles at time $T=2$. The figure exhibits the edge singularities predicted by the decomposition \eqref{decompostionEdget} and illustrates the qualitative behavior of the solution. 

\begin{figure}[H] 
\makebox[\linewidth]{
\includegraphics[width=12cm]{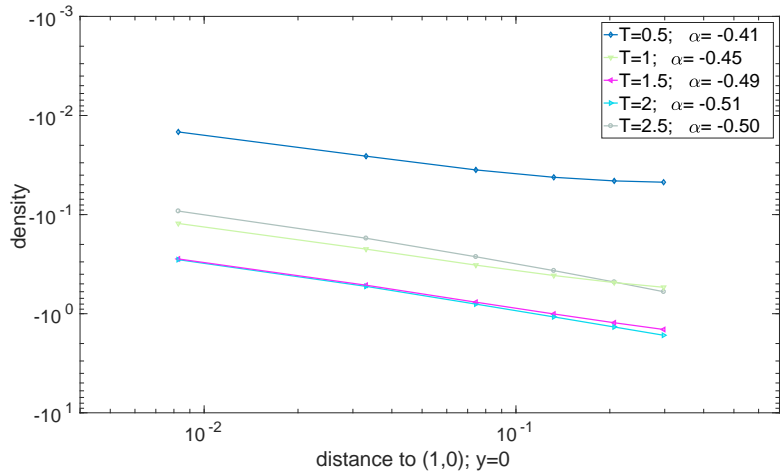} }
\caption{Asymptotic behavior of the solution to the hypersingular equation near edge along $y=0$, Example \ref{example2a0}}
\label{WedgeGradedann}
\end{figure}

Figure  \ref{WedgeGradedann} examines the detailed singular behavior at the circular edge along the $x$-axis near the point $(1,0)$. It plots the numerical density at times up to $T=2.5$ against the distance to the edge. For the singular exponents, we numerically obtain values within $5 \%$ of the theoretical value of $\frac{1}{2}$, except at the earliest time $T=0.5$, when compute an exponent $0.41$. 

\begin{figure}[H]
 \makebox[\linewidth]{
 \includegraphics[width=12cm]{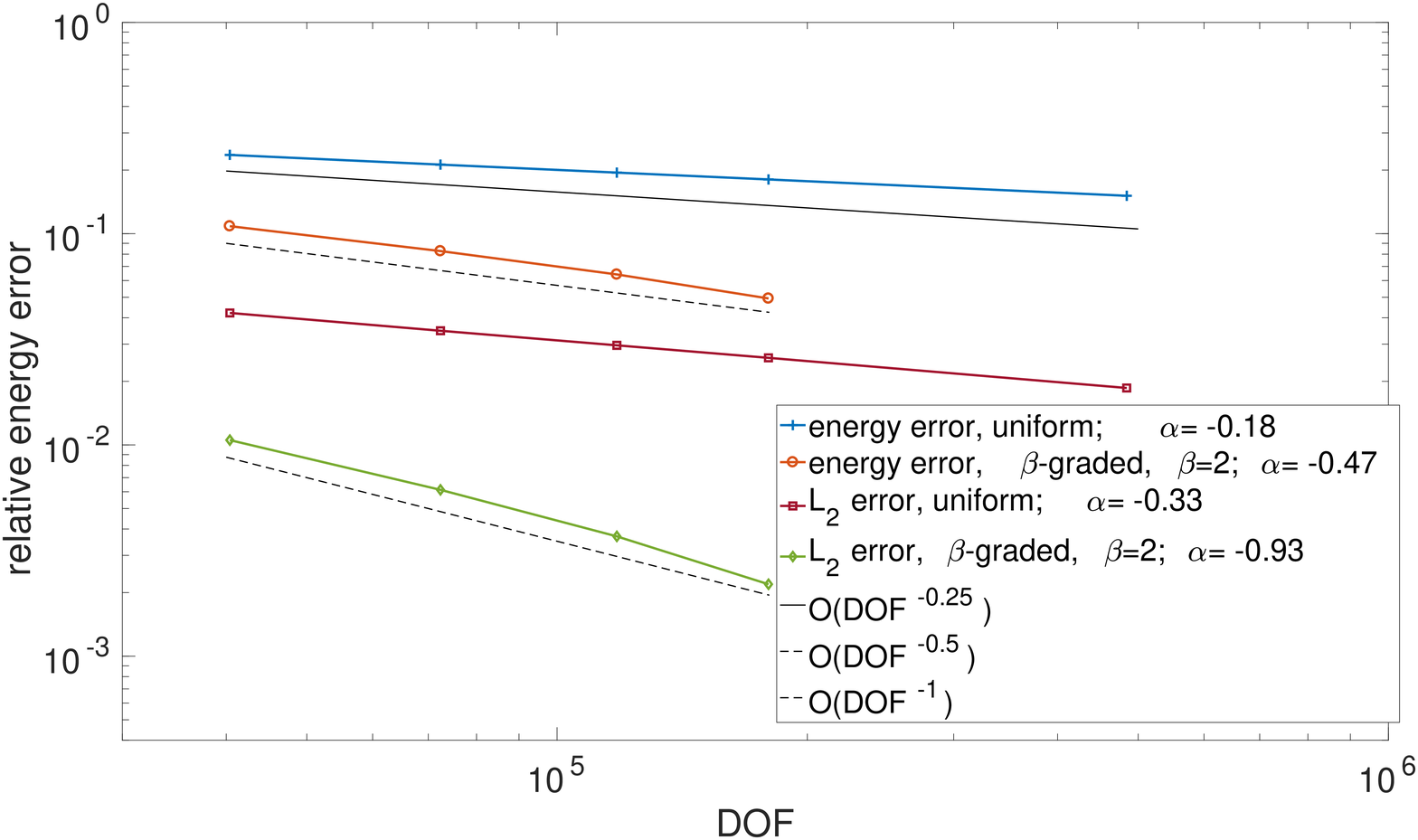}
 }
 \caption{$L_2([0,T],L_2(\Gamma))$ and energy error for hypersingular equation on circular screen, Example \ref{example2a0}} 
 \label{l2Wann}
\end{figure}

Finally, Figure \ref{l2Wann} shows the error in both the energy and $L_2([0,T],L_2(\Gamma))$ norms with respect to the benchmark solution. 
The convergence rate in terms of the degrees of freedom on the $2$-graded mesh is $-0.47$ in energy and $-0.93$ in $L_2$. It is in close agreement with a convergence proportional to $\sim h$ (equivalently, $\sim DOF^{-1/2}$) predicted by the approximation properties in the energy norm, and $\sim h^2$ (equivalently, $\sim DOF^{-1}$) in $L_2$.  On the uniform mesh the rate is $-0.18$ in energy and $-0.33$ in $L_2$.  

\begin{example}\label{example2a}
Using the discretization from Section \ref{discretization}, with test and ansatz functions as in Example \ref{example2a0}, we compute the solution to the integral equation $W\phi=g$ on {$\mathbb{R}_t^+ \times \Gamma$ with} the square screen $\Gamma = [-1,1]^2\times\{0\}$. We prescribe the right hand side
\begin{align*}
 g{(t,x)}&=(-\frac{3}{4}+\cos(\frac{\pi}{2}(4-t))+\frac{\pi}{2}\sin(\frac{\pi}{2}(4-t)) -\frac{1}{4}(\cos({\pi}(4-t))+ \pi\sin({\pi}(4-t))))\\ & \qquad \times [H(4-t)-H(-t)],
\end{align*}
{where $H$ is the Heaviside function,} and set $\Delta t= 0.01$, $T= 4$. The finest graded mesh consists of $2312$ triangles, and we use the solution on this mesh as reference solution using the same $\Delta t= 0.01$.
\end{example}
The density along  the diagonal $x=y$, respectively along $y=0$, exhibit the corner and edge singularities predicted by the decomposition \eqref{decompositiont}. The qualitative behavior of the solution at $T=2$ along the diagonal $y=x$ of the square screen is shown in Figure \ref{Wcornerplot}, illustrating the singularity in the corners. Figure \ref{Wedgeplot} shows the behaviour along $y=0$, with the edge singularity at the boundary of the screen. As the solution to the hypersingular equation lies in ${H}^{\frac{1}{2}}_\sigma(\R^+, \tilde{H}^{\frac{1}{2}}(\Gamma))$,  its conforming numerical approximation tends to zero at both edges and corners.

\begin{figure}[H]{
 \makebox[\linewidth]{
 \includegraphics[width=12cm]{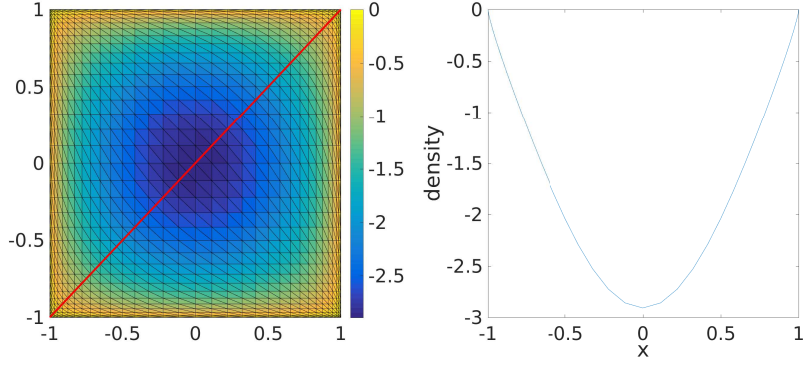}
 }
\caption{Solution of the hypersingular equation at $T = 2$ along $y = x$ on the square screen, Example \ref{example2a}}
 \label{Wcornerplot} }
\end{figure}

\begin{figure}[H]{
\centering
 \makebox[\linewidth]{
 \includegraphics[width=12cm]{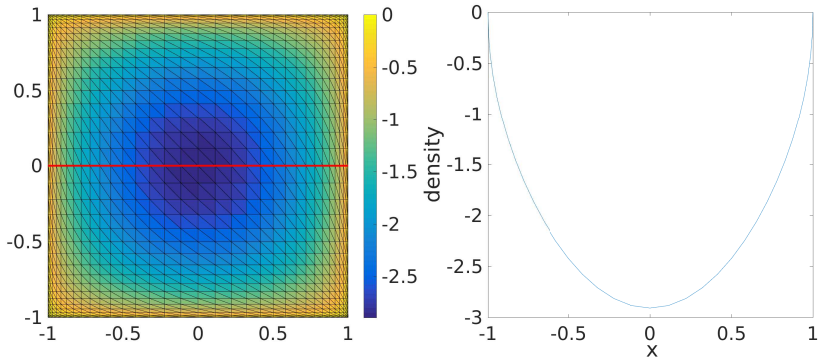}
 }
\caption{Solution of the hypersingular equation at $T = 2$ along $y = 0$ on the square screen, Example \ref{example2a} }
\label{Wedgeplot} }
\end{figure}

%
%

Figure  \ref{Wcornerexp} examines the detailed singular behavior near the corner $(1,1)$. It plots the numerical density at times  up to $T=2.5$ against the distance to the corner along the diagonal of the screen. The numerically computed singular exponents in the corner of around $0.67$ do not show good agreement with the exact corner exponent $\gamma$.  The density as a function of $x$ for $y=0$, perpendicular to the edge, is shown in Figure \ref{Wedgeexp} at the same times. Unlike for the corner exponent, the numerically computed singular exponent at the edge, around $0.48$, is witin $8\%$ of the exact value $\frac{1}{2}$ for early times, and within $4\%$ for $T\geq 1.5$.

\begin{figure}[H] 
\makebox[\linewidth]{
\includegraphics[width=12cm]{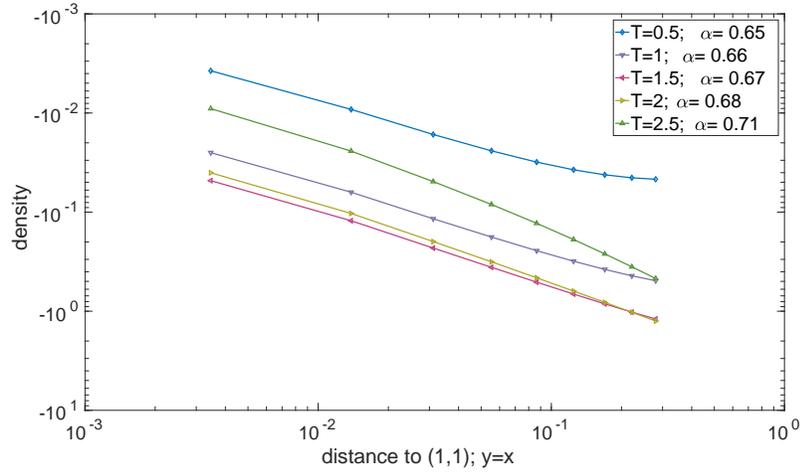} }
\caption{Asymptotic behavior of the solution to the hypersingular equation near corner along $y=x$, Example \ref{example2a0}}
\label{Wcornerexp}\end{figure}
\begin{figure}[H] 
\makebox[\linewidth]{
\includegraphics[width=12cm]{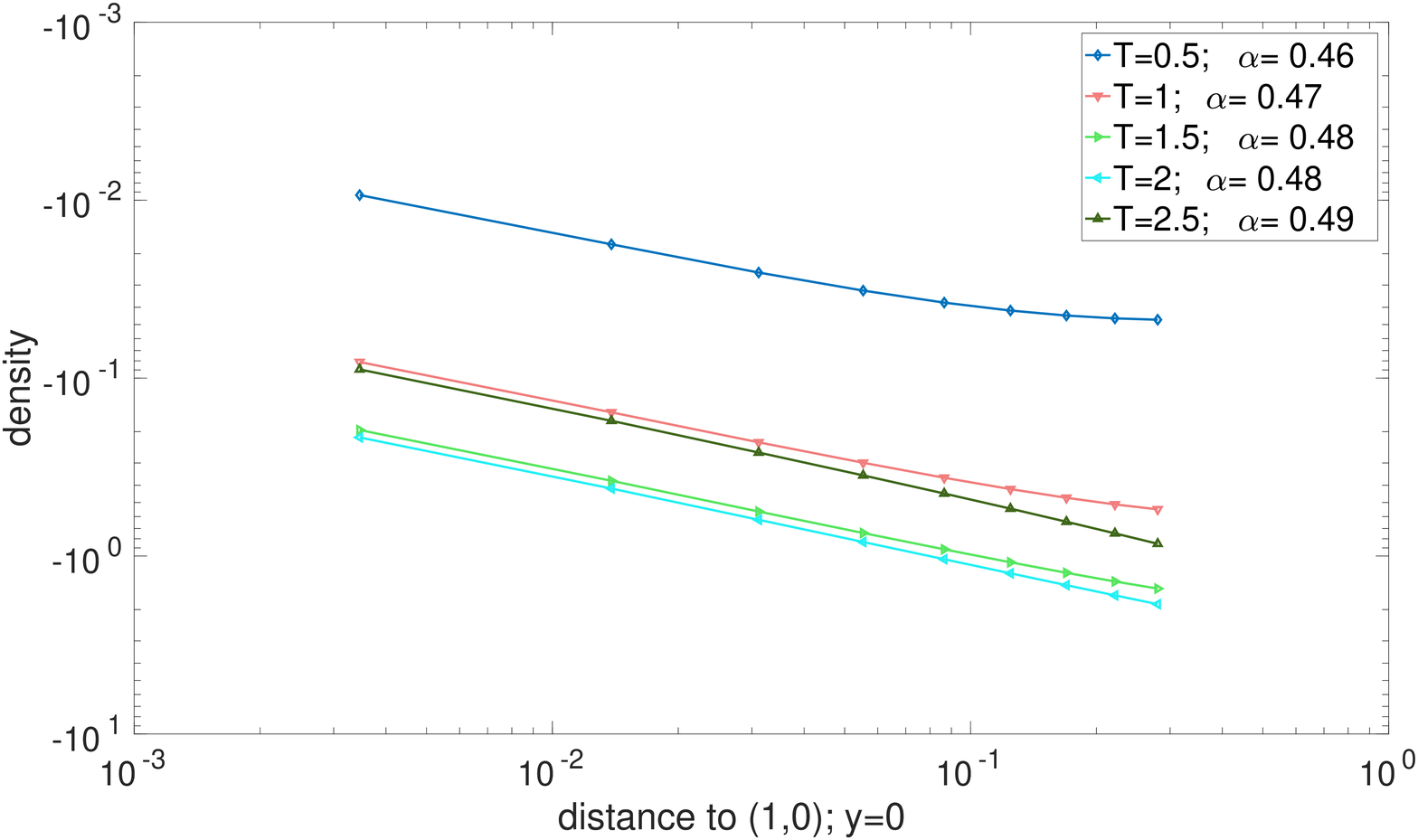} }
\caption{Asymptotic behavior of the solution to the hypersingular equation near edge along $y=0$, Example \ref{example2a0}}
\label{Wedgeexp}
\end{figure}

\begin{figure}[H]
 \makebox[\linewidth]{
 \includegraphics[width=12cm]{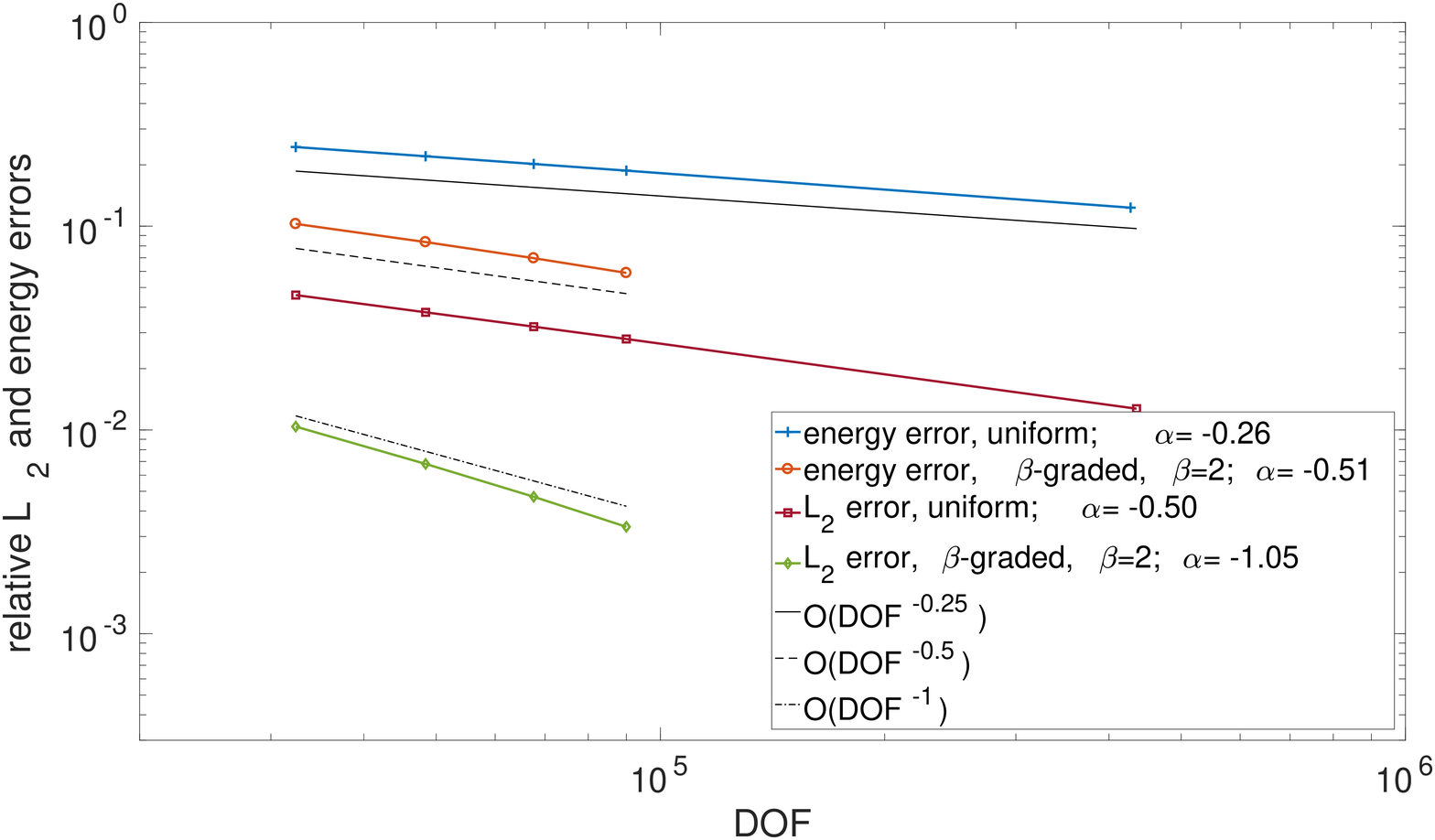}
 }
 \caption{$L_2([0,T],L_2(\Gamma))$ and energy error for hypersingular equation on square screen, Example \ref{example2a}}
 \label{l2W}
\end{figure}
Finally, Figure \ref{l2W} shows the error in both the energy and $L_2([0,T],L_2(\Gamma))$ norms with respect to the benchmark solution. 
The convergence rate in terms of the degrees of freedom on the $2$-graded mesh is $-0.51$ in energy and $-1.05$ in $L_2$. On the uniform mesh the rate is $-0.26$ in energy and $-0.50$ in $L_2$. The rates on the $2$-graded meshes are in close agreement with a convergence proportional to $\sim h$ (equivalently, $\sim DOF^{-1/2}$) predicted by the approximation properties in the energy norm, and $\sim h^{1/2}$ ($\sim DOF^{-1/4}$) on uniform meshes. Also in $L_2$ norm, the convergence corresponds to the expected rates: Approximately $\sim h^2$ (equivalently, $\sim DOF^{-1}$) on $2$-graded meshes,  $\sim h$ (equivalently, $\sim DOF^{-1/2}$) on uniform meshes. In all cases the convergence is twice as fast on the $2$-graded compared to the uniform meshes.

\subsection{Dirichlet-to-Neumann operator}

In addition to the single layer and hypersingular operators in the previous subsections, we also consider the Dirichlet-to-Neumann operator on the screen. Compared to the hypersingular operator, the Dirichlet-to-Neumann operator is {not} available in closed form and requires approximation. It is of interest to see the  influence of the approximation of the operator on the numerical solution.

\begin{example}\label{example2}
Using the discretization from Section \ref{discretization}, we compute the solution to the integral equation $\mathcal{S}u=h$ {on $\mathbb{R}_t^+ \times \Gamma$ with} $\Gamma = [-1,1]^2\times\{0\}$. We prescribe the right hand side
\begin{align*}
  h{(t,x)}&=\textstyle{(-\frac{3}{4}+\cos(\frac{\pi}{2}(4-t))+\frac{\pi}{2}\sin(\frac{\pi}{2}(4-t)) -\frac{1}{4}(\cos({\pi}(4-t))+ \pi\sin({\pi}(4-t))))} \\ &\qquad \times[H(4-t)-H(-t)],
\end{align*}
{where $H$ is the Heaviside function,} and  set $\Delta t=0.01$, $T=0.65$. The finest graded mesh consists of $2312$ triangles, and we use the solution on this mesh as reference solution using the same $\Delta t=0.01$.
\end{example}
Figures \ref{DNcorner} and \ref{DNedge} show the density along a cross-section and along a longitudinal section on a $\beta$-graded mesh with $\beta$=2 and $2312$ triangles at time $T=0.5$. Both figures exhibit the corner and edge singularities predicted by the decomposition \ref{decompositiont} and illustrate the qualitative behavior of the solution. As the solution to the Dirichlet-to-Neumann equation lies in ${H}^{\frac{1}{2}}_\sigma(\R^+, \tilde{H}^{\frac{1}{2}}(\Gamma))$,  its conforming numerical approximation is zero at the boundary of the screen.

Figure  \ref{DNcornerexp} examines the detailed singular behavior near the corner $(1,1)$. It plots the numerical density at times $T=0.25,\ 0.5,\ 0.6,\ 0.65$ against the distance to the corner along the diagonal of the screen. In the log-log plot the slope of the curve near $0$ corresponds to the corner exponent in the singular expansion. Similarly, Figure \ref{DNedgeexp} shows the density as a function of $y$ for $x=-0.8754$, perpendicular to the edge, at the same times. The numerically computed singular exponents of the edge, around $0.4$, are in qualitative agreement with the exact value $\frac{1}{2}$. For the corner, the computed value above $0.6$ differs significantly from the exact value $\gamma$. A similar difference was observed in the previous section for the hypersingular operator, so that the approximation involved in computing the Dirichlet-to-Neumann operator is not the source of this discrepancy. 
  
\begin{figure}[H]{
 \makebox[\linewidth]{
 \includegraphics[width=12cm]{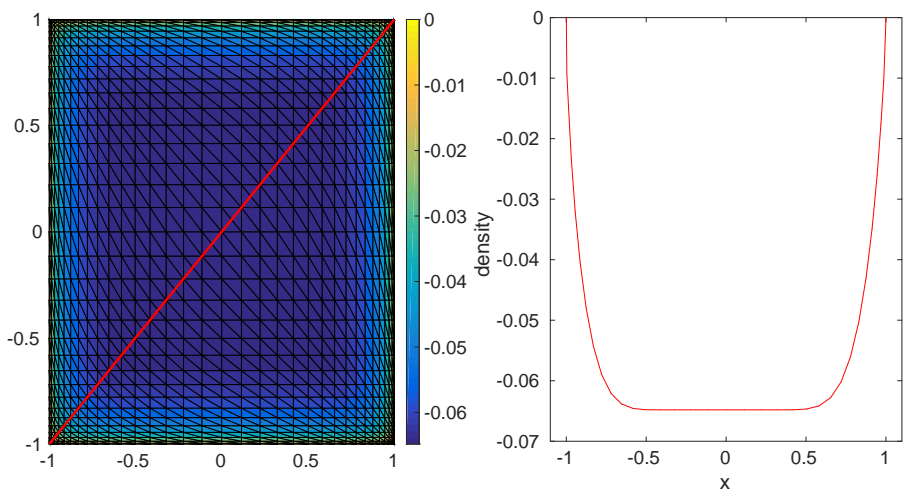}
 }
\caption{Solution of the Dirichlet-to-Neumann equation at $T=0.65$ along $y=x$ on the square screen, Example \ref{example2}}
 \label{DNcorner} }
\end{figure}

\begin{figure}[H]{
\centering
 \makebox[\linewidth]{
 \includegraphics[width=12cm]{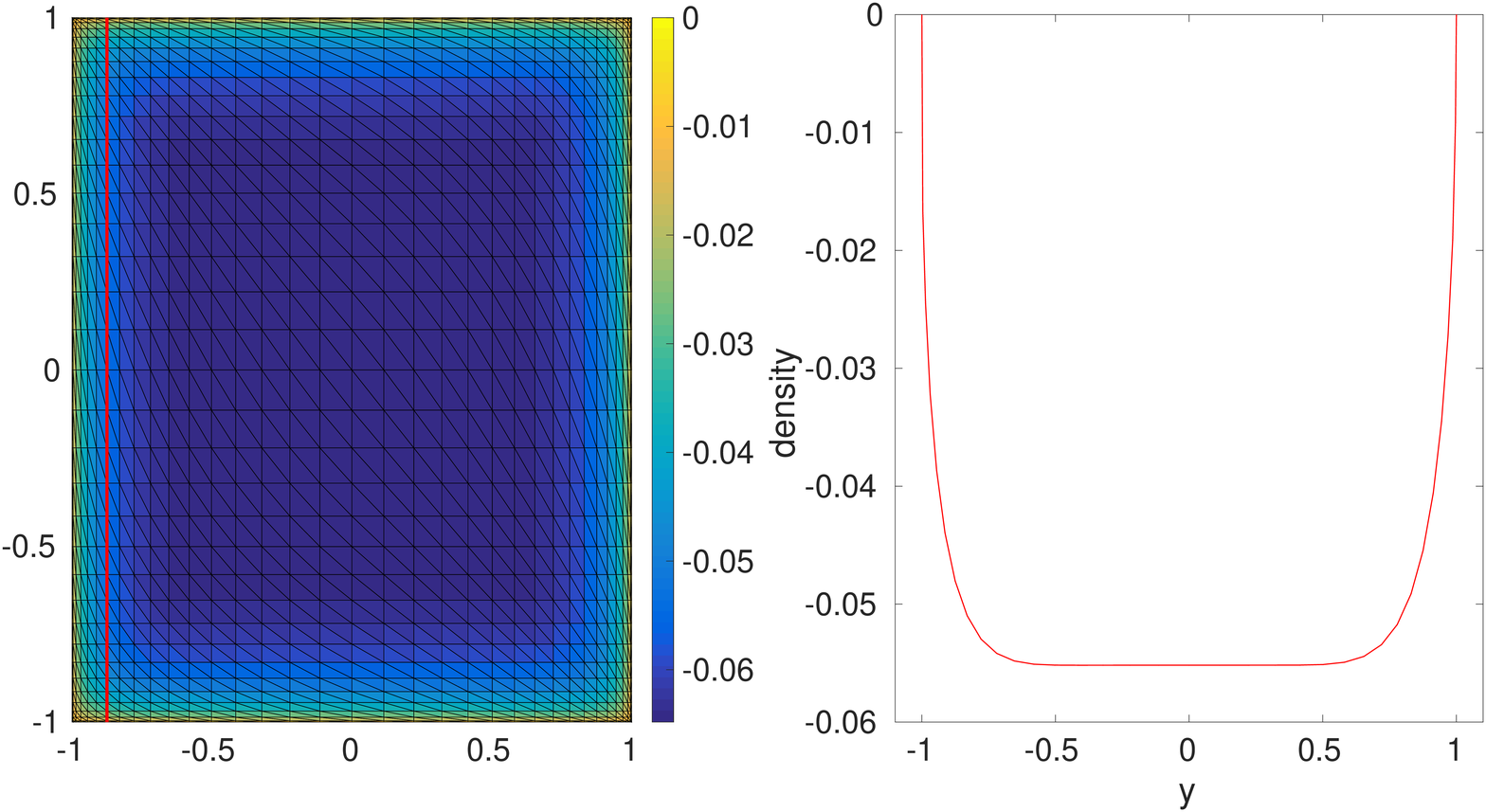}
 }
\caption{Solution of the Dirichlet-to-Neumann equation at $T=0.65$ along $x=-0.8754$ on the square screen, Example \ref{example2}}
\label{DNedge} }
\end{figure}
\begin{figure}[H] 
\makebox[\linewidth]{
\includegraphics[width=12cm]{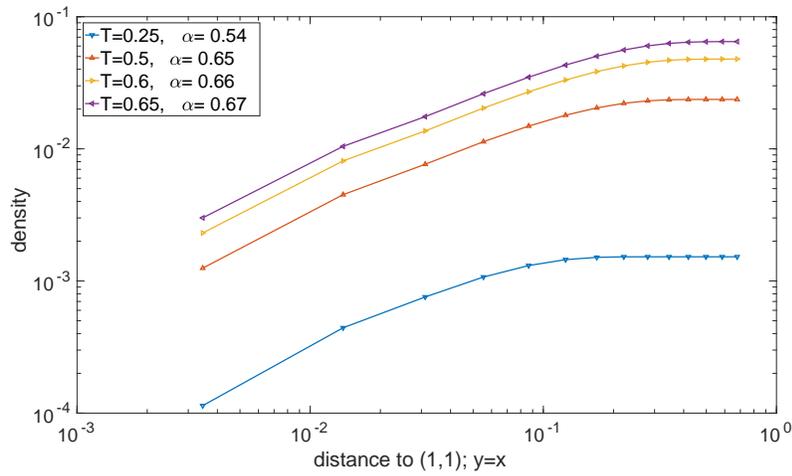} }
\caption{Asymptotic behavior of the solution to the Dirichlet-to-Neumann equation near corner along $y=x$, Example \ref{example2}}
\label{DNcornerexp}\end{figure}
\begin{figure}[H] 
\makebox[\linewidth]{
\includegraphics[width=12cm]{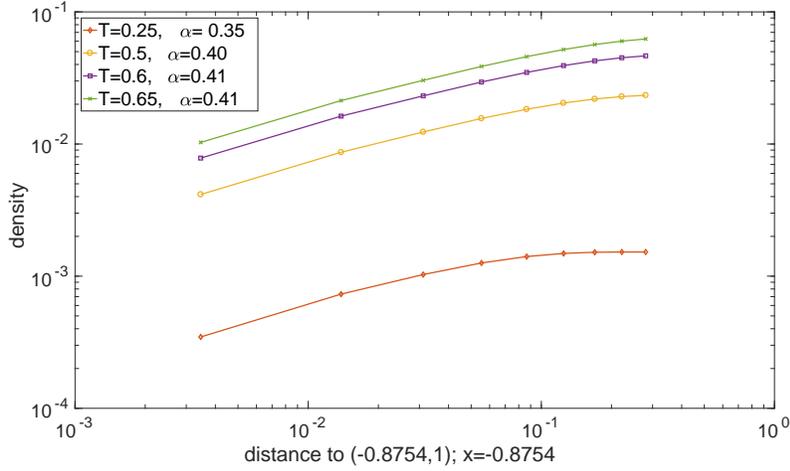} }
\caption{Asymptotic behavior of the solution to the Dirichlet-to-Neumann equation near edge along $x=-0.8754$, Example \ref{example2}}
\label{DNedgeexp}
\end{figure}
\begin{figure}[H]
 \makebox[\linewidth]{
 \includegraphics[width=12cm]{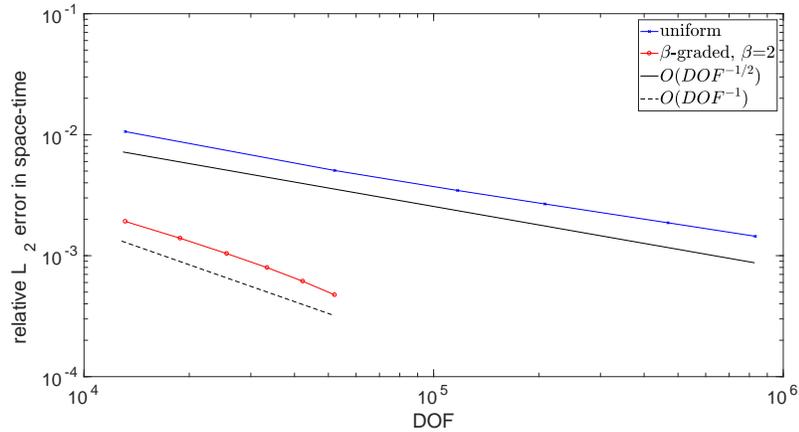}
 }
 \caption{Error in $L_2([0,T],L_2(\Gamma))$ norm for Dirichlet-to-Neumann equation on square screen, Example \ref{example2}}
 \label{l2steklov}
\end{figure}
Figure \ref{l2steklov} shows the error in $L^2([0,T] \times \Gamma)$ compared to the benchmark solution. 
The convergence in this norm is proportional to $\sim h^2$ (equivalently, $\sim DOF^{-1}$) on the 2-graded mesh, while the convergence is $\sim h^{1}$ ($\sim DOF^{-\frac{1}{2}}$) on a uniform mesh. This coincides with the rates expected from the approximation property of the graded, respectively uniform meshes, and it is also in agreement with the rates obtained for the hypersingular operator on the square screen in the previous section.

\section{Applications to traffic noise: Horn effect}

For applications in traffic noise, the natural (simplified) geometry is that of a half-space $\mathbb{R}^3_+$ with a tire, as displayed in Figure~\ref{fig:horn:geometry}. The horn like geometry between the tire and the street amplifies sound sources close to the contact patch, and it is of interest to compute the amplification for a broad band of frequencies. See also \cite{banz2, kropp200sound}. See \cite{w1, w2} for the complementary problem of the tire dynamics in contact with the road.  

We consider the wave equation for the sound pressure scattered by the tire, with homogeneous Neumann conditions on the street $\Gamma_\infty = \mathbb{R}^2 \times \{0\}$ and inhomogeneous Neumann conditions on the tire. Note that the boundary conditions jump in the cuspidal geometry between the tire and the road surface.  The relevant Green's function in $\mathbb{R}^3_+$ is given by
\begin{align}\label{greentraffic}
  G(t,x,y)=\frac{\delta(t-|x-y|)}{4\pi|x-y|}+ \frac{\delta(t-|x-y'|)}{4\pi|x-y'|}\ ,
\end{align}
where $y'$ is the reflection of $y$ on $\Gamma_\infty$. We use it in a single layer potential ansatz for a sound pressure scattered by the tire,
\begin{align}\label{phalf}
p(t,x)=\frac{1}{4\pi}\int_\Gamma \frac{\phi(t-|x-y|,y)}{|x-y|}\,ds_y +\frac{1}{4\pi}\int_\Gamma \frac{\phi(t-|x-y'|,y)}{|x-y'|}\,ds_y\ ,
\end{align}
with $\phi(s,y) = 0$ for $s\leq 0$. The Neumann problem for the scattered sound translates into an integral equation for $\phi$: 
\begin{align}\label{eq:neumann-direct}
 \left(-I+K^\prime\right)\phi(t,x) = 2\frac{\partial p}{\partial n}(t,x)= -2\frac{\partial p^I}{\partial n}(t,x)\ ,
\end{align}
with $p^I$ the incoming wave and the adjoint double layer  operator $K'$ from \eqref{operators},
\begin{align*}
 K'\phi(t,x)=& \frac{1}{2\pi}\int_\Gamma\frac{n_x^\top(y-x)}{|x-y|}\left(\frac{\phi(t-|x-y|,y)}{|x-y|^2}+\frac{\dot \phi(t-|x-y|,y)}{|x-y|}\right)\,ds_y\\
               &+\frac{1}{2\pi}\int_\Gamma\frac{n_x^\top(y'-x)}{|x-y'|}\left(\frac{\phi(t-|x-y'|,y)}{|x-y'|^2}+\frac{\dot \phi(t-|x-y'|,y)}{|x-y'|}\right)\,ds_y .
\end{align*} 
The weak formulation reads:\\
Find $\phi \in H^\frac{1}{2}_\sigma(\mathbb{R}^+, \widetilde{H}^{-\frac{1}{2}}(\Gamma))$ such that for all test functions $\psi \in H^\frac{1}{2}_\sigma(\mathbb{R}^+, H^{-\frac{1}{2}}(\Gamma))$
\begin{align}\label{eq:var1}
\int_{0}^\infty \int_{\Gamma} \left(-I+K'\right)\phi\ \psi\ \,ds_x\,d_\sigma t =-2\int_{0}^\infty \int_{\Gamma} \frac{\partial p^I}{\partial n}\ \psi\,ds_x\,d_\sigma t\ .
\end{align}
It is discretized with piecewise constant ansatz and test functions $\psi_i^h(x)  \gamma^n_{\Delta t}(t) \in V^{0,0}_{t,h}$  in space and time.\\

To obtain the sound amplification for the entire frequency spectrum in one time domain computation, we consider the sound emitted by a Dirac point source. It is located in the point  $y_{src}=(0.08,0,0)$ near the horn,
\begin{align}
  p^I=\frac{\delta(t-|x-y_{src}|)}{4\pi|x-y_{src}|}+ \frac{\delta(t-|x-y'_{src}|)}{4\pi|x-y'_{src}|}\ .
\end{align}
 The right hand side of the discretization of the  integral equation \eqref{eq:var1} is calculated to be \cite{banz2}
\begin{align*}
& -2\int_0^\infty \int_\Gamma \frac{\partial p^I}{\partial n} \psi_i^h  \gamma^n_{\Delta t} \,ds_x \,dt = - \int_{T_i \cap E(y_{src})} \frac{n_{x}^\top(y_{src}-x)}{\pi\arrowvert x-y_{src}\arrowvert^{3}} \,ds_x +  n_{x}^\top(y_{src}-x) \left\{ \frac{\zeta(t_{n-1})}{\pi t_{n-1}^{2}}-\frac{\zeta(t_{n})}{\pi t_{n}^{2}}\right\}\ .
 \end{align*}
 The first term is an integral over the domain of influence $E(y_{src}) {= \{x \in \Gamma : t_{n-1}\leq |x-y_{src}| \leq t_n\}}$ of $y_{src}$, {intersected with $T_i = \mathrm{supp}\ \psi_i^h$}, and it is computed in the same way as the entries of the Galerkin matrix. In the second term, $\zeta(t)$ denotes the length of the curve segment $T_{i} \cap \{\arrowvert x-y_{src}\arrowvert=t\}$ inside the triangle $T_i$. 

 \begin{figure}[tbp]
   \centering
   \includegraphics[trim = 0mm 0mm 0mm 0mm, clip,width=55.0mm, keepaspectratio]{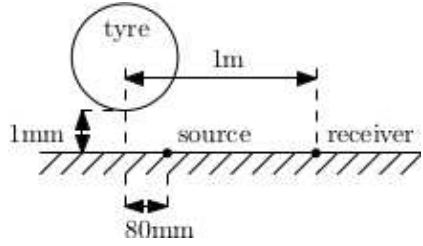}
   \caption{Cross section of geometrical setup for horn effect.}
   \label{fig:horn:geometry}
 \end{figure}

After solving the discretization of \eqref{eq:var1} for the density $\phi$, we obtain the sound pressure $p$  in the receiver point $x_{fp}=(1,0,0)$ from \eqref{phalf}. From \cite[Eq.~7]{kropp200sound}, the amplification factor is given by:
\begin{equation*}
 \Delta L_H(\omega)=20 \log_{10}\left(\frac{|\hat{p}(\omega,x_{fp})+\hat{p}^I(\omega,x_{fp})|}{|\hat{p}^I(\omega,x_{fp})|}\right)\ .
\end{equation*}
Here, $\hat{p}$ and $\hat{p}^I$ denote the Fourier transformed incident and scattered sound pressure fields. The Fourier transformation is calculated using a discrete FFT, where the time step size is the same as for the computation of the density.


In  the geometry given by Figure \ref{fig:horn:geometry}, we compute the sound amplification in standard units for a grown slick 205/55R16 tire at 2 bar pressure. It is subject to 3415N axle load at 50 km/h on a street with an ISO 10844 surface, and a mesh with 6027 nodes is depicted in Figure \ref{tyrepic}. We use this and a refined graded mesh and consider the sound amplification for frequencies between $200$ and $2000$ Hz. The total time interval is $T=24$ and the time step sizes $\Delta t =0.005,\ 0.01, \ 0.04$. For smaller time step sizes more reflections in the horn can be resolved, and these are responsible for the sound amplification.\\

We compare the results for the uniform mesh with a refined, graded-like mesh with grading parameter $\beta =2$, see Figure \ref{tyrepic}. Figure \ref{gradedhorn1} shows approximations of the amplification factor in the horn geometry, discretized using the graded mesh, across the frequency range for the time step sizes $\Delta t =0.005,\ 0.01, \ 0.04$.  We also show the approximation given by the uniform tire mesh for $\Delta t=0.005$. The figure, in particular, exhibits several resonances between $1000$ and $2000$ Hz, at which the different approximations lead to significant differences in the computed amplification factors.\\

 The differences between the computed amplification factors are depicted in Figure \ref{gradedhorn2}. The first subfigure considers the differences between the graded and uniform meshes for a given time step size, $\Delta t =0.005,\ 0.01, \ 0.04$. Outside the resonance frequencies the differences are negligible. Especially in the strong resonances around $1300$ and $1900$ Hz, however, the difference between graded and uniform meshes becomes more and more relevant for smaller $\Delta t$, as the small time step allows to resolve the reflections in the horn geometry more accurately. The second subfigure of Figure \ref{gradedhorn2} compares the computed amplification for graded meshes for different $\Delta t$. As before, the differences are mostly relevant near resonance frequencies, and the discretization error for a fixed mesh decreases with $\Delta t$. For $\Delta t=0.005$ the differences between the spatial, resp.~temporal discretizations in Figure \ref{gradedhorn2} are both around $6$ dB near $1300$ Hz. Such differences in sound pressure are significant to the human perception. They indicate the relevance of graded meshes for computations of traffic noise.\\

\begin{figure}[H]
   \centering

   \subfigure[]{
   \includegraphics[height=6cm]{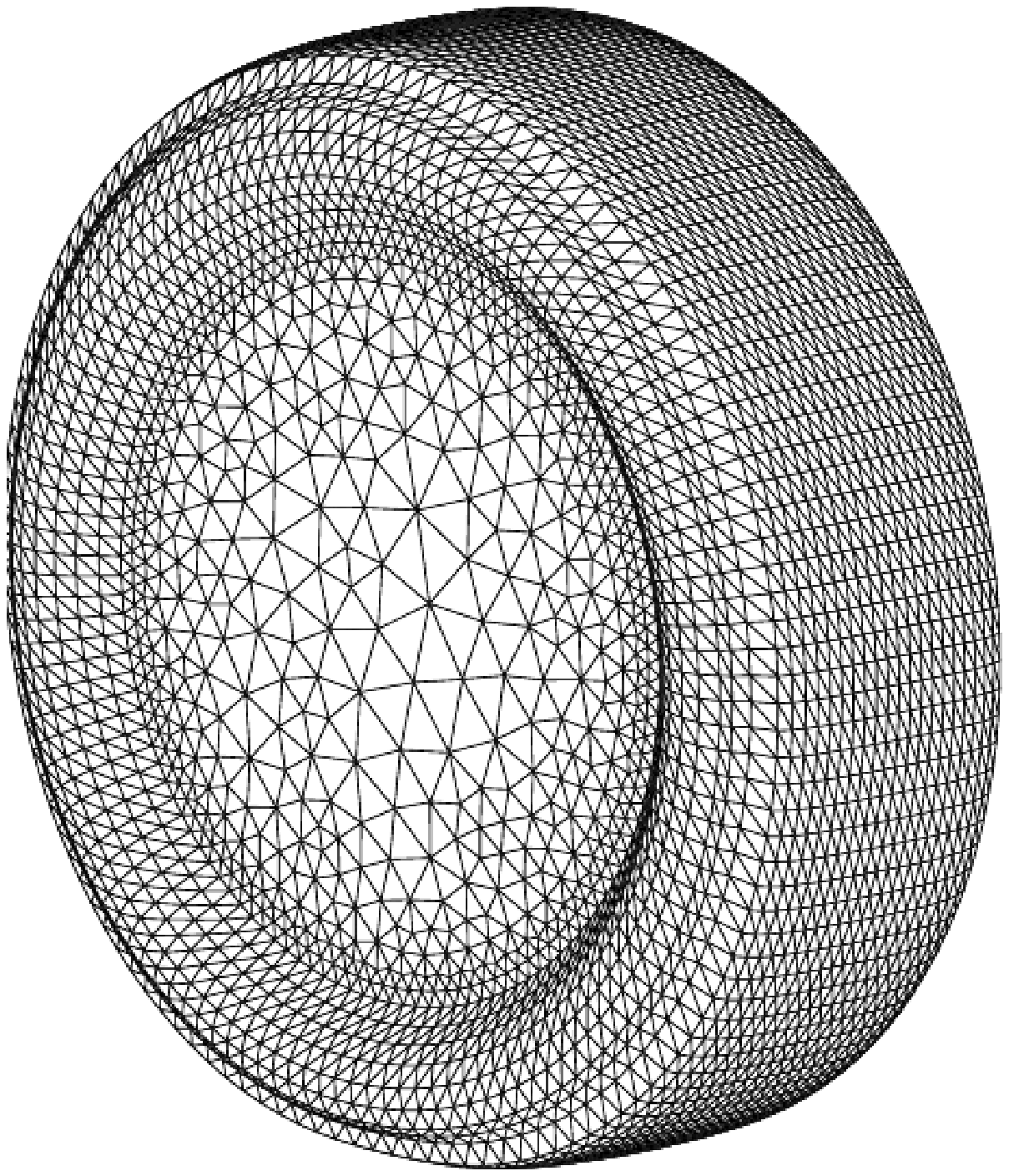}}
 \subfigure[]{
   \includegraphics[width=8cm]{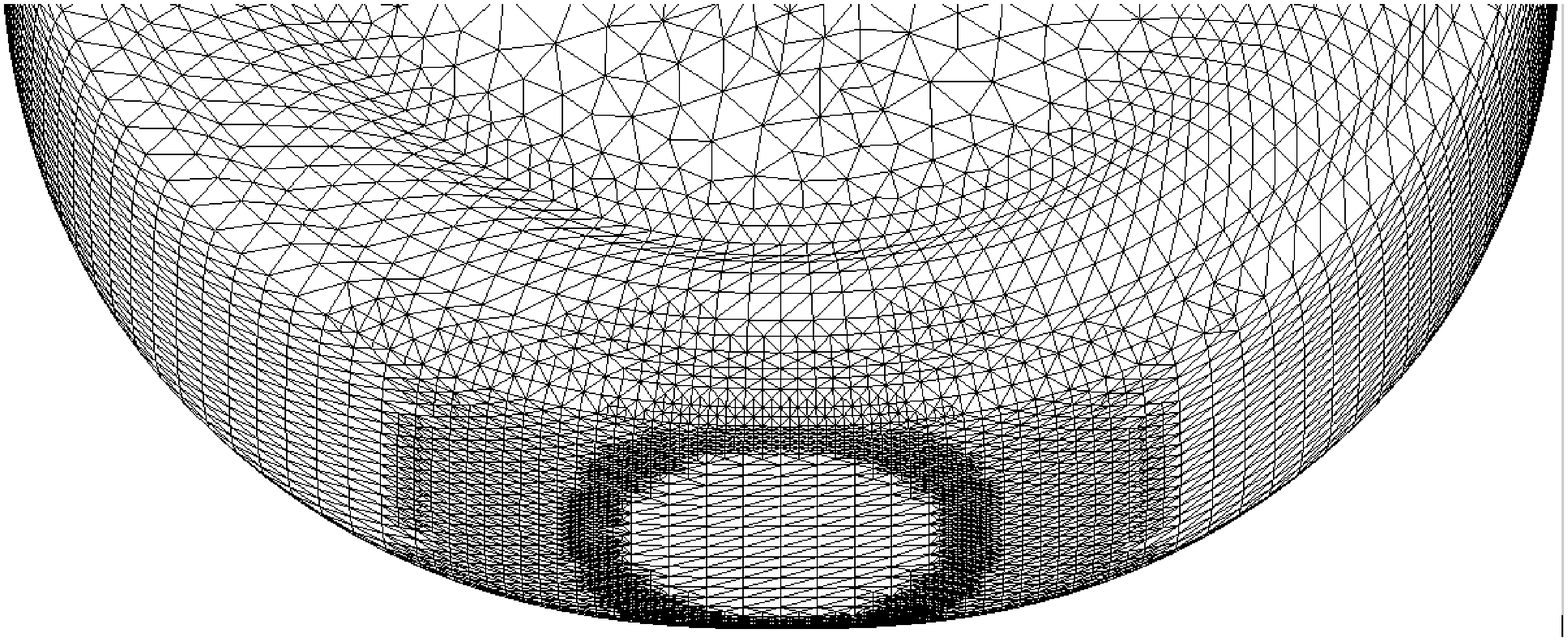}}
   \caption{Mesh of (a) slick 205/55R16 tire and (b) graded refinement.}
\label{tyrepic}
\end{figure}

\begin{figure}[H]
   \centering
   \includegraphics[height=6cm]{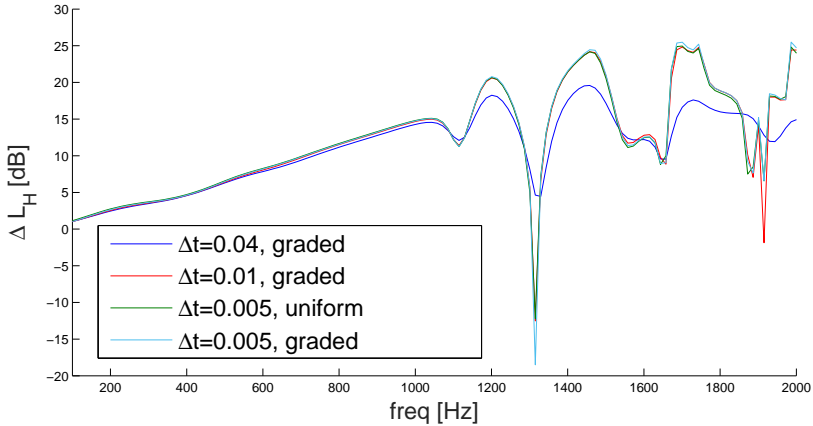}
   \caption{Amplification due to horn effect: Graded mesh approximations for different $\Delta t$, compared to a uniform mesh approximation.}
\label{gradedhorn1}
\end{figure}

\begin{figure}[H]
   \centering
   \includegraphics[height=7cm]{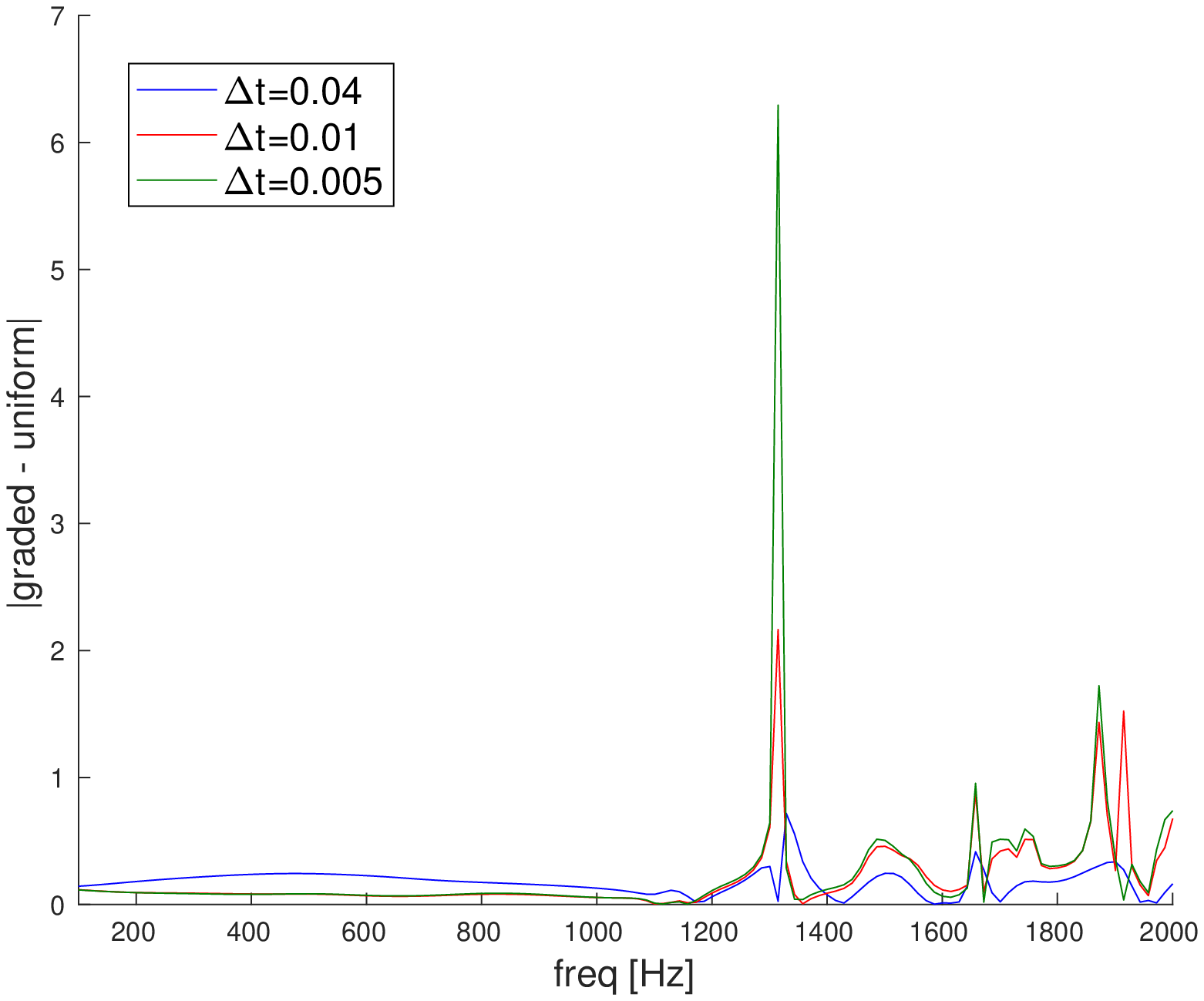}
   \includegraphics[height=7cm]{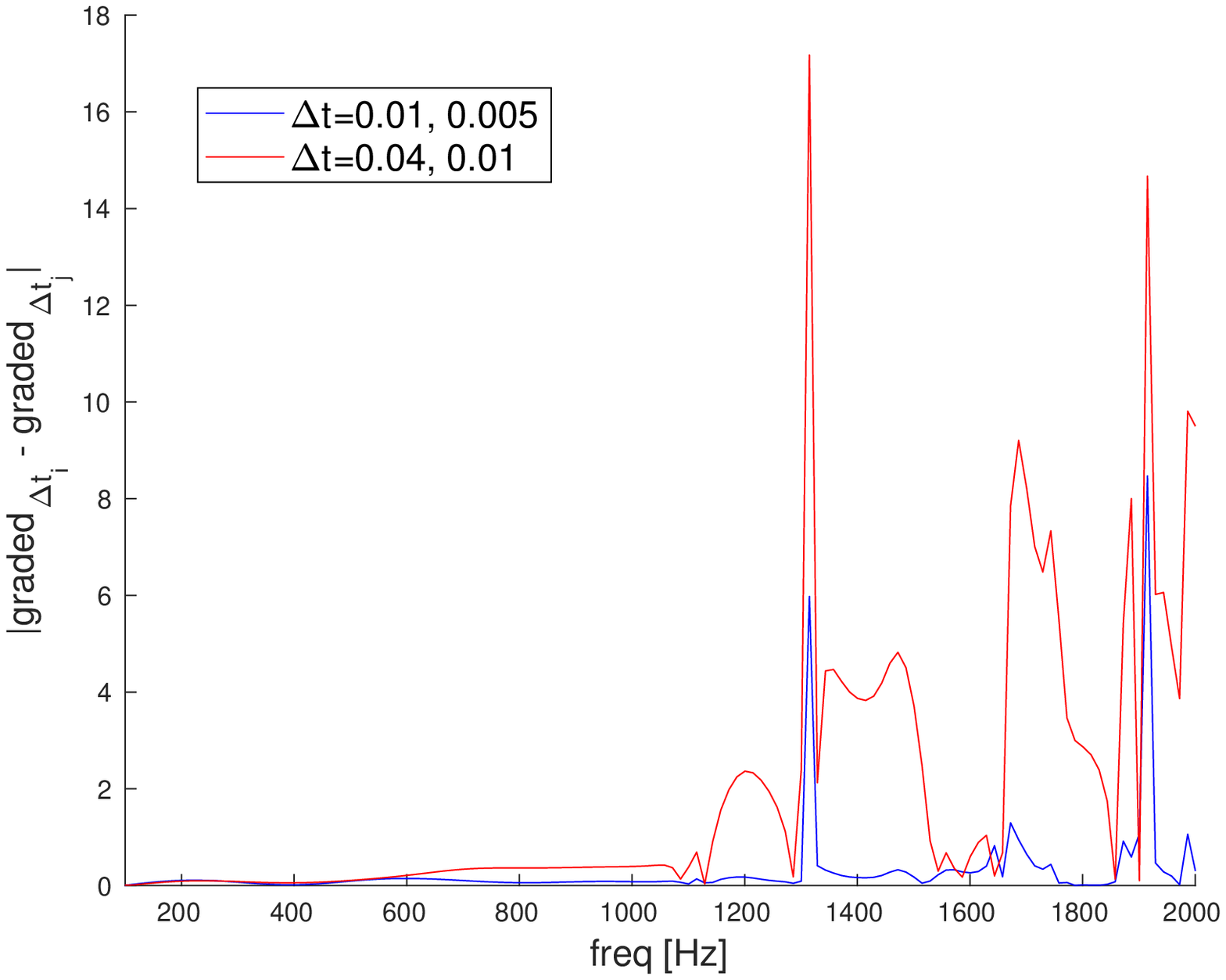}
   \caption{Differences of amplification factors in $dB$ between graded and uniform meshes for fixed $\Delta t$, resp.~between graded meshes for different $\Delta t$.}
\label{gradedhorn2}
\end{figure}

{
\section*{Acknowledgement}
We thank one of the referees for particularly helpful suggestions which improved the article.
}

  \bibliographystyle{abbrv}

\end{document}